\documentclass[reqno]{amsart}
\usepackage{enumerate}
\usepackage{amsmath,amssymb,amsthm,mathtools,mathrsfs,enumitem,array,booktabs,nccmath,url,color,dsfont}
\usepackage[left=3.2cm,right=3.2cm,top=4cm,bottom=4cm]{geometry}
\usepackage[colorlinks=true]{hyperref}
\hypersetup{
  linkcolor=blue,
  citecolor=blue,
  filecolor=blue,
  urlcolor=cyan,
  pdftitle={Global well-posedness for two-dimensional gPAM on the whole space},
  pdfauthor={Hao Shen, Rongchan Zhu, and Xiangchan Zhu}
}
\numberwithin{equation}{section}
\allowdisplaybreaks
\newtheorem{theorem}{Theorem}[section]
\newtheorem{lemma}[theorem]{Lemma}
\newtheorem{proposition}[theorem]{Proposition}
\newtheorem{corollary}[theorem]{Corollary}
\theoremstyle{definition}
\newtheorem{definition}[theorem]{Definition}

\newtheorem{remark}[theorem]{Remark}

\newcommand{\la}{\langle}
\newcommand{\ra}{\rangle}
\newcommand{\norm}[1]{\left\lVert #1\right\rVert}

\newcommand{\R}{\mathbb R}

\newcommand{\cL}{\mathcal L}
\newcommand{\cC}{\mathcal C}
\newcommand{\cE}{\mathcal E}
\newcommand{\cR}{\mathcal R}

\newcommand{\sI}{\mathcal I}
\newcommand{\para}{\prec}
\newcommand{\Par}{\succ}
\newcommand{\res}{\circ}

\newcommand{\dlt}{\mathfrak D}

\newcolumntype{L}[1]{>{\raggedright\arraybackslash}p{#1}}

\title[Whole-space gPAM]{Global well-posedness for generalized parabolic Anderson model on the whole plane}

\author{Hao Shen}
\address[H. Shen]{Department of Mathematics, University of Wisconsin - Madison, USA}
\email{pkushenhao@gmail.com}

\author{Rongchan Zhu}
\address[R. Zhu]{Department of Mathematics, Beijing Institute of Technology, Beijing 100081, China}
\email{zhurongchan@126.com}

\author{Xiangchan Zhu}
\address[X. Zhu]{State Key Laboratory of Mathematical Sciences, Academy of Mathematics and Systems Science, Chinese Academy of Sciences, Beijing 100190, China}
\email{zhuxiangchan@126.com}


\date{}

\begin{document}
\begin{abstract}
For every \(0<\kappa<\sqrt{5}-2\), we prove global existence for the two-dimensional generalized parabolic Anderson model on the whole plane $\mathbb R^2$ with nonlinearity $F\in C_b^2(\mathbb R)$, driven by an enhanced noise $(\eta,\Psi)$. The noise $\eta$ has polynomially weighted spatial Besov--H\"older regularity $-1-\kappa$, and $\Psi$ is the corresponding renormalized second-order object. If $F''$ is globally Lipschitz, the solution is unique.

The proof combines a weight-compatible annular high--low decomposition with a paracontrolled transport representation. The final remainder is estimated simultaneously in a weighted $L^\infty$ norm
and in a higher-order weighted parabolic H\"older norm, using two strictly different polynomial weights.
This weight gap absorbs the polynomial losses generated by the enhanced noise, the localization procedure, and the transport coefficient.
Several refinements of earlier work allow the maximum-principle and Schauder estimates to yield a global a priori bound for a larger range of $\kappa$. Uniqueness is proved in a time-dependent exponentially weighted topology.
\end{abstract}
\maketitle

\setcounter{tocdepth}{2}
\tableofcontents

\section{Introduction}

We consider the two-dimensional generalized parabolic Anderson model on the whole plane,
\[
  \mathcal L u = F(u)\eta,
  \qquad \mathcal L=\partial_t-\Delta,
  \qquad x\in\mathbb R^2.
\]
The spatial distribution \(\eta\) is assumed to have regularity \(C^{-1-\kappa}\).  The heat equation then suggests \(u\in C^{1-\kappa}\), so the product \(F(u)\eta\) is not defined by classical multiplication.  As in the theories of regularity structures and paracontrolled distributions, we formulate the equation with respect to an enhanced noise, written as
\[
  \boldsymbol\eta=(\eta,\Psi),
  \qquad \Psi = 
  \sI(\eta)\res\eta,
\]
where the latter expression is understood after renormalization \cite{Hai14,GIP15}.
{Once this enhancement has been prescribed, the equation is treated as a deterministic singular PDE.}

In the linear case,
{the parabolic Anderson model admits global constructions on the whole space} 
\cite{HL2d,HL}. (See also \cite{CFG17} for existence of density using Malliavin calculus and regularity structures.)
For genuinely nonlinear gPAM on compact domains, global-in-time a priori estimates were recently obtained by Chandra, de Lima Feltes and Weber in the regularity-structure framework \cite{CFW24} and, by paracontrolled methods, in \cite{SZZ24}.  Both proofs use a transport representation.

More generally, the finite-volume theory of global solutions for singular SPDEs has developed substantially.  For instance, global a priori estimates and coming down from infinity were obtained for the dynamic $\Phi^4_3$ model on the torus by Mourrat and Weber \cite{MW17b}.  In the direction of singular stochastic fluid equations on compact domains, we refer to the works \cite{HZZ23a,HZZ23b,HR24,HZ25}.  
Moreover, Bringmann and Cao proved global well-posedness for the dynamical sine--Gordon model up to $6\pi$ and for the two-dimensional stochastic Abelian--Higgs equations on the torus \cite{BC26,BC24}.  Thus, on compact spatial domains there is now a substantial body of global theory.
By contrast, in unbounded spatial domains, the
global theory for nonlinear singular SPDEs is much more limited.  Besides the linear whole-space parabolic Anderson model \cite{HL2d,HL}, the whole-space constructions closest to the present work include the two-dimensional dynamic $\Phi^4$ model in the plane of Mourrat and Weber \cite{MW17}, the Euclidean $\Phi^4$ results of Gubinelli and Hofmanov\'{a} \cite{GH18,GH18a}, and the singular HJB/KPZ theory of Zhang, Zhu and Zhu \cite{ZZZ22}.

{The purpose of this paper is to extend the finite-volume global theory for gPAM to the whole plane \(\mathbb R^2\) in polynomially weighted spaces.
We also refine several estimates from \cite{SZZ24}, which improves the admissible range of $\kappa$.}

The additional difficulty on \(\mathbb R^2\) is not only the possible behavior at spatial infinity, but also the fact that the constants in local estimates must remain uniform as the spatial region expands.  On a compact domain one can work in unweighted spaces and the low frequencies have no spatial tail.  On the whole space, by contrast, both the enhanced noise and the solution may have polynomial growth, and products between solution-dependent coefficients and rough objects continuously transfer polynomial weights from one factor to another.  Thus every paraproduct, resonant product, Schauder estimate, and maximum-principle estimate has to keep track of regularity and spatial weight at the same time.  We use
\[
  \rho_D(x)=\langle x\rangle^{-D},
\]
with positive \(D\) corresponding to allowed polynomial growth.

\subsection{Main idea and comparison with previous work}\label{subsec:intro-main-idea}

The finite-volume results for nonlinear gPAM {already contain the basic mechanism, namely a transport representation for the nonlinear term;} see \cite{CFW24,SZZ24}.  {In the present paper the main difficulty is to implement this mechanism on \(\mathbb R^2\), where local estimates must be uniform over expanding spatial regions and where every product carries both regularity and polynomial weight.  The main novelty is therefore a weight-compatible version of the finite-volume transport argument, together with two refinements of the decomposition and of the regularity parameters that enlarge the admissible range of \(\kappa\).}

 First, we prove a weighted annular localization estimate  in Section~\ref{sec:parameters-weights}.  This produces a decomposition, in which the high-frequency component has the prescribed spatial decay needed for weighted products, while the low-frequency component gains regularity with an explicitly controlled polynomial loss. This is in the  spirit of  \cite[Lemma~2.4]{GH18}, but compared with that work, our present formulation separates the regularity increment from the polynomial weight change. This separation is crucial for the transport argument, because the same rough noise component is estimated at several different regularity levels. If every gain in regularity were tied to a comparable loss in polynomial weight, the resulting weight loss would be too large to be absorbed by the maximum-principle estimate; see Remark~\ref{rem:gh-localization-comparison}.

Second, the final remainder is controlled in two different weighted norms.  The maximum-principle estimate is applied with a weaker polynomial weight, whereas the Schauder estimate is applied in a higher regularity norm with a slightly stronger weight.  The gap between the two weights absorbs the finite polynomial losses generated by the enhanced noise, the localization procedure, and the transport coefficient.

Third, we use a refined placement of the second-order enhanced term in the paracontrolled decomposition.  Only the paraproduct part is included in the first paracontrolled equation; the remaining second-order contributions are kept in the forcing term.  Consequently these terms are not differentiated in the transport representation.  This avoids an unnecessary high-norm factor in the estimate of the first paracontrolled component; see Remark~\ref{rem:comparison-szz24-split}.

Another ingredient in improving the admissible range of $\kappa$ is to treat the high-frequency component of the noise $\eta$ in several different regularity regimes, depending on its role in the argument: in commutator estimates, rough resonant products, fixed-cutoff and time-increment estimates, and products involving the high-regularity norm of the remainder.  At the same time, different components of the solution are estimated in different regularity spaces.  After choosing the cutoff levels, the analytic estimates reduce to finitely many scalar inequalities.  In the limiting unweighted calculation, the decisive condition is
\[
  \kappa^2+4\kappa-1<0,
\]
which is equivalent to \(0<\kappa<\sqrt5-2\).  The polynomial weights introduce only perturbative losses and admissibility constraints, which are absorbed by taking the total weight scale sufficiently small; see Remark~\ref{rem:szz24-parameter-comparison}.

The uniqueness argument is also genuinely weighted.  We compare two solutions in a time-dependent exponentially weighted topology following  \cite{HL2d, HL}, while keeping the polynomial weights sufficiently small to absorb the remaining polynomial losses.  An additional high--low decomposition is used in the difference estimates to reserve extra spatial decay for terms without an exponential time gap.

The paper is organized as follows.  Section~\ref{sec:target} states the main result.  Section~\ref{sec:parameters-weights}  constructs the annular high--low decomposition of the enhanced noise.  Section~\ref{sec:model-pde-pc} introduces the cutoff paracontrolled equation and proves an a priori estimate for the solution to the first equation.  Section~\ref{sec:transport-representation} proves the weighted transport representation and weighted maximum principle.  Section~\ref{sec:closure} estimates the forcing terms and proves the a priori bound.  Section~\ref{sec:difference-details} proves uniqueness by a time-dependent exponential-weight argument.

\section*{Acknowledgements}
The authors gratefully acknowledge the assistance of GPT-5.5 Pro in adjusting the parameter and in optimizing the admissible range of $\kappa$, especially the point mentioned in Remark~\ref{rem:szz24-parameter-comparison}.

H.S. gratefully acknowledges financial support from NSF grants DMS-1954091 and CAREER DMS-2044415. X.Z. is grateful for financial support in part from the NSFC (No. 12595281). R.Z. and X.Z. are grateful for financial support from the National Key R\&D Program of China (No. 2022YFA1006300) and from the NSFC (No. 12426205). R.Z. is grateful for financial support from the NSFC (No. 12271030). X.Z. is grateful for financial support in part from the NSFC (No. 12288201), for support from the Key Lab of Random Complex Structures and Data Science, Chinese Academy of Sciences.

\section{Main result}\label{sec:target}

We first introduce the norms used in the paper.
Throughout this section, unless a more restrictive range is explicitly stated, the spatial regularity index satisfies $\beta\in\mathbb R$ and the time-weight exponent satisfies $0\leqslant\gamma<1$.  We write $A\lesssim B$ if $A\leqslant CB$, where the implicit constant may depend on the fixed parameters of the paper but not on the cutoff levels or on the unknown scalar quantities in the a priori estimates.
Let $(\Delta_j)_{j\geqslant-1}$ be the Littlewood-Paley blocks and let
$\rho_D(x)=\la x\ra^{-D}.$
Thus positive values of \(D\) correspond to allowed polynomial growth, whereas weights of the form \(\rho_{-\sigma}\) encode additional decay. We use \(C^\beta(\rho_D)\) and \(\cC_D^\beta\) interchangeably when no ambiguity can arise.
\begin{itemize}
\item
We use the weighted Besov--H\"older norm
\begin{equation}\label{eq:weighted-holder}
  \norm{f}_{C^\beta(\rho_D)}
  :=\sup_{j\geqslant -1}2^{\beta j}\norm{\rho_D\Delta_j f}_{L^\infty}.
\end{equation}
For simplicity of notation we sometimes write
$\norm{f}_{\cC_D^\beta}:=\norm{f}_{C^\beta(\rho_D)}$.
\item
For time-dependent functions define, for any Banach space $E$,
\begin{equation*}
  \norm{f}_{L_T^{\infty,\gamma}E}:=\sup_{0<t\leqslant T}t^\gamma\norm{f(t)}_E.
\end{equation*}
When $E$ corresponds to spatial Besov--H\"older norms we write
\begin{equation*}
  \norm{v}_{C_T^{\beta,\gamma}(\rho_D)}
  :=\sup_{0<t\leqslant T}t^\gamma\norm{v(t)}_{C^\beta(\rho_D)}.
\end{equation*}
\item
We also use the time-H\"older seminorm 
\begin{equation*}
  \norm{v}_{C_T^{\nu,\gamma}L^\infty(\rho_D)}
  :=\sup_{0<s<t\leqslant T}s^\gamma
  \frac{\norm{v(t)-v(s)}_{L^\infty(\rho_D)}}{|t-s|^\nu},
  \qquad 0<\nu<1.
\end{equation*}
\item
The parabolic solution norm is defined by
\begin{equation*}
  \norm{v}_{\mathbb S_T^{\beta,\gamma;\nu}(\rho_D)}
  :=\norm{v}_{C_T^{\beta,\gamma}(\rho_D)}
    +\norm{v}_{C_T^{\nu,\gamma}L^\infty(\rho_D)}.
\end{equation*}
When $\nu$ is omitted  and $0<\beta<2$, it means that we use the parabolic convention
\[
  \mathbb S_T^{\beta,\gamma}(\rho_D)
  :=\mathbb S_T^{\beta,\gamma;\beta/2}(\rho_D).
\]
In general we keep the third index because later we also use these norms with \(\nu\ne\beta/2\).
\end{itemize}

We consider
\begin{equation}\label{eq:gpam}
  \mathcal L u=F(u)\eta,
  \qquad
  \mathcal L=\partial_t-\Delta,
  \qquad
  u(0)=u_0 .
\end{equation}
The a priori estimates are proved under \(F\in C_b^2(\mathbb R)\).  The uniqueness  additionally assumes
\begin{equation}\label{eq:F-lip-hyp}
  F''\text{ is globally Lipschitz.}
\end{equation}
Throughout the paper
\begin{equation*}
  (\mathcal I f)(t):=\int_0^t P_{t-r}f(r)\,\mathrm{d}r,
  \qquad
  P_t=e^{t\Delta}.
\end{equation*}
Thus \(\mathcal L\mathcal I f=f\) and \((\mathcal I f)(0)=0\).

\begin{definition} 	 \label{enhanced}
	For $s_0>0$, we say that  $\eta\in L_T^\infty\cC_{s_0}^{-1-\kappa}$ with $\kappa<1/3$ can be lifted to an enhanced noise if there exist $\eta_n\in L^\infty_TC^\infty$
	such that
	$\eta_n$ converges to $\eta$ in $L_T^\infty \cC_{s_0}^{-1-\kappa}$, and  there are constants $c_n$ and a function $\Psi\in L_T^\infty \cC_{s_0}^{-2\kappa}$ such that
	\begin{align*}
		\lim_{n\to\infty}\|\sI (\eta_n)\circ\eta_n-c_n-\Psi\|_{L_T^\infty \cC_{s_0}^{-2\kappa}}=0\;,
	\end{align*}
	where $\circ$ denotes the resonant product in Bony's decomposition.

\end{definition}


The following is our main theorem. The precise notion of the solution requires additional notation, which we postpone; see Definition~\ref{def:pc-solution}.

\begin{theorem}\label{thm:main}
Let \(0<\kappa<\sqrt{5}-2\). There exists \(\mu_*>0\) such that, for all \(s_0,D_0>0\) with \(D_0+2s_0<\mu_*\), the following holds.
Let $\eta\in L^\infty_T\cC_{s_0}^{-1-\kappa}$ be liftable to an enhanced noise. Assume \(F\in C_b^2(\mathbb R)\), \(u_0\in C^{1-\kappa}(\rho_{D_0})\). Then, on every finite time interval, there exists a global solution to \eqref{eq:gpam} in the sense of Definition~\ref{def:pc-solution}.
If, in addition, \eqref{eq:F-lip-hyp} holds, then the solution is unique.
\end{theorem}

\begin{remark}[Relation with sine--Gordon]\label{rem:sine-gordon}
We expect that the approach developed in this paper can be used to prove global well-posedness of the dynamical sine--Gordon model on \(\mathbb R^2\) (see e.g. \cite{HS16,CHS18,BC26}), which would require a parabolic space-time, finite-dimensional vector-valued version of the same deterministic argument, with a matrix enhanced noise and space--time paraproducts.  For the dynamical sine--Gordon equation on \(\mathbb R^2\), write
\[
  (\partial_t-\frac12\Delta)\Phi=\sin(b\Phi)+\zeta,
  \qquad (\partial_t-\frac12\Delta)Z=\zeta,
\]
for space-time white noise $\zeta$ on \(\mathbb R^2\). After the Da Prato--Debussche transform \(\Phi=Z+u\), the Wick fields
\[
  \xi_c=:\!\cos(bZ)\!:,
  \qquad
  \xi_s=:\!\sin(bZ)\!:
\]
have parabolic regularity below \(-\lambda_b\), where \(\lambda_b=b^2/(4\pi)\).  The threshold of this paper suggests the range
\[
  \lambda_b<\sqrt{5}-1,
  \qquad\text{equivalently}\qquad
  b^2<4\pi(\sqrt{5}-1).
\]
We do not pursue this parabolic/vector extension here.
\end{remark}

\section{High--low decomposition}\label{sec:parameters-weights}


In this section we construct a high--low decomposition of the enhanced noise. We prove the annular localization lemma, Lemma~\ref{lem:bridge-localization},  which is a variant of the localization technique used in \cite{GH18,GH18a}.  Proposition~\ref{prop:annular-split} then applies this lemma to the enhanced noise.

Let \((\chi_k)_{k\geqslant0}\) be a smooth dyadic annular partition of unity on
\(\mathbb R^2\), and let \((\widetilde\chi_k)_{k\geqslant0}\) be slightly enlarged
cutoffs such that
\[
  \widetilde\chi_k\chi_k=\chi_k,
\]
(see Appendix \ref{sec:tools} for the detailed definition).
For an integer $J$, we use the Littlewood--Paley truncation notation
\[
  \Delta_{>J}:=\sum_{j>J}\Delta_j,
  \qquad
  \Delta_{\leqslant J}:=\sum_{j\leqslant J}\Delta_j .
\]

\begin{lemma}\label{lem:bridge-localization}

Let $f\in C^{\beta_0}(\rho_s)$ with $\beta_0\in\mathbb{R}$ and $s>0$, where $\rho_s(x)=\la x\ra^{-s}$ as in \eqref{eq:weighted-holder}.
For \(\beta_h<\beta_0\), \(\sigma>0\), and \(R\geqslant0\), there exists a decomposition of $f=f_R^>+f_R^\leqslant$ such that
 uniformly in \(R\),
\[
  \|f_R^>\|_{C^{\beta_h}(\rho_{-\sigma})}
  \lesssim
  2^{-(\beta_0-\beta_h)R}\|f\|_{C^{\beta_0}(\rho_s)},
\]
and
\[
  \|f_R^>\|_{C^{\beta_0}(\rho_s)}
  +
  \|f_R^\leqslant\|_{C^{\beta_0}(\rho_s)}
  \lesssim \|f\|_{C^{\beta_0}(\rho_s)}.
\]
Moreover, for every \(\beta<\beta_0\),
\[
  \|f_R^>\|_{C^\beta(\rho_{-\sigma(\beta)})}
  \lesssim
  2^{-(\beta_0-\beta)R}\|f\|_{C^{\beta_0}(\rho_s)},
  \qquad
  \sigma(\beta)
  :=
  \frac{\beta_0-\beta}{\beta_0-\beta_h}(s+\sigma)-s ,
\]
and, for every \(\beta_\ell>\beta_0\),
\[
  \|f_R^\leqslant\|_{C^{\beta_\ell}(\rho_{d(\beta_\ell)})}
  \lesssim
  2^{(\beta_\ell-\beta_0)R}\|f\|_{C^{\beta_0}(\rho_s)},
  \qquad
  d(\beta_\ell)
  :=
  s+(\beta_\ell-\beta_0)
  \frac{s+\sigma}{\beta_0-\beta_h}.
\]
\end{lemma}

\begin{proof} Let  \(f_k=\chi_k f\). Then by Lemma \ref{lem:weighted-lp-annular} it satisfies
\[
  \|f_k\|_{C^{\beta_0}}\leqslant \|f\|_{C^{\beta_0}(\rho_s)}2^{sk},
  \qquad k\geqslant0.
\]
 Let \[
  J_k:=
  \left\lceil
    R+\frac{s+\sigma}{\beta_0-\beta_h}k
  \right\rceil ,
\]
and
\[
  f_R^>
  :=
  \sum_{k\geqslant0}\widetilde\chi_k\Delta_{>J_k}f_k,
  \qquad
  f_R^\leqslant
  :=
  \sum_{k\geqslant0}\widetilde\chi_k\Delta_{\leqslant J_k}f_k .
\]
It follows directly that $f=f_R^>+f_R^\leqslant$.
The integer parts $\left\lceil\cdots \right\rceil$ in \(J_k\) only affect the implicit constants.  For any
\(\beta<\beta_0\), Bernstein's inequality gives
\[
  \|\Delta_{>J_k}f_k\|_{C^\beta}
  \lesssim
  2^{-(\beta_0-\beta)J_k}\|f_k\|_{C^{\beta_0}}
  \lesssim
  \|f\|_{C^{\beta_0}(\rho_s)}2^{-(\beta_0-\beta)R}
  2^{-\sigma(\beta)k}.
\]
Taking \(\beta=\beta_h\) gives \(\sigma(\beta_h)=\sigma\).  Lemma \ref{lem:weighted-lp-annular} then yields the high-block bounds.

The rough bounds follow from the uniform boundedness of
\(\Delta_{>J_k}\) and \(\Delta_{\leqslant J_k}\) on \(C^{\beta_0}\):
\[
  \|\Delta_{>J_k}f_k\|_{C^{\beta_0}}
  +
  \|\Delta_{\leqslant J_k}f_k\|_{C^{\beta_0}}
  \lesssim \|f\|_{C^{\beta_0}(\rho_s)}2^{sk}.
\]
Lemma \ref{lem:weighted-lp-annular}  gives the two \(C^{\beta_0}(\rho_s)\) bounds.

For \(\beta_\ell>\beta_0\), Bernstein's inequality in the opposite direction gives
\[
  \|\Delta_{\leqslant J_k}f_k\|_{C^{\beta_\ell}}
  \lesssim
  2^{(\beta_\ell-\beta_0)J_k}\|f_k\|_{C^{\beta_0}}
  \lesssim
  \|f\|_{C^{\beta_0}(\rho_s)}2^{(\beta_\ell-\beta_0)R}2^{d(\beta_\ell)k}.
\]
Lemma \ref{lem:weighted-lp-annular}  gives the low-block estimate.
\end{proof}

\begin{remark}[Comparison with the localization lemma]\label{rem:gh-localization-comparison}
 In \cite[Lemma~2.4]{GH18},  \(
  J_k:=
  \left\lceil
    R+c_k
  \right\rceil
\) for $c_k\simeq k$, which means $s+\sigma=\beta_0-\beta_h$. This choice ties changes in regularity to changes in the polynomial weight: a loss of \(\delta\) derivatives for the
high block is accompanied by a gain of \(\delta\) powers in the weight, and
a gain of \(\gamma\) derivatives for the low block is accompanied by a loss
of \(\gamma\) powers in the weight. For the estimates in the present paper, however, this fixed coupling between regularity and polynomial weight is too restrictive. 
In the diagonal form of \cite[Lemma~2.4]{GH18}, the required regularity gain, of order \(1+\kappa+\varepsilon\), would necessarily produce a polynomial weight loss of the same order. 
Such a loss is incompatible with the small-weight regime used in the maximum-principle estimate.
\end{remark}

In preparation for the annular high--low decomposition and the later a priori estimates, we now fix the regularity exponents and polynomial weights used throughout the argument.
Fix $0<\kappa<\sqrt5-2$.  In the following $\varepsilon>0$ is chosen after $\kappa$ and is sufficiently small.
Recall $D_0,s_0$ in Theorem~\ref{thm:main}.

The choices below are made in the following order. First \(\kappa\) is fixed, then \(\varepsilon\) is chosen sufficiently small, and finally the total polynomial-weight scale \(\mu_{\rm wt}\) is taken small enough. All strict inequalities used later are ensured by these successive choices.

\begin{definition}\label{def:kappa-admissible-weights}
 We define the following parameters and weights.

(1) (Some regularity parameters and a total small weight.)
Set
\begin{equation}\label{eq:alpha-eps-choice}
  \alpha:=1-\kappa-5\varepsilon,\qquad s:=2-\varepsilon/2,\qquad \mu_{\rm wt}:=D_0+2s_0.
\end{equation}
After decreasing $\varepsilon$ if necessary, we shall use
\[
  2\alpha-1-\kappa-2\varepsilon>0,
  \qquad
  \alpha-2\kappa-3\varepsilon>0.
\]
These two positive quantities are the high-frequency gains for $\eta$ and $\Psi$, respectively.

(2) (Some parameters for the final a priori bound  in Section~\ref{sec:closure}.)
Set
\begin{equation}\label{eq:m-GammaR}
  m_\varepsilon:=\left(\frac{s}{\alpha}(2\alpha-1-\kappa-2\varepsilon)\right)\wedge(1-\kappa-\varepsilon),\qquad   \theta:=\frac{1+\kappa+\varepsilon}{1+\kappa+\varepsilon+m_\varepsilon}.
\end{equation}
 For the weight used later, we then introduce the following two parameters:
\begin{equation}\label{eq:theta-ell-def}
\ell_+:=\frac{(1+(\kappa+5\varepsilon))^{-1}-\frac{1+\varepsilon}{s}}{1-\frac{1+\varepsilon}{s}},\quad \ell:=\frac{\theta+\ell_+}{2}.
\end{equation}

(3) (More ``$D$-type'' weights.)
By using $\ell$ above, the solution weights are defined as
\begin{equation}\label{eq:D-scale}
  D_L:=\mu_{\rm wt}^{1/2},
  \qquad
  D:=\frac{D_L}{\ell},
  \qquad
  D_1:=D_L-\mu_{\rm wt}^{3/4}.
\end{equation}
The interpolation weights are defined as
\begin{equation}\label{eq:Dalpha-Deps}
  D_\alpha:=\left(1-\frac\alpha s\right)D_L+\frac\alpha sD.
\end{equation}

(4)
(Parameters used in splitting noise $\eta$.)
Set
\begin{equation}\label{eq:sigma-eta-endpoints}
\begin{aligned}
  \sigma_{\eta,-}&:=
  \max\left\{
      D_\alpha-D_1,
      \frac{2\alpha-1-\kappa-2\varepsilon}{1-\kappa-\varepsilon}(D-D_L+s_0)-s_0
  \right\},\\
  \sigma_{\eta,+}&:=
  \frac{2\alpha-1-\kappa-2\varepsilon}{1+\kappa+\varepsilon}(D_L-2s_0)-s_0,
\end{aligned}
\end{equation}
\begin{equation}\label{eq:sigma-eta-midpoint}
  \sigma_\eta:=\frac12(\sigma_{\eta,-}+\sigma_{\eta,+}),
\end{equation}
\begin{equation}\label{eq:sigma-etad-deta}
\begin{aligned}
  \sigma_{\eta,\varepsilon}
  &:=\frac{1-\kappa-\varepsilon}{2\alpha-1-\kappa-2\varepsilon}(s_0+\sigma_\eta)-s_0,\\
  d_\eta
  &:=s_0+(1+\kappa+\varepsilon)
  \frac{s_0+\sigma_\eta}{2\alpha-1-\kappa-2\varepsilon}.
\end{aligned}
\end{equation}

(5) (Parameters used in splitting noise $\Psi$.)
We set
\begin{equation}\label{eq:sigma-Psi-endpoints}
  \sigma_{\Psi,-}:=D_\alpha-D_1,
  \qquad
  \sigma_{\Psi,+}:=\frac{\alpha-2\kappa-3\varepsilon}{2\kappa+\varepsilon}(D_L-s_0)-s_0,
  \qquad
  \sigma_\Psi:=\frac12(\sigma_{\Psi,-}+\sigma_{\Psi,+}),
\end{equation}
\begin{equation}\label{eq:dPsi}
  d_\Psi:=s_0+(2\kappa+\varepsilon)
  \frac{s_0+\sigma_\Psi}{\alpha-2\kappa-3\varepsilon}.
\end{equation}
\end{definition}

Let us briefly describe the purposes of these parameters in the above Definition~\ref{def:kappa-admissible-weights}:
\begin{enumerate}[itemsep=5pt, parsep=0pt]
\item
In \eqref{eq:alpha-eps-choice},
 $\alpha,s$ are the parameters for regularity of the solutions used later.
\item
The quantities $m_\varepsilon,\theta$ in \eqref{eq:m-GammaR} are chosen so that the
monomials of the scalar quantities are sublinear in the final estimate; these are important in
the scalar closure arguments in the proof of Theorem~\ref{thm:apriori}.
\item
In \eqref{eq:theta-ell-def}, $\ell_+$ will play the role of the upper admissible weight-ratio
constraint imposed by the Schauder estimate for the transport term $B_R(u)\cdot\nabla w$ (treated as a product), while $\ell$ is
the actual chosen ratio $D_L/D$, with strict room between the scalar
cutoff constraint $\theta$ and the transport weight constraint $\ell_+$.
\item
In \eqref{eq:D-scale},
the smaller weight $D_L$ is used in the $L^\infty$ norm, while the larger weight $D$ is
used in the higher regularity norm.  The slightly smaller weight $D_1$ is
reserved for  $u^\sharp_1$ in the following decomposition; the small difference
$D_L-D_1=\mu_{\rm wt}^{3/4}$ leaves room for the stochastic term weight $s_0$ in
products that are finally estimated with weight $D_L$.
The weight $D_\alpha$ in \eqref{eq:Dalpha-Deps} is the weight obtained by interpolating the
$L^\infty(\rho_{D_L})$ and $C^s(\rho_D)$  to obtain the
$C^\alpha$ control used later.
\item
Regarding \eqref{eq:sigma-eta-endpoints}--\eqref{eq:sigma-eta-midpoint}:
The parameter $\sigma_\eta$ is the polynomial decay assigned to the high
 block $\eta^>_R$.  The lower endpoint in \(\sigma_{\eta,-}\) guarantees $D_\alpha-\sigma_\eta<D_1$ and
$D-\sigma_{\eta,\varepsilon}<D_L$, which are needed when the high block noise is
paired respectively with the $C^\alpha$ component and with the  remainder
$w$ below.
The upper endpoint \(\sigma_{\eta,+}\)  ensures that the induced low-block weight fits
inside the maximum-principle weight, in particular $d_\eta+s_0<D_L$.

The quantities $\sigma_{\eta,\varepsilon}$ and $d_\eta$  in \eqref{eq:sigma-etad-deta} are the weights produced by the high-low decomposition
when the high block is measured in $C^{-2+\varepsilon}$ and the low
block is measured in $C^\varepsilon$.
\item
Similarly, in \eqref{eq:sigma-Psi-endpoints}--\eqref{eq:dPsi}, $\sigma_\Psi$ is the decay assigned to the high block
$\Psi^>_{R_1}$, while $d_\Psi$ is the growth of the corresponding low block.
The condition on $\sigma_{\Psi,-}$ gives
$D_\alpha-\sigma_\Psi<D_1$, and $\sigma_{\Psi,+}$ gives
$d_\Psi<D_L$, which are important later.
\end{enumerate}


We first record only the elementary consequences of the parameter choices that are needed for the annular high--low decomposition.

\begin{lemma}
\label{lem:split-weights}

It holds that
\begin{equation}\label{eq:ell-window}
  \theta<\ell<\ell_+
\end{equation}
and
\begin{equation}\label{eq:split-intervals}
  \sigma_{\eta,-}<\sigma_{\eta,+},
  \qquad
  \sigma_{\Psi,-}<\sigma_{\Psi,+}.
\end{equation}
\end{lemma}

\begin{proof}
We first prove \(\theta<\ell_+\).  At \(\varepsilon=0\),
\[
  \alpha\to1-\kappa,
  \qquad s\to2.
\]
Moreover, the second term in the definition of $m_\varepsilon$ is active for sufficiently small $\varepsilon$
precisely when  \(\kappa^2+4\kappa-1<0\) which is true since $\kappa<\sqrt 5-2$, so
 \(m_\varepsilon\to1-\kappa\). Also,
\[
  \theta\to\frac{1+\kappa}{2},
  \qquad
  \ell_+\to\frac{1-\kappa}{1+\kappa}.
\]
The inequality
\[
  \frac{1+\kappa}{2}<\frac{1-\kappa}{1+\kappa}
\]
is  equivalent to \(\kappa^2+4\kappa-1<0\).  Therefore \(\theta<\ell_+\) for sufficiently small \(\varepsilon\), and so \(\theta<\ell<\ell_+\).

We next show the first inequality in \eqref{eq:split-intervals}.  Since
\[
  D_L=\mu_{\rm wt}^{1/2},
  \qquad
  D=\frac{D_L}{\ell},
  \qquad
  D_1=D_L-\mu_{\rm wt}^{3/4},
\]
one has \(s_0/D_L\to0\) and \((D_L-D_1)/D_L\to0\) as \(\mu_{\rm wt}\downarrow0\).  Thus it suffices to check the leading inequalities.  Since
\[
  D-D_L=\frac{1-\ell}{\ell}D_L,
  \qquad
  D_\alpha-D_1=\frac{\alpha}{s}(D-D_L)+(D_L-D_1),
\]
the two leading comparisons needed for \(\sigma_{\eta,-}<\sigma_{\eta,+}\) are
\[
  \frac{\alpha}{s}(D-D_L)<\frac{1-3\kappa}{1+\kappa}D_L,
  \qquad
  \frac{1-3\kappa}{1-\kappa}(D-D_L)<\frac{1-3\kappa}{1+\kappa}D_L.
\]
They are equivalent to
\[
  \ell>\frac{1+\kappa}{1+\kappa+\frac{s}{\alpha}(1-3\kappa)},
  \qquad
  \ell>\frac{1+\kappa}{1-\kappa+1+\kappa}.
\]
Both follow from
\[
  m_\varepsilon\leqslant\frac{s}{\alpha}(1-3\kappa),
  \qquad
  m_\varepsilon\leqslant 1-\kappa,
  \qquad
  \theta=\frac{1+\kappa+\varepsilon}{1+\kappa+\varepsilon+m_\varepsilon}<\ell.
\]
The lower-order terms \(s_0\) and \(D_L-D_1\) are \(o(D_L)\), so decreasing \(\mu_{\rm wt}\) gives \(\sigma_{\eta,-}<\sigma_{\eta,+}\).

For the second inequality, set \(A_\Psi=\alpha-2\kappa-3\varepsilon\) and \(C_\Psi=2\kappa+\varepsilon\).  The leading comparison for \(\sigma_{\Psi,-}<\sigma_{\Psi,+}\) is
\[
  \frac{\alpha}{s}(D-D_L)<\frac{A_\Psi}{C_\Psi}D_L,
\]
or equivalently
\[
  \ell>\frac{\alpha C_\Psi}{\alpha C_\Psi+sA_\Psi}.
\]
At \(\varepsilon=0\), the right-hand side converges to
\[
  \frac{\kappa(1-\kappa)}{1-2\kappa-\kappa^2}.
\]
Furthermore,
\[
  \frac{1+\kappa}{2}-\frac{\kappa(1-\kappa)}{1-2\kappa-\kappa^2}
  =
  \frac{1-3\kappa-\kappa^2-\kappa^3}{2(1-2\kappa-\kappa^2)}>0
\]
for \(0<\kappa<\sqrt5-2\).  Since \(\theta\to(1+\kappa)/2\) and \(\ell>\theta\), the leading inequality holds for small \(\varepsilon\).  Again the perturbative terms are \(o(D_L)\), so shrinking $\mu_{\rm wt}$ gives \(\sigma_{\Psi,-}<\sigma_{\Psi,+}\).
\end{proof}

In the next proposition, we split the enhanced noise into high and low components and record the corresponding bounds.

\begin{proposition}\label{prop:annular-split}
Assume that
$  \boldsymbol\eta=(\eta,\Psi)$
is an enhanced noise in the sense of Definition~\ref{enhanced}.  Let \(\sigma_{\eta,\varepsilon}\), \(d_\eta\), and \(d_\Psi\) be fixed as in Definition~\ref{def:kappa-admissible-weights}, see \eqref{eq:sigma-etad-deta} and \eqref{eq:dPsi}. For every \(R,R_1\geqslant0\), there are decompositions $\eta=\eta_R^>+\eta_R^\leqslant$ and $\Psi=\Psi_{R_1}^>+\Psi_{R_1}^\leqslant$ such that the following bounds hold uniformly in \(R,R_1\):
\begin{align}
  \norm{\eta_R^>}_{L_T^\infty\cC_{-\sigma_\eta}^{-(2\alpha-2\varepsilon)}}&\leqslant C2^{-(2\alpha-1-\kappa-2\varepsilon)R},
  &\norm{\eta_R^>}_{L_T^\infty\cC_{s_0}^{-1-\kappa}}+
  \norm{\eta_R^\leqslant}_{L_T^\infty\cC_{s_0}^{-1-\kappa}}&\leqslant C,
  \label{eq:eta-high-lambda}\\
  \norm{\eta_R^>}_{L_T^\infty\cC_{s_0}^{\alpha-2}}&\leqslant C2^{-(1-\kappa-\alpha)R},
  &\norm{\eta_R^>}_{L_T^\infty\cC_{-\sigma_{\eta,\varepsilon}}^{-2+\varepsilon}}&\leqslant C2^{-(1-\kappa-\varepsilon)R},
  \label{eq:eta-high-alpha-time}\\
  \norm{\eta_R^\leqslant}_{L_T^\infty\cC_{d_\eta}^{\varepsilon}}&\leqslant C2^{(1+\kappa+\varepsilon)R}.&&
  \label{eq:eta-low}
\end{align}
Moreover,
\begin{align}
  \norm{\Psi_{R_1}^>}_{L_T^\infty\cC_{-\sigma_\Psi}^{-\alpha+2\varepsilon}}&\leqslant C2^{-(\alpha-2\kappa-3\varepsilon)R_1},
  &\norm{\Psi_{R_1}^\leqslant}_{L_T^\infty\cC_{d_\Psi}^{\varepsilon}}&\leqslant C2^{(2\kappa+\varepsilon)R_1},
  \label{eq:Psi-high}\\
  \norm{\Psi_{R_1}^>}_{L_T^\infty\cC_{s_0}^{-2\kappa}}
  +\norm{\Psi_{R_1}^\leqslant}_{L_T^\infty\cC_{s_0}^{-2\kappa}}&\leqslant C.&&
  \label{eq:Psi-rough-bound}
\end{align}
\end{proposition}

\begin{proof}
We apply Lemma~\ref{lem:bridge-localization} for $\eta$ with
\[
  \beta_0=-1-\kappa,\qquad
  \beta_h=-(2\alpha-2\varepsilon),\qquad
  \sigma=\sigma_\eta.
\]
This gives \eqref{eq:eta-high-lambda}--\eqref{eq:eta-low}.
For the split of $\Psi$, we apply Lemma~\ref{lem:bridge-localization} with
\[
  \beta_0=-2\kappa,\qquad
  \beta_h=-\alpha+3\varepsilon,\qquad
  \beta_\ell=\varepsilon,\qquad
  \sigma=\sigma_\Psi.
\]
Then $\beta_0-\beta_h=\alpha-2\kappa-3\varepsilon$ and $\beta_\ell-\beta_0=2\kappa+\varepsilon$, which give the decay factor and the low-block weight in \eqref{eq:Psi-high}.  The high block is first bounded in $C^{-\alpha+3\varepsilon}(\rho_{-\sigma_\Psi})$ and hence, by the continuous embedding $C^{-\alpha+3\varepsilon}\hookrightarrow C^{-\alpha+2\varepsilon}$, also in the displayed space.  The rough bound \eqref{eq:Psi-rough-bound} follows from the uniform boundedness of the high and low pieces in the original \(C^{-2\kappa}(\rho_{s_0})\) norm.
\end{proof}

  For later use, we also introduce the following notations:
\[
 \Psi_R^{>>}=\sI(\eta_R^>)\res\eta_R^>
\]
with the renormalization induced by the full enhancement.
 Equivalently, and this is the definition for a rough enhanced datum, it is given by
\begin{equation*}
 \Psi_R^{>>}:=\Psi-\Psi_R^{\rm mix},
\end{equation*}
where the mixed terms contain at least one low factor and are classically well-defined,
\begin{equation*}
  \Psi_R^{\rm mix}
  :=(\sI\eta_R^\leqslant)\res\eta_R^\leqslant
    +(\sI\eta_R^\leqslant)\res\eta_R^>
    +(\sI\eta_R^>)\res\eta_R^\leqslant.
\end{equation*}
Then
\begin{equation}\label{eq:inherited-high-second}
 \Psi_R^{>>}
  =\Psi_{R_1}^>+\Theta_{R,R_1},
  \qquad
  \Theta_{R,R_1}:=\Psi_{R_1}^\leqslant-\Psi_R^{\rm mix}.
\end{equation}

\begin{lemma}\label{lem:inherited-second-level}
It holds that
\begin{equation}\label{eq:mix-second-level-bound}
  \norm{\Psi_R^{\rm mix}}_{L_T^\infty\cC_{s_0+d_\eta}^{\varepsilon}}
  \leqslant C2^{(1+\kappa+\varepsilon)R},
\end{equation}
\begin{equation}\label{eq:Theta-catalogue-term}
  \norm{\Theta_{R,R_1}}_{L_T^\infty L^\infty(\rho_{D_L})}
  \lesssim
2^{(1+\kappa+\varepsilon)R}+2^{(2\kappa+\varepsilon)R_1}.
\end{equation}
\end{lemma}

\begin{proof}
  From \eqref{eq:eta-high-lambda} and Lemma~\ref{lem:mild-time-increments}
\[
  \norm{\sI\eta_R^>}_{L_T^\infty\cC_{s_0}^{1-\kappa}}
  +\norm{\sI\eta_R^\leqslant}_{L_T^\infty\cC_{s_0}^{1-\kappa}}
  \leqslant C.
\]
From \eqref{eq:eta-low},
\[
  \norm{\sI\eta_R^\leqslant}_{L_T^\infty\cC_{d_\eta}^{2+\varepsilon}}
  +\norm{\eta_R^\leqslant}_{L_T^\infty\cC_{d_\eta}^{\varepsilon}}
  \leqslant C2^{(1+\kappa+\varepsilon)R}.
\]
Lemma \ref{lem:para} then gives the positive regularity bound
\[
  \norm{(\sI\eta_R^\leqslant)\res\eta_R^\leqslant}_{L_T^\infty\cC_{s_0+d_\eta}^{\varepsilon}}
  +\norm{(\sI\eta_R^\leqslant)\res\eta_R^>}_{L_T^\infty\cC_{s_0+d_\eta}^{\varepsilon}}
  +\norm{(\sI\eta_R^>)\res\eta_R^\leqslant}_{L_T^\infty\cC_{s_0+d_\eta}^{\varepsilon}}
  \leqslant C2^{(1+\kappa+\varepsilon)R}.
\]
All three resonant products are well defined since \(1+\varepsilon-\kappa>0\),
and we embed \(C^{1+\varepsilon-\kappa}\) into \(C^\varepsilon\).

We now check \eqref{eq:Theta-catalogue-term}.
Set
\[
  A_\eta:=2\alpha-1-\kappa-2\varepsilon,
  \qquad
  C_\eta:=1+\kappa+\varepsilon.
\]
Since the first split interval is non-empty by Lemma~\ref{lem:split-weights} and \(\sigma_\eta\) is its midpoint, we have \(\sigma_\eta<\sigma_{\eta,+}\).  Hence, by \eqref{eq:sigma-eta-endpoints},
\[
  \sigma_\eta<\frac{A_\eta}{C_\eta}(D_L-2s_0)-s_0,
  \qquad\text{so}\qquad
  \frac{C_\eta(s_0+\sigma_\eta)}{A_\eta}<D_L-2s_0.
\]
Using the definition of \(d_\eta\) in \eqref{eq:sigma-etad-deta}, this gives
\begin{align}\label{eq:DL}
  s_0+d_\eta
  =2s_0+\frac{C_\eta(s_0+\sigma_\eta)}{A_\eta}
  <D_L.
\end{align}
Similarly, set
\[
  A_\Psi:=\alpha-2\kappa-3\varepsilon,
  \qquad
  C_\Psi:=2\kappa+\varepsilon.
\]
Since \(\sigma_\Psi<\sigma_{\Psi,+}\), \eqref{eq:sigma-Psi-endpoints} gives
\[
  \sigma_\Psi<\frac{A_\Psi}{C_\Psi}(D_L-s_0)-s_0,
  \qquad\text{so}\qquad
  \frac{C_\Psi(s_0+\sigma_\Psi)}{A_\Psi}<D_L-s_0.
\]
By the definition of \(d_\Psi\) in \eqref{eq:dPsi},
\[
  d_\Psi=s_0+\frac{C_\Psi(s_0+\sigma_\Psi)}{A_\Psi}<D_L.
\]
Therefore
\[
  \norm{\Psi_R^{\rm mix}}_{L_T^\infty L^\infty(\rho_{D_L})}
  \lesssim
  \norm{\Psi_R^{\rm mix}}_{L_T^\infty C^\varepsilon(\rho_{s_0+d_\eta})},
\]
and 
\[
  \norm{\Psi_{R_1}^\leqslant}_{L_T^\infty L^\infty(\rho_{D_L})}
  \lesssim
  \norm{\Psi_{R_1}^\leqslant}_{L_T^\infty C^\varepsilon(\rho_{d_\Psi})}.
\]
    Combining \eqref{eq:mix-second-level-bound} with \eqref{eq:Psi-high} proves \eqref{eq:Theta-catalogue-term}.
\end{proof}

\section{Paracontrolled decomposition}\label{sec:model-pde-pc}

In this section we  formulate the paracontrolled solution  and then replace the  ansatz by a cutoff representative based on the high--low split of the enhanced noise.
Via this representative we separate a directly controlled first term from a final remainder, so that the latter can later be estimated by a maximum principle and a Schauder estimate in two different polynomial weights.

\subsection{Paracontrolled decomposition}\label{sec:pc}

We first give the definition of a paracontrolled solution. The ansatz is:
\begin{equation}\label{eq:pc-ansatz}
  u=F(u)\prec \mathcal I\eta+u^\sharp .
\end{equation}
Set
\[
  H(u):=F(u)F'(u),
  \qquad
  \operatorname{Com}(f,g,h):=(f\prec g)\circ h-f(g\circ h).
\]
The product after Bony decomposition is
\begin{equation}\label{eq:pc-product}
\begin{aligned}
 F(u) \eta
  :=\;&F(u)\prec\eta+F(u)\succ\eta
  +\big(F(u)-F'(u)\prec u\big)\circ\eta
  +(F'(u)\prec u^\sharp)\circ\eta                                      \\
  &+\operatorname{Com}\big(F'(u),F(u)\prec\mathcal I\eta,\eta\big)
  +F'(u)\operatorname{Com}\big(F(u),\mathcal I\eta,\eta\big)
  +H(u)\Psi .
\end{aligned}
\end{equation}

To define a paracontrolled solution we introduce the following regularity exponents
\begin{equation}\label{eq:s-r-gamma}
  r:=1+\kappa+7\varepsilon,
  \qquad
  \gamma:=\frac{2\alpha+\kappa-1}{2}+\varepsilon .
\end{equation}
Moreover, after decreasing \(\varepsilon\) if necessary,
\begin{equation}\label{eq:time-product-margin}
  2\gamma>r-1+\kappa,
  \qquad 0<\gamma<1 .
\end{equation}

\begin{definition}[Paracontrolled solution]\label{def:pc-solution}
 A pair \((u,u^\sharp)\) is called a paracontrolled solution if
\[
  u\in \mathbb S_T^{\alpha,\gamma\alpha/s;\alpha/2}(\rho_{D_\alpha}),
  \qquad
  u^\sharp\in C_T C^{1-\kappa}(\rho_D)
  \cap \mathbb S_T^{r,\gamma;\alpha/2}(\rho_D),
\]
and if \eqref{eq:pc-ansatz} holds
		and the equation \eqref{eq:gpam} holds in the analytically weak sense with the  product $F(u)  \eta$  given by \eqref{eq:pc-product}.
\end{definition}

Let
\begin{equation*}
  \cR(u):=F(u)-F'(u)\para u.
\end{equation*}
For cutoffs $R,R_1$ to be selected dynamically, write the cutoff representative as
\begin{equation}\label{eq:ansatz}
  u=F(u)\para\sI(\eta_R^>)+u^\sharp_R,
  \qquad
  u^\sharp_R=u_1^\sharp+w.
\end{equation}
Since \((\sI\eta_R^>)(0)=0\), the choices \(u_1^\sharp(0)=u_0\) and \(w(0)=0\) give the original initial condition \(u(0)=u_0\).  $u_1^\sharp$ is defined in mild form by
\begin{equation}\label{eq:u1-mild}
  u_1^\sharp
  =P_tu_0+[\sI,F(u)\para]\eta_R^>
   +\sI\bigl(F(u)\Par\eta_R^>+H(u)\para\Psi_{R_1}^>\bigr),
\end{equation}
where
\[
  [\sI,F(u)\para]\eta_R^>
  :=\sI(F(u)\para\eta_R^>)-F(u)\para\sI(\eta_R^>).
\]
 The remainder solves
\begin{equation}\label{eq:w-eq-pre}
\begin{cases}
  \cL w=\cR(u)\res\eta_R^>+F(u)\eta_R^\leqslant+(F'(u)\para u)\res\eta_R^>-H(u)\para\Psi_{R_1}^>,\\
  w(0)=0.
\end{cases}
\end{equation}

\begin{remark}[Comparison with the decomposition of \cite{SZZ24}]\label{rem:comparison-szz24-split}
  In \cite{SZZ24} the equation of $u_1^\sharp$ contains the full  product
\(
  H(u)\Psi^>
= H(u)\para\Psi^>+H(u)\res\Psi^>+H(u)\Par\Psi^> \).
The last two pieces require positive regularity of the coefficient, for instance
\(H(u)\in C^\alpha\), and their estimates therefore involve the size of the solution.  Consequently, the high-regularity estimate for \(u_1^\sharp\) contains a contribution of the form
$  2^{-(\alpha-2\kappa-3\varepsilon)R_1} \|u\|_\alpha $, see \cite[(4.9)]{SZZ24}.
This creates a stronger restriction because the transport representation later differentiates \(u_1^\sharp\), through terms such as
\(B_R(u)\cdot\nabla u_1^\sharp\) and $\cE_R^{\rm fr}(x)$ defined in \eqref{eq:E-freeze-def}.

In the present decomposition we put only the  paraproduct
\(
  H(u)\para\Psi_{R_1}^>
\) into the equation of $u_1^\sharp$.
This term is estimated by the paraproduct bound using only \(\|H(u)\|_{L^\infty}\), not the \(C^\alpha\)-norm of \(H(u)\).  This is why Proposition~\ref{prop:first-layer} gives \eqref{eq:u1-r}
with no factor \(U_\alpha\) multiplying the term with $R_1$.  The part
$  H(u)\res\Psi_{R_1}^>+H(u)\Par\Psi_{R_1}^>$
remains in the forcing term; see \eqref{eq:pc-resonant-expanded}.  There it is not differentiated by the transport representation, and it is estimated directly.
This is one of the mechanisms that reduce the final restrictions and yield the result for the larger range \(0<\kappa<\sqrt5-2\).
\end{remark}

\begin{remark}[Equivalence of the cutoff representative]\label{rem:cutoff-equivalence}
The decomposition with cutoffs \(R,R_1\) is only a representative of the solution in Definition~\ref{def:pc-solution}.    For any fixed admissible cutoffs, the cutoff ansatz
\eqref{eq:ansatz}
is equivalent to \eqref{eq:pc-ansatz} by the algebraic change of representative
\[
  u^\sharp=u^\sharp_R-F(u)\para\sI(\eta_R^\leqslant),
  \qquad
  u^\sharp_R=u^\sharp+F(u)\para\sI(\eta_R^\leqslant).
\]
Since \(\eta_R^\leqslant\) is a classical low-frequency object for fixed \(R\), this change preserves the regularities required in Definition~\ref{def:pc-solution} and in the cutoff space below.  Moreover,  \(\Psi_R^{>>}=\Psi_{R_1}^>+\Theta_{R,R_1}\); after adding back the low and mixed terms, the cutoff equations \eqref{eq:ansatz}--\eqref{eq:w-eq-pre} reconstruct exactly the full resonant product \eqref{eq:pc-product}.  Thus a solution constructed from the \((R,R_1)\)-decomposition is a  solution in the sense of Definition~\ref{def:pc-solution}, and conversely every Definition~\ref{def:pc-solution} solution admits this cutoff representative.
\end{remark}

\subsection{Estimate of \texorpdfstring{$u_1^\sharp$}{u1 sharp}}

We first define the scalar quantities
\begin{equation}\label{eq:L-S-def}
  L:=\norm{w}_{L_T^\infty L^\infty(\rho_{D_L})},
  \qquad
  S:=\norm{w}_{\mathbb S_T^{s,\gamma;s/2}(\rho_D)}.
\end{equation}
Here \(\gamma\) is the time parameter fixed in \eqref{eq:s-r-gamma}.
For later use, we also introduce the following weight parameter
 \[
  d_0:=2\mu_{\rm wt}.
\]
 By Lemma \ref{lem:interp},
\begin{equation}\label{eq:Ualpha-def}
  \norm{w}_{\mathbb S_T^{\alpha,\gamma\alpha/s;\alpha/2}(\rho_{D_\alpha})}
  \lesssim L^{1-\alpha/s}S^{\alpha/s}=:U_\alpha.
\end{equation}


\begin{proposition}\label{prop:first-layer}
There exist \(R_\eta,R_\Psi\geqslant1\) such that the following estimates hold for every \(R\geqslant R_\eta\) and every \(R_1\geqslant R_\Psi\):
\begin{equation}\label{eq:u1-low}
  \norm{u_1^\sharp}_{C_T C^{1-\kappa}(\rho_{d_0})}\leqslant C,
\end{equation}
 and
\begin{equation}\label{eq:u1-r}
  \norm{u_1^\sharp}_{\mathbb S_T^{r,\gamma;\alpha/2}(\rho_{D_1})}
  \leqslant C\bigl(1+2^{-(2\alpha-1-\kappa-2\varepsilon)R}U_\alpha+2^{-(\alpha-2\kappa-3\varepsilon)R_1}\bigr).
\end{equation}
Consequently
\begin{equation}\label{eq:u-alpha}
  \norm{u}_{\mathbb S_T^{\alpha,\gamma\alpha/s;\alpha/2}(\rho_{D_\alpha})}
  \leqslant C(1+U_\alpha).
\end{equation}
\end{proposition}

\begin{proof}
We estimate the four terms in the mild formula \eqref{eq:u1-mild}.  Since $F\in C_b^2(\mathbb R)$ and $H=FF'$ has bounded first derivative, the standard composition estimate for Besov--H\"older norms gives
\begin{equation}\label{eq:composition-first-layer-check}
  \|F(u)\|_{L_T^\infty L^\infty}+
  \|H(u)\|_{L_T^\infty L^\infty}
  \leqslant C,
  \qquad
  \|F(u)\|_{\mathbb S_T^{\alpha,\gamma\alpha/s;\alpha/2}(\rho_{D_\alpha})}
  \leqslant C(1+U_\alpha^{\rm tot}),
\end{equation}
where
\[
  U_\alpha^{\rm tot}:=
  \|u\|_{\mathbb S_T^{\alpha,\gamma\alpha/s;\alpha/2}(\rho_{D_\alpha})}.
\]

  By \eqref{eq:D-scale} and the definition of \(d_0\), for $\mu_{\rm wt}$ small enough we obtain
\begin{equation}\label{eq:initial-weight-margin}
  0<D_0<d_0<D_1<D_L,
  \qquad 0<s_0<d_0,
  \qquad D_1+s_0<D_L.
\end{equation}

\smallskip
\noindent\emph{Proof of the low regularity bound \eqref{eq:u1-low}.}
Lemma~\ref{lem:mild-time-increments}  gives
\begin{equation}\label{eq:first-layer-heat-low-detail}
  \|P_tu_0\|_{C_T C^{1-\kappa}(\rho_{d_0})}
  \lesssim \|u_0\|_{C^{1-\kappa}(\rho_{D_0})},
\end{equation}
since \(D_0<d_0\) from \eqref{eq:initial-weight-margin}.

Using Lemma~\ref{lem:mild-time-increments}, \eqref{eq:para-Linfty-left},  \eqref{eq:eta-high-lambda}, and  \(s_0<d_0\) from \eqref{eq:initial-weight-margin} give

\begin{equation}\label{eq:comm-low-duhamel-detail}
  \|\sI(F(u)\para\eta_R^>)\|_{C_T C^{1-\kappa}(\rho_{d_0})}\lesssim\|F(u)\para\eta_R^>\|_{L_T^\infty C^{-1-\kappa}(\rho_{d_0})}
  \lesssim
  \|F(u)\|_{L_T^\infty L^\infty}
  \|\eta_R^>\|_{L_T^\infty C^{-1-\kappa}(\rho_{s_0})}
  \lesssim1.
\end{equation}
Also \eqref{eq:para-Linfty-left} gives
\begin{equation}\label{eq:comm-low-product-detail}
  \|F(u)\para\sI\eta_R^>\|_{C_T C^{1-\kappa}(\rho_{d_0})}\lesssim\|\sI\eta_R^>\|_{L_T^\infty C^{1-\kappa}(\rho_{s_0})}
  \lesssim
  \|\eta_R^>\|_{L_T^\infty C^{-1-\kappa}(\rho_{s_0})}
  \lesssim1.
\end{equation}
 Lemma~\ref{lem:para}
and  Lemma~\ref{lem:mild-time-increments} yield
\begin{equation}\label{eq:right-para-low-detail}
  \|\sI(F(u)\Par\eta_R^>)\|_{C_T C^{1-\kappa}(\rho_{d_0})}
  \lesssim \|F(u)\Par\eta_R^>\|_{L_T^\infty C^{-1-\kappa}(\rho_{s_0})}
  \lesssim
  \|\eta_R^>\|_{L_T^\infty C^{-1-\kappa}(\rho_{s_0})}
  \|F(u)\|_{L_T^\infty L^\infty}
  \lesssim1.
\end{equation}
 By \eqref{eq:para-Linfty-left} and \eqref{eq:Psi-high},
\begin{equation*}
  \|H(u)\para\Psi_{R_1}^>\|_{L_T^\infty C^{-\alpha+2\varepsilon}(\rho_{-\sigma_\Psi})}
  \lesssim
  \|H(u)\|_{L_T^\infty L^\infty}
  \|\Psi_{R_1}^>\|_{L_T^\infty C^{-\alpha+2\varepsilon}(\rho_{-\sigma_\Psi})}
  \lesssim 2^{-(\alpha-2\kappa-3\varepsilon)R_1}.
\end{equation*}
Using \(r=1+\kappa+7\varepsilon>1+\kappa+\varepsilon\) and Lemma~\ref{lem:mild-time-increments}, we obtain
\begin{equation}\label{eq:Psi-left-low-detail}
  \|\sI(H(u)\para\Psi_{R_1}^>)\|_{C_T C^{1-\kappa}(\rho_{d_0})}
  \lesssim1.
\end{equation}
Combining \eqref{eq:first-layer-heat-low-detail}, \eqref{eq:comm-low-duhamel-detail}, \eqref{eq:comm-low-product-detail}, \eqref{eq:right-para-low-detail}, and \eqref{eq:Psi-left-low-detail} proves \eqref{eq:u1-low}.

\medskip
We now prove \eqref{eq:u1-r}, which is divided in several steps.

\noindent\emph{High spatial regularity.}
We will use
\begin{equation}\label{eq:first-level-high-product-margin}
  D_\alpha-\sigma_\eta<D_1 .
\end{equation}
Indeed, by \eqref{eq:split-intervals}  \(\sigma_\eta>\sigma_{\eta,-}\).  Since \(\sigma_{\eta,-}\geqslant D_\alpha-D_1\), we have \(\sigma_\eta>D_\alpha-D_1\), namely \eqref{eq:first-level-high-product-margin}.

Using Lemma~\ref{lem:parabolic-I-comm},
\eqref{eq:first-level-high-product-margin}, together with  \eqref{eq:eta-high-lambda} and \eqref{eq:composition-first-layer-check}, we obtain
\begin{equation*}
  \|[\sI,F(u)\para]\eta_R^>\|_{C_T^{r,\gamma\alpha/s}(\rho_{D_1})}
  \lesssim
  2^{-(2\alpha-1-\kappa-2\varepsilon)R}(1+U_\alpha^{\rm tot}).
\end{equation*}

Moreover, \eqref{eq:para-neg} gives
\begin{equation}\label{eq:u1-para-term-spatial}
\begin{aligned}
  \|F(u)\Par\eta_R^>\|_{L_T^{\infty,\gamma\alpha/s}C^{-\alpha+2\varepsilon}(\rho_{D_1})}
  &\lesssim
  \|\eta_R^>\|_{L_T^\infty C^{-(2\alpha-2\varepsilon)}(\rho_{-\sigma_\eta})}
  \|F(u)\|_{\mathbb S_T^{\alpha,\gamma\alpha/s;\alpha/2}(\rho_{D_\alpha})}  \\
  &\lesssim
  2^{-(2\alpha-1-\kappa-2\varepsilon)R}(1+U_\alpha^{\rm tot}).
\end{aligned}
\end{equation}
where we used \eqref{eq:first-level-high-product-margin} and the high-block estimate in \eqref{eq:eta-high-lambda}.

 \eqref{eq:para-Linfty-left} and \eqref{eq:Psi-high} give
\begin{equation}\label{eq:u1-Psi-term-spatial}
  \|H(u)\para\Psi_{R_1}^>\|_{L_T^{\infty,\gamma\alpha/s}C^{-\alpha+2\varepsilon}(\rho_{D_1})}
  \lesssim 2^{-(\alpha-2\kappa-3\varepsilon)R_1}.
\end{equation}
Using Lemma~\ref{lem:mild-time-increments}, we obtain
\begin{equation*}
  \|P_tu_0\|_{C_T^{r,\gamma}(\rho_{D_1})}+ \|P_tu_0\|_{C_T^{\alpha,\gamma\alpha/s}(\rho_{D_1})}
  \lesssim \|u_0\|_{C^{1-\kappa}(\rho_{D_0})},
\end{equation*}
because \((r-(1-\kappa))/2<\gamma\)  by \eqref{eq:time-product-margin}, and \(D_0<D_1\).  Combining the above gives
\begin{equation*}
  \|u_1^\sharp\|_{C_T^{r,\gamma}(\rho_{D_1})}+ \|u_1^\sharp\|_{C_T^{\alpha,\gamma\alpha/s}(\rho_{D_1})}
  \lesssim
  1+2^{-(2\alpha-1-\kappa-2\varepsilon)R}U_\alpha^{\rm tot}
  +2^{-(\alpha-2\kappa-3\varepsilon)R_1}.
\end{equation*}

\smallskip
\noindent\emph{Time increments.}
 This part is the weighted analogue of the time-regularity step in the finite-volume proof \cite[Section~4.3]{SZZ24}.  We write
\[
  \omega_\alpha:=\gamma\alpha/s.
\]
Since \(\alpha<s, D_1<D_\alpha\),  it is enough to prove estimates with the smaller weight \(\omega_\alpha,D_1\).

Since \(1-\kappa>\alpha\) and \(D_0<D_1\) by \eqref{eq:initial-weight-margin}, Lemma~\ref{lem:mild-time-increments} gives
\begin{equation*}
  \|P_tu_0\|_{C_T^{\alpha/2,\omega_\alpha}L^\infty(\rho_{D_1})}
  \lesssim \|u_0\|_{C^{1-\kappa}(\rho_{D_0})}.
\end{equation*}

Next set
\[
  G_\eta:=F(u)\Par\eta_R^>,
  \qquad
  G_\Psi:=H(u)\para\Psi_{R_1}^>.
\]
Lemma~\ref{lem:mild-time-increments} and \eqref{eq:u1-para-term-spatial}--\eqref{eq:u1-Psi-term-spatial} yield
\begin{equation*}
\begin{aligned}
  &\|\sI G_\eta\|_{C_T^{\alpha/2,\omega_\alpha}L^\infty(\rho_{D_1})}
  +\|\sI G_\Psi\|_{C_T^{\alpha/2,\omega_\alpha}L^\infty(\rho_{D_1})} \\
  &\qquad\lesssim
  2^{-(2\alpha-1-\kappa-2\varepsilon)R}(1+U_\alpha^{\rm tot})
  +2^{-(\alpha-2\kappa-3\varepsilon)R_1}.
\end{aligned}
\end{equation*}

  By  \eqref{eq:para-Linfty-left} and \eqref{eq:eta-high-alpha-time}, we have
\begin{equation*}
\begin{aligned}
 \|F(u)\para\eta_R^>\|_{L_T^{\infty,\omega_\alpha}C^{\alpha-2}(\rho_{D_1})}
  &\lesssim
  \|F(u)\|_{L_T^\infty L^\infty}
  \|\eta_R^>\|_{L_T^\infty C^{\alpha-2}(\rho_{s_0})}
  &\lesssim 2^{-(1-\kappa-\alpha)R},
\end{aligned}
\end{equation*}
which by Lemma~\ref{lem:mild-time-increments} implies
\begin{equation*}
  \|\sI(F(u)\para\eta_R^>)\|_{C_T^{\alpha/2,\omega_\alpha}L^\infty(\rho_{D_1})}
  \lesssim 1 .
\end{equation*}
For the product part set
\[
  u_1:=F(u)\para\sI\eta_R^> .
\]
 For \(0<s<t\leqslant T\), decompose
\[
  u_1(t)-u_1(s)
  =\bigl(F(u(t))-F(u(s))\bigr)\para\sI\eta_R^>(t)
   +F(u(s))\para\bigl(\sI\eta_R^>(t)-\sI\eta_R^>(s)\bigr).
\]
We estimate the first term by  \eqref{eq:composition-first-layer-check}, \eqref{eq:eta-high-lambda} and Lemma~\ref{lem:mild-time-increments}:
\[
\begin{aligned}
&s^{\omega_\alpha}
  \frac{\|\bigl(F(u(t))-F(u(s))\bigr)\para\sI\eta_R^>(t)\|_{L^\infty(\rho_{D_1})}}
       {|t-s|^{\alpha/2}} \\
&\qquad\lesssim
  \|F(u)\|_{C_T^{\alpha/2,\omega_\alpha}L^\infty(\rho_{D_\alpha})}
  \|\sI\eta_R^>\|_{L_T^\infty L^\infty(\rho_{-\sigma_\eta})}
  \lesssim
  2^{-(2\alpha-1-\kappa-2\varepsilon)R}U_\alpha^{\rm tot}.
\end{aligned}
\]
The second term is controlled by \eqref{eq:eta-high-alpha-time} and boundedness of \(F\):
\[
\begin{aligned}
&s^{\omega_\alpha}
  \frac{\|F(u(s))\para(\sI\eta_R^>(t)-\sI\eta_R^>(s))\|_{L^\infty(\rho_{D_1})}}
       {|t-s|^{\alpha/2}}
  \lesssim_T
  \|F(u)\|_{L_T^\infty L^\infty}
  \|\sI\eta_R^>\|_{C_T^{\alpha/2,0}L^\infty(\rho_{s_0})}
  \lesssim 1.
\end{aligned}
\]
Here we used
\(
  s_0<D_1,D_\alpha-\sigma_\eta<D_1,
\)
by \eqref{eq:initial-weight-margin} and \eqref{eq:first-level-high-product-margin}.  Combining the two product estimates with \eqref{eq:comm-low-product-detail} for the spatial part gives
\begin{equation}\label{eq:singular-ansatz-bound}
  \|F(u)\para\sI(\eta_R^>)\|_{\mathbb S_T^{\alpha,\omega_\alpha;\alpha/2}(\rho_{D_\alpha})}
  \leqslant C+C2^{-(2\alpha-1-\kappa-2\varepsilon)R}U_\alpha^{\rm tot}.
\end{equation}

Combining  with the  estimates above gives
\begin{equation}\label{eq:u1-high-before-absorb}
  \|u_1^\sharp\|_{\mathbb S_T^{r,\gamma;\alpha/2}(\rho_{D_1})}
  \leqslant C\Bigl(1+2^{-(2\alpha-1-\kappa-2\varepsilon)R}U_\alpha^{\rm tot}
  +2^{-(\alpha-2\kappa-3\varepsilon)R_1}\Bigr),
\end{equation}
and
\begin{equation}\label{eq:u1-alpha-direct}
  \|u_1^\sharp\|_{\mathbb S_T^{\alpha,\omega_\alpha;\alpha/2}(\rho_{D_\alpha})}
  \leqslant C\Bigl(1+2^{-(2\alpha-1-\kappa-2\varepsilon)R}U_\alpha^{\rm tot}
  +2^{-(\alpha-2\kappa-3\varepsilon)R_1}\Bigr).
\end{equation}

\smallskip
\noindent\emph{Closing the ansatz.}
  Finally, by the ansatz and the interpolation estimate \eqref{eq:Ualpha-def},
\begin{equation*}
  U_\alpha^{\rm tot}
  \leqslant
  \|F(u)\para\sI(\eta_R^>)\|_{\mathbb S_T^{\alpha,\omega_\alpha;\alpha/2}(\rho_{D_\alpha})}
  +\|u_1^\sharp\|_{\mathbb S_T^{\alpha,\omega_\alpha;\alpha/2}(\rho_{D_\alpha})}
  +U_\alpha .
\end{equation*}
Using \eqref{eq:u1-alpha-direct} and \eqref{eq:singular-ansatz-bound}  gives
\begin{equation}\label{eq:Utot-ansatz-bound}
  U_\alpha^{\rm tot}
  \leqslant C\Bigl(1+U_\alpha+2^{-(2\alpha-1-\kappa-2\varepsilon)R}U_\alpha^{\rm tot}
  +2^{-(\alpha-2\kappa-3\varepsilon)R_1}\Bigr).
\end{equation}
  Let $C_*$ be the constant multiplying the last two terms in \eqref{eq:Utot-ansatz-bound}.
  Choose deterministic  $R_\eta,R_\Psi\geqslant1$ so large that
\begin{equation}\label{eq:fixed-floor-cutoffs}
  C_*2^{-(2\alpha-1-\kappa-2\varepsilon)R_\eta}\leqslant\frac14,
  \qquad
  C_*2^{-(\alpha-2\kappa-3\varepsilon)R_\Psi}\leqslant\frac14.
\end{equation}
  For all $R\geqslant R_\eta$ and $R_1\geqslant R_\Psi$, the first inequality in \eqref{eq:fixed-floor-cutoffs} absorbs the term containing $U_\alpha^{\rm tot}$.  Since $2^{-(\alpha-2\kappa-3\varepsilon)R_1}\leqslant1$, we obtain \eqref{eq:u-alpha}.
Substituting \eqref{eq:u-alpha} into \eqref{eq:u1-high-before-absorb} gives
 \eqref{eq:u1-r}.
\end{proof}

\section{Weighted transport representation}\label{sec:transport-representation}

We next derive a weighted version of the transport representation and estimate different components in different regularity spaces.
We follow the notation of the finite-volume transport representation \cite[Lemma~4.1]{SZZ24}.

We use the notation
\begin{equation*}
  u_1:=F(u)\para\sI(\eta_R^>),
  \qquad u=u_1+u^\sharp_R,
  \qquad u^\sharp_R=u_1^\sharp+w .
\end{equation*}
Let \(K_i\) be the kernel of \(\Delta_i\) and let \(K_{<i-1}\) be the kernel of \(S_{i-1}=\sum_{j<i-1}\Delta_j\).  As in \cite{SZZ24}, write
\begin{equation*}
  \cR(u)=\sum_{i\geqslant-1}(\Delta_i F(u)- S_{i-1} F'(u)\Delta_iu)
  =\sum_{i\geqslant-1}v_i,
\end{equation*}
where
\begin{equation*}
\begin{aligned}
 v_i(x)=\iint K_i(x-y)K_{<i-1}(x-z)
 \bigl[&F(u(y))-F(u(z))\\
 &-F'(u(z))(u(y)-u(z))\bigr] \,\mathrm{d}y\,\mathrm{d}z .
\end{aligned}
\end{equation*}
Put \(\delta_{yz}f=f(y)-f(z)\).  Then
\begin{equation*}
  v_i=I_{1,i}+I_{2,i},
\end{equation*}
with
\begin{equation*}
\begin{aligned}
I_{1,i}(x)=\iint K_i(x-y)K_{<i-1}(x-z)&
\left[\int_0^1F'(u(z)+\tau\delta_{yz}u)\,\mathrm{d}\tau-F'(u(z))\right]\\
&\times\delta_{yz}u^\sharp_R\,\mathrm{d}y\,\mathrm{d}z,
\end{aligned}
\end{equation*}
\begin{equation}\label{eq:I2-def}
\begin{aligned}
I_{2,i}(x)=\iint K_i(x-y)K_{<i-1}(x-z)&
\int_0^1\!\int_0^1
\tau F''(u(z)+\tau r\delta_{yz}u)\,\mathrm{d}r\,\mathrm{d}\tau\\
&\times \delta_{yz}u\,\delta_{yz}u_1\,\mathrm{d}y\,\mathrm{d}z .
\end{aligned}
\end{equation}

Next split \(I_{1,i}\) according to
\begin{equation*}
  I_{1,i}=I_{111,i}+I_{112,i}+I_{12,i},
\end{equation*}
where
\begin{equation}\label{eq:I111-def}
\begin{aligned}
I_{111,i}(x)=\iint K_i(x-y)K_{<i-1}(x-z)&
\left[\int_0^1F'(u(z)+\tau\delta_{yz}u)\,\mathrm{d}\tau-F'(u(z))\right]\\
&\times\bigl(u_1^\sharp(y)-u_1^\sharp(z)-\nabla u_1^\sharp(z)\cdot(y-z)\bigr)\,\mathrm{d}y\,\mathrm{d}z,
\end{aligned}
\end{equation}
\begin{equation}\label{eq:I112-def}
\begin{aligned}
I_{112,i}(x)=\iint K_i(x-y)K_{<i-1}(x-z)&
\left[\int_0^1F'(u(z)+\tau\delta_{yz}u)\,\mathrm{d}\tau-F'(u(z))\right]\\
&\times\bigl(w(y)-w(z)-\nabla w(z)\cdot(y-z)\bigr)\,\mathrm{d}y\,\mathrm{d}z,
\end{aligned}
\end{equation}
\begin{equation*}
\begin{aligned}
I_{12,i}(x)=\iint K_i(x-y)K_{<i-1}(x-z)&
\left[\int_0^1F'(u(z)+\tau\delta_{yz}u)\,\mathrm{d}\tau-F'(u(z))\right]\\
&\times (y-z)\cdot \nabla u^\sharp_R(z)\,\mathrm{d}y\,\mathrm{d}z .
\end{aligned}
\end{equation*}
These are the same four groups \(I_2,I_{111},I_{112},I_{12}\) as in   \cite[Lemma~4.1]{SZZ24}.

We now define the transport coefficient:
\begin{equation*}
  G_{1,i}(x)
  :=\iint K_i(x-y)K_{<i-1}(x-z)
      A_1(y,z)(y-z)\,\mathrm{d}y\,\mathrm{d}z,
\end{equation*}
\begin{equation*}
  G_{2,i}(x)
  :=\iint K_i(x-y)K_{<i-1}(x-z)
      A_2(y,z)(y-z)\,\mathrm{d}y\,\mathrm{d}z,
\end{equation*}
where
\begin{align}
A_1(y,z)
&:=\int_0^1
\Bigl(
F'\bigl(u(z)+\tau\delta_{yz}u\bigr)
-F'\bigl(u(z)+\tau\delta_{yz}w\bigr)
\Bigr)\,\mathrm{d}\tau, \notag\\
A_2(y,z)
&:=\int_0^1
\Bigl(
F'\bigl(u(z)+\tau\delta_{yz}w\bigr)-F'(u(z))
\Bigr)\,\mathrm{d}\tau .
\label{eq:A2-def}
\end{align}
Finally set
\begin{equation*}
  G_1:=\sum_{i\geqslant-1}G_{1,i},
  \qquad
  G_2:=\sum_{i\geqslant-1}G_{2,i},
  \qquad
  G(u):=G_1+G_2.
\end{equation*}
The transport error is defined by
\begin{equation*}
  \cE_R
  :=\sum_i \Big(I_{2,i}\res\eta_R^>+I_{111,i}\res\eta_R^>+I_{112,i}\res\eta_R^>\Big)+\cE_R^{\rm fr},
\end{equation*}
where
\begin{equation}\label{eq:E-freeze-def}
\begin{aligned}
\cE_R^{\rm fr}(x)
 :=\sum_{i}\sum_{k\sim \ell,\,k\sim i}
 \Bigl(\Delta_k I_{12,i}(x)
       -\nabla u^\sharp_R(x)\cdot\Delta_k(G_{1,i}+G_{2,i})(x)\Bigr)
 \Delta_\ell\eta_R^>(x).
\end{aligned}
\end{equation}
Equivalently, using the kernels, the bracket in \eqref{eq:E-freeze-def} is
\begin{equation}\label{eq:E-freeze-kernel}
\begin{aligned}
&\int K_k(x-x')\iint K_i(x'-y)K_{<i-1}(x'-z)
      \bigl(A_1(y,z)+A_2(y,z)\bigr) \\
&\qquad\qquad\qquad\qquad\times (y-z)\cdot
      \bigl(\nabla u^\sharp_R(z)-\nabla u^\sharp_R(x)\bigr)
      \,\mathrm{d}y\,\mathrm{d}z\,\mathrm{d}x'.
\end{aligned}
\end{equation}
With these definitions the transport representation is
\begin{equation}\label{eq:transport-identity}
\begin{aligned}
  \cR(u)\res\eta_R^>
  &=B_R(u)\cdot\nabla w+B_R(u)\cdot\nabla u_1^\sharp+
  \cE_R,
  \\
   B_R(u)&:=G(u)\res\eta_R^>.
\end{aligned}
\end{equation}

Now we estimate each term above.

Recall all the parameters in Definition~\ref{def:kappa-admissible-weights},
as well as $r,\gamma$ in \eqref{eq:s-r-gamma}.
We first introduce the following weight parameter
 \[
  D_\varepsilon:=\left(1-\frac{1+\varepsilon}{s}\right)D_L+
  \frac{1+\varepsilon}{s}D .
\]
 By Lemma \ref{lem:interp},
\begin{equation}\label{eq:K-def}
  \norm{w}_{C_T^{1+\varepsilon,\gamma \frac{1+\varepsilon}{s}}(\rho_{D_\varepsilon})}
  \lesssim L^{1-\frac{1+\varepsilon}{s}}S^{\frac{1+\varepsilon}{s}}=:K.
\end{equation}

\begin{lemma}\label{lem:G-bounds}
It holds that, uniformly in $t\in [0,T]$,
\begin{equation}\label{eq:G1-bound-detailed}
  \norm{G_1(t)}_{C^{2-\kappa}(\rho_{d_0})}
  \leqslant C,
\end{equation}
\begin{equation}\label{eq:G2-bound-detailed}
  \norm{G_2(t)}_{C^{1+(\kappa+5\varepsilon)}(\rho_{(\kappa+5\varepsilon) D_\varepsilon})}
  \leqslant C\norm{w(t)}_{C^{1+\varepsilon}(\rho_{D_\varepsilon})}^{\kappa+5\varepsilon}.
\end{equation}
Consequently, with
\begin{equation*}
  D_{\rm dr}:=(\kappa+5\varepsilon) D_\varepsilon+s_0,
  \qquad
  \omega_B:=\gamma(\kappa+5\varepsilon) \frac{1+\varepsilon}{s},
\end{equation*}
we have
\begin{equation}\label{eq:B-time-bound}
  \sup_{0<t\leqslant T}t^{\omega_B}
  \norm{B_R(u)(t)}_{C^\varepsilon(\rho_{D_{\rm dr}})}
  \leqslant C(1+K^{\kappa+5\varepsilon}).
\end{equation}
\end{lemma}

\begin{proof}
For \(G_{1,i}\), we have
\[
  |A_1(y,z)|
  \lesssim |\delta_{yz}(u_1+u_1^\sharp)|.
\]
It follows that \(u_1=F(u)\para\sI(\eta_R^>)\in C^{1-\kappa}(\rho_{s_0})\), while \(s_0<d_0\) and \eqref{eq:u1-low} give
\[
  \norm{u_1+u_1^\sharp}_{C^{1-\kappa}(\rho_{d_0})}
  \leqslant C.
\]
 Since the kernels are localized at scale \(2^{-i}\),
\begin{equation}\label{eq:G1i-dyadic}
  \norm{G_{1,i}}_{L^\infty(\rho_{d_0})}
  \lesssim
  2^{(-2+\kappa)i}.
\end{equation}
The Fourier support of \(G_{1,i}\) is contained in a fixed annulus of size \(2^i\).  Hence \eqref{eq:G1i-dyadic} implies \eqref{eq:G1-bound-detailed}.

For \(G_{2,i}\),  \(A_2\) in \eqref{eq:A2-def} satisfies, by boundedness of \(F'\) and \(F''\),
\[
  |A_2(y,z)|\lesssim \min\{|\delta_{yz}w|,1\}
  \lesssim |\delta_{yz}w|^{\kappa+5\varepsilon}
  \lesssim
  \norm{w}_{C^{1+\varepsilon}(\rho_{D_\varepsilon})}^{\kappa+5\varepsilon}\sup_{\theta\in[0,1]}
  \rho_{(\kappa+5\varepsilon)D_\varepsilon}(z+\theta(y-z))^{-1}|y-z|^{\kappa+5\varepsilon}.
\]
 For a fixed large \(N\),
\[
  \rho_{D_\varepsilon}(z+\theta(y-z))^{-1}
  \lesssim
  \rho_{D_\varepsilon}(x)^{-1}
  \la (x-y)\ra^N\la (x-z)\ra^N .
\]
The two polynomial factors are absorbed by the finite moments of the kernels.
The extra factor \((y-z)\) gives
\begin{equation*}
  \norm{G_{2,i}}_{L^\infty(\rho_{(\kappa+5\varepsilon) D_\varepsilon})}
  \lesssim
  2^{(-1-(\kappa+5\varepsilon))i}
  \norm{w}_{C^{1+\varepsilon}(\rho_{D_\varepsilon})}^{\kappa+5\varepsilon},
\end{equation*}
which gives \eqref{eq:G2-bound-detailed}.

Moreover, we have
\begin{equation}\label{eq:G1-weight-margin}
  d_0+s_0<D_{\rm dr},
  \qquad D_{\rm dr}+D_\varepsilon<D .
\end{equation}
Indeed,
\[
  \frac{D_\varepsilon}{D_L}
  =1-\frac{1+\varepsilon}{s}+\frac{1+\varepsilon}{s\ell},
\]
so \(D_\varepsilon\geqslant c_\kappa D_L\) uniformly for the fixed admissible parameters, whereas \(d_0+s_0=O(\mu_{\rm wt})=o(D_L)\).  This gives \(d_0+s_0<D_{\rm dr}\).  For the second inequality, divide by \(D\) and use \(D_L=\ell D\): up to the perturbative term \(s_0/D=o(1)\), it is enough to have
\[
  \bigl(1+(\kappa+5\varepsilon)\bigr)
  \left(\left(1-\frac{1+\varepsilon}{s}\right)\ell+\frac{1+\varepsilon}{s}\right)<1.
\]
This is exactly \(\ell<\ell_+\) by \eqref{eq:theta-ell-def}; the strict separation \eqref{eq:ell-window} and a further reduction of the uniform weight bound absorb the term \(s_0\).
Thus the assertion follows from Lemma \ref{lem:para}.
\end{proof}

\begin{proposition}\label{prop:transport-error}
With $R,R_1$ as in Proposition~\ref{prop:first-layer}, it holds that
\begin{equation}\label{eq:transport-error-bound}
\begin{aligned}
  \norm{\cE_R}_{L_T^{\infty,\gamma}L^\infty(\rho_{D_L})}
  \leqslant C\bigl(&
  1+2^{-(2\alpha-1-\kappa-2\varepsilon)R}U_\alpha
 +2^{-(1-\kappa-\varepsilon)R}S\bigr).
\end{aligned}
\end{equation}
\end{proposition}

\begin{proof}

\smallskip
\noindent\emph{The \(I_2\)-part.}
 In \eqref{eq:I2-def}, the variables \(y\) and \(z\) are localized near \(x\), up to rapidly decaying tails.  Thus, for a fixed large \(N\), the weighted increment bounds can be written schematically as
\[
\begin{aligned}
  |\delta_{yz}u(t)|
  &\lesssim \|u(t)\|_{C^\alpha(\rho_{D_\alpha})}
     \rho_{D_\alpha}(x)^{-1}|y-z|^\alpha
     \la (x-y)\ra^N\la (x-z)\ra^N,\\
  |\delta_{yz}u_1(t)|
  &\lesssim \|u_1(t)\|_{C^\alpha(\rho_{s_0})}
     \rho_{s_0}(x)^{-1}|y-z|^\alpha
     \la (x-y)\ra^N\la (x-z)\ra^N .
\end{aligned}
\]
The polynomial factors are absorbed by the finite moments of \(K_i\) and \(K_{<i-1}\).
while the two \(\alpha\)-increments give the factor \(2^{-2\alpha i}\).  Using  \(\|u_1\|_{L_T^\infty C^\alpha(\rho_{s_0})}\lesssim1\), we get
\begin{equation*}
\begin{aligned}
  \|I_{2,i}(t)\|_{L^\infty(\rho_{D_\alpha+s_0})}
  &\lesssim 2^{-2\alpha i}
    \|u(t)\|_{C^\alpha(\rho_{D_\alpha})}
    \|u_1(t)\|_{C^\alpha(\rho_{s_0})}  \\
  &\lesssim 2^{-2\alpha i}\bigl(1+\|u(t)\|_{C^\alpha(\rho_{D_\alpha})}\bigr).
\end{aligned}
\end{equation*}
Using \eqref{eq:eta-high-lambda} and \eqref{eq:res-est},we obtain
\begin{equation*}
\begin{aligned}
  t^\gamma\|\sum_iI_{2,i}\circ \eta_R^>\|_{L^\infty(\rho_{D_\alpha+s_0-\sigma_\eta})}
  &\lesssim
  2^{-(2\alpha-1-\kappa-2\varepsilon)R}\bigl(1+U_\alpha\bigr).
\end{aligned}
\end{equation*}
By \eqref{eq:first-level-high-product-margin} we get \(D_\alpha-\sigma_\eta<D_1\), and  \eqref{eq:initial-weight-margin} gives \(D_1+s_0<D_L\).  Therefore

\begin{equation}\label{eq:E2-NL-NS-detailed}
  \|\sum_iI_{2,i}\circ \eta_R^>\|_{L_T^{\infty,\gamma}L^\infty(\rho_{D_L})}
  \lesssim 1+2^{-(2\alpha-1-\kappa-2\varepsilon)R}U_\alpha .
\end{equation}

\smallskip
\noindent\emph{The Taylor remainder \(I_{111,i}\).}
For \(I_{111,i}\) in \eqref{eq:I111-def}, the bracket
\(u_1^\sharp(y)-u_1^\sharp(z)-\nabla u_1^\sharp(z)\cdot(y-z)\)
is controlled by the \(C^r(\rho_{D_1})\)-norm of \(u_1^\sharp\).  More explicitly,
\[
\begin{aligned}
&u_1^\sharp(y)-u_1^\sharp(z)-\nabla u_1^\sharp(z)\cdot(y-z) \\
&\qquad =\int_0^1
   \bigl(\nabla u_1^\sharp(z+\theta(y-z))-\nabla u_1^\sharp(z)\bigr)\cdot(y-z)\,\mathrm{d}\theta .
\end{aligned}
\]
The weighted H\"older characterization therefore gives, for some fixed large \(N\),
\[
\begin{aligned}
&\rho_{D_1}(x)
 \bigl|u_1^\sharp(y)-u_1^\sharp(z)-\nabla u_1^\sharp(z)\cdot(y-z)\bigr| \\
&\qquad\lesssim
 \|u_1^\sharp\|_{C^r(\rho_{D_1})}|y-z|^r
 \la (x-y)\ra^N\la (x-z)\ra^N .
\end{aligned}
\]
 Hence
\[
\begin{aligned}
\rho_{D_1}(x)|I_{111,i}(t,x)|
&\lesssim \|u_1^\sharp(t)\|_{C^r(\rho_{D_1})}
\iint |K_i(x-y)|\,|K_{<i-1}(x-z)|\,|y-z|^r \\
&\qquad\qquad\times
 \la (x-y)\ra^N\la (x-z)\ra^N\,\mathrm{d}y\,\mathrm{d}z \\
&\lesssim 2^{-ri}\|u_1^\sharp(t)\|_{C^r(\rho_{D_1})}.
\end{aligned}
\]
Since \(I_{111,i}\) has Fourier support in a fixed annulus of size \(2^i\), this is exactly the dyadic \(C^r(\rho_{D_1})\)-bound.
 Therefore, by \eqref{eq:eta-high-lambda}, \eqref{eq:res-est} and Proposition~\ref{prop:first-layer}, we have
\begin{equation}\label{eq:E111-detailed}
\begin{aligned}
  t^\gamma\|\sum_iI_{111,i}\circ \eta_R^>\|_{L^\infty(\rho_{D_1+s_0})}
  &\lesssim
  t^\gamma\|u_1^\sharp(t)\|_{C^r(\rho_{D_1})}\|
  \eta_R^>\|_{C^{-1-\kappa}(\rho_{s_0})}  \\
  &\lesssim
  1+2^{-(2\alpha-1-\kappa-2\varepsilon)R}\bigl(1+U_\alpha\bigr)+2^{-(\alpha-2\kappa-3\varepsilon)R_1}.
\end{aligned}
\end{equation}
Here we used \(r-1-\kappa=7\varepsilon>0\) and \(D_1+s_0<D_L\) by \eqref{eq:initial-weight-margin}.

\smallskip
\noindent\emph{The remainder Taylor term \(I_{112,i}\).}
For \(I_{112,i}\) in \eqref{eq:I112-def}, arguing as in the estimate for \(I_{111,i}\) gives
\begin{equation*}
  \|I_{112,i}(t)\|_{L^\infty(\rho_D)}
  \lesssim 2^{-si}\|w(t)\|_{C^s(\rho_D)}.
\end{equation*}
By \eqref{eq:eta-high-alpha-time} and
\eqref{eq:res-est}, we have
\begin{equation}\label{eq:E112-detailed}
  t^\gamma\|\sum_iI_{112,i}\circ \eta_R^>(t)\|_{L^\infty(\rho_{D-\sigma_{\eta,\varepsilon}})}
  \lesssim 2^{-(1-\kappa-\varepsilon)R}S,
\end{equation}
where we used
\begin{equation}\label{eq:transport-error-weight-margin}
  D-\sigma_{\eta,\varepsilon}<D_L .
\end{equation}
Indeed, from the midpoint choice of \(\sigma_\eta\) and the lower endpoint in \eqref{eq:sigma-eta-endpoints}, with \(A=2\alpha-1-\kappa-2\varepsilon\) and \(B=1-\kappa-\varepsilon\),
\[
  \sigma_\eta>\frac{A}{B}(D-D_L+s_0)-s_0.
\]
Therefore
\[
  \frac{B}{A}(s_0+\sigma_\eta)-s_0>D-D_L.
\]
The left-hand side is \(\sigma_{\eta,\varepsilon}\), so \eqref{eq:transport-error-weight-margin} follows.

\smallskip
\noindent\emph{The freezing-gradient error.}
 We split
\[
  \nabla u^\sharp_R(z)-\nabla u^\sharp_R(x)
  =\bigl(\nabla u_1^\sharp(z)-\nabla u_1^\sharp(x)\bigr)
   +\bigl(\nabla w(z)-\nabla w(x)\bigr).
\]
We denote by $\cE_R^{{\rm fr},u_1^\sharp}$ and $\cE_R^{{\rm fr},w}$ the corresponding two parts of $\cE_R^{\rm fr}$ generated by this decomposition.
  By \eqref{eq:E-freeze-kernel}, boundedness of \(F'\), \(F''\), and  \(k\sim\ell\sim i\), the $L^\infty(\rho_{D_1+s_0})$ norm of the part with \(u_1^\sharp\) is bounded by
\begin{equation*}
\begin{aligned}
  2^{-(r-1-\kappa)i}
  \|u_1^\sharp\|_{C^r(\rho_{D_1})}
  \|\eta_R^>\|_{C^{-1-\kappa}(\rho_{s_0})} .
\end{aligned}
\end{equation*}
The factor \(2^{-(r-1-\kappa)i}\) is summable since \(r-1-\kappa=7\varepsilon>0\).  After multiplying by \(t^\gamma\), using Proposition~\ref{prop:first-layer}, \eqref{eq:eta-high-lambda}, and \(D_1+s_0<D_L\) from \eqref{eq:initial-weight-margin}, this gives
\begin{equation*}
  \|\cE_R^{{\rm fr},u_1^\sharp}\|_{L_T^{\infty,\gamma}L^\infty(\rho_{D_L})}
  \lesssim 1+2^{-(2\alpha-1-\kappa-2\varepsilon)R}\bigl(1+U_\alpha\bigr)+2^{-(\alpha-2\kappa-3\varepsilon)R_1}.
\end{equation*}

For the $L^\infty(\rho_{D-\sigma_{\eta,\varepsilon}})$ norm of the part with \(w\), the same kernel identity \eqref{eq:E-freeze-kernel} gives it bounded by
\begin{equation*}
\begin{aligned}
  2^{-(s-2+\varepsilon)i}
  \|w\|_{C^s(\rho_D)}
  \|\eta_R^>\|_{C^{-2+\varepsilon}(\rho_{-\sigma_{\eta,\varepsilon}})} .
\end{aligned}
\end{equation*}
Here \(s-2+\varepsilon=\varepsilon/2>0\), so the dyadic sum is summable.  Using  \eqref{eq:eta-high-alpha-time} and \eqref{eq:transport-error-weight-margin}, we obtain
\begin{equation*}
  \|\cE_R^{{\rm fr},w}\|_{L_T^{\infty,\gamma}L^\infty(\rho_{D_L})}
  \lesssim 2^{-(1-\kappa-\varepsilon)R}S.
\end{equation*}
\smallskip
Adding \eqref{eq:E2-NL-NS-detailed}, \eqref{eq:E111-detailed}, \eqref{eq:E112-detailed}, and the two freezing estimates proves \eqref{eq:transport-error-bound}.
\end{proof}

 We now record the weighted parabolic estimate used to control equations with a time-singular transport drift.

\begin{lemma}[Weighted maximum principle with time-singular drift]\label{lem:max}
Let $v$ solve
\begin{equation*}
  \cL v=B\cdot\nabla v+f,
  \qquad v(0)=0.
\end{equation*}
Assume
\begin{equation*}
  A_B:=\sup_{0<t\leqslant T}t^{\omega_B}
  \norm{B(t)}_{L^\infty(\rho_{D_B})}<\infty,
  \qquad 0\leqslant\omega_B<1,\qquad 0<\gamma<1,
  \qquad 0\leqslant D_B<1.
\end{equation*}
Then, one has
\begin{equation*}
  \norm{v}_{L_T^\infty L^\infty(\rho_{D_L})}
  \leqslant C(1+A_B)^{\frac{D_L}{1-D_B}}
  \norm{f}_{L_T^{\infty,\gamma}L^\infty(\rho_{D_L})}.
\end{equation*}
\end{lemma}

\begin{proof}
As in \cite{SZZ24} we apply a probabilistic method, with reversed time.
		For a space-time function $f$, we set
		$$
		f^T(t,x):=f(T-t,x)\;.
		$$
		It follows that $v^T(t,x)=v(T-t,x)$ solves the following backward equation:
		\begin{align*}
			\partial_t v^T+\Delta v^T+B^T\cdot\nabla v^T+f^T=0\;,
		\end{align*}
		with  the final condition
		\begin{align*}
			v^T(T,x)=v(0,x)=0\;.
		\end{align*}
		 For each $(t,x)\in[0,T]\times\mathbb{R}^d$,
		it is well known that the following SDE has a (probabilistically) weak solution starting from $x$ at time $t$
		$$
		X^{t,x}_s=x+\sqrt{2}(W_s-W_t)+\int^s_tB^T(r,X^{t,x}_r)\,\mathrm{d}r\;,\ \ \forall s\in[t,T]\;,
		$$
		where $W$ is a $d$-dimensional Brownian motion on some stochastic basis $(\Omega',\mathcal{F}',\mathbb{P})$.  By the same argument as in \cite{SZZ24}, we obtain
\begin{equation}\label{eq:fk-representation}
  v(T-t,x)=\mathbb E\int_t^T f(T-s,X_s^{t,x})\,\mathrm{d}s .
\end{equation}
Set
\[
  Y_s:=\sup_{t\leqslant r\leqslant s}\langle X_r^{t,x}\rangle,
  \qquad W_s^*:=\sup_{t\leqslant r\leqslant s}|W_r-W_t|.
\]
From the weighted drift bound,
\[
  Y_s\leqslant C(\langle x\rangle+W_s^*)
  +CA_B\int_t^s(T-r)^{-\omega_B}Y_r^{D_B}\,\mathrm{d}r .
\]
With \(B_T^0:=C(\langle x\rangle+W_T^*)\geqslant1\) and
\[
  U_s:=B_T^0+CA_B\int_t^s(T-r)^{-\omega_B}Y_r^{D_B}\,\mathrm{d}r,
\]
one has \(Y_s\leqslant U_s\) and
\[
  U_s'\leqslant CA_B(T-s)^{-\omega_B}U_s^{D_B}.
\]
Therefore
\[
  U_s^{1-D_B}\leqslant (B_T^0)^{1-D_B}
  +CA_B\int_t^s(T-r)^{-\omega_B}\,\mathrm{d}r
  \leqslant (B_T^0)^{1-D_B}+C_TA_B,
\]
and hence
\begin{equation}\label{eq:Y-growth}
  Y_s\leqslant C\bigl(\langle x\rangle+W_T^*+(1+A_B)^{1/(1-D_B)}\bigr).
\end{equation}
Using \eqref{eq:fk-representation}, \eqref{eq:Y-growth}, and finite polynomial moments of \(W_T^*\),
\[
\begin{aligned}
  \rho_{D_L}(x)|v(T-t,x)|
  &\leqslant C\|f\|_{L_T^{\infty,\gamma}L^\infty(\rho_{D_L})}
  \int_t^T(T-s)^{-\gamma}
  \rho_{D_L}(x)\mathbb E\rho_{D_L}(X^{t,x}_s)^{-1}\,\mathrm{d}s\\
  &\leqslant C_T(1+A_B)^{D_L/(1-D_B)}
  \|f\|_{L_T^{\infty,\gamma}L^\infty(\rho_{D_L})}.
\end{aligned}
\]
\end{proof}

\section{Estimates of the remainder terms and the a priori bound}\label{sec:closure}

In this section we close the  estimate of the remainder terms and prove Theorem~\ref{thm:apriori} and Corollary~\ref{cor:apriori-low-w}.

We recall the four scalar quantities that will be used throughout the section.  The two basic sizes of the final remainder, already defined in \eqref{eq:L-S-def}, are
\[
  L=\norm{w}_{L_T^\infty L^\infty(\rho_{D_L})},\qquad
  S=\norm{w}_{\mathbb S_T^{s,\gamma;s/2}(\rho_D)} .
\]
The two interpolation quantities derived from \eqref{eq:Ualpha-def} and \eqref{eq:K-def} are
\[
   \norm{w}_{\mathbb S_T^{\alpha,\gamma\alpha/s;\alpha/2}(\rho_{D_\alpha})}
  \lesssim U_\alpha
  =L^{1-\alpha/s}S^{\alpha/s},
  \qquad
  \norm{w}_{C_T^{1+\varepsilon,\gamma\frac{1+\varepsilon}{s}}(\rho_{D_\varepsilon})}\lesssim K=L^{1-\frac{1+\varepsilon}{s}}S^{\frac{1+\varepsilon}{s}}.
\]

In this section, we concentrate on the equation for $w$.  We first expand the right-hand side of \eqref{eq:w-eq-pre}.
As in \cite[Eq.~(3.3)]{SZZ24}, we obtain
\begin{equation}\label{eq:pc-resonant-expanded}
\begin{aligned}
 &(F'(u)\para u)\res\eta_R^>-H(u)\para\Psi_{R_1}^> \\
&=\underbrace{(F'(u)\para u^\sharp_R)\res\eta_R^>}_{\mathcal Q_R^{u^\sharp_R}}
  +\underbrace{\mathrm{Com}(F'(u),u_1,\eta_R^>)}_{\mathcal Q_R^{\rm com,1}}
  +\underbrace{F'(u)\,\mathrm{Com}(F(u),\sI(\eta_R^>),\eta_R^>)}_{\mathcal Q_R^{\rm com,2}} \\
&\quad +\underbrace{H(u)\Theta_{R,R_1}}_{\mathcal Q_R^{\Theta}}
  +\underbrace{H(u)\res\Psi_{R_1}^>+H(u)\Par\Psi_{R_1}^>}_{\mathcal Q_R^{\Psi,>}},
\end{aligned}
\end{equation}
where we used \eqref{eq:inherited-high-second}
\[
  H(u)\Psi_R^{>>}-H(u)\para\Psi_{R_1}^>
  =H(u)\Theta_{R,R_1}+H(u)\res\Psi_{R_1}^>+H(u)\Par\Psi_{R_1}^>,
\]

Using \eqref{eq:transport-identity} and \eqref{eq:pc-resonant-expanded}, the final equation is
\begin{equation}\label{eq:w-transport}
  \cL w=B_R(u)\cdot\nabla w+f_R,
  \qquad w(0)=0,
\end{equation}
where
\begin{equation*}
\begin{aligned}
  f_R={}&\cE_R
  +F(u)\eta_R^\leqslant
  +B_R(u)\cdot\nabla u_1^\sharp
  +\mathcal Q_R^{u^\sharp_R}
  +\mathcal Q_R^{\rm com,1}
  +\mathcal Q_R^{\rm com,2}
  +\mathcal Q_R^\Theta
  +\mathcal Q_R^{\Psi,>} .
\end{aligned}
\end{equation*}

\smallskip
 We first estimate the non-transport forcing term $f_R$ in
\(L_T^{\infty,\gamma}L^\infty(\rho_{D_L})\) and prove Proposition~\ref{prop:catalogue}.  This is enough for the Schauder step as well, because
\(L^\infty(\rho_{D_L})\hookrightarrow C^{-\varepsilon/2}(\rho_D)\) when \(D_L<D\), see Proposition~\ref{prop:schauder-step}.  The only term not included in this norm is the transport product \(B_R(u)\cdot\nabla w\): it is treated as a drift in the maximum-principle estimate and is estimated separately in the Schauder estimate.

 The analytic estimate is first proved for fixed cutoffs $R,R_1$ (see Proposition~\ref{prop:catalogue}).  The actual choices of the cutoffs are made later as functions of $S$ (see the proof of Proposition~\ref{prop:schauder-step}); this converts every localization factor, such as $2^{-(2\alpha-1-\kappa-2\varepsilon)R}$ or $2^{(2\kappa+\varepsilon)R_1}$,
into a monomial in the two unknown sizes $L$ and $S$.  After this reduction,
the proof reduces to scalar inequalities: one first bounds $L$ by the maximum
principle, then inserts this bound into the Schauder inequality, and finally checks that every
remaining power of $S$ is strictly less than one.

\smallskip

We shall use the abbreviation
\begin{equation}\label{eq:Ta-Tb-def}
  T_\eta:=2^{-(2\alpha-1-\kappa-2\varepsilon)R}U_\alpha,
\end{equation}
and
\begin{equation*}
  \norm{h}_{\mathfrak N}:=\norm{h}_{L_T^{\infty,\gamma}L^\infty(\rho_{D_L})}.
\end{equation*}
Recall $R_\eta$ and $R_\Psi$ as in
\eqref{eq:fixed-floor-cutoffs}.

\begin{proposition}\label{prop:catalogue}
For every \(R\geqslant R_\eta\) and \(R_1\geqslant R_\Psi\), one has
\begin{equation}\label{eq:term-before-cutoff}
  \norm{f_R}_{\mathfrak N}\leqslant C\mathfrak F_{R,R_1}(L,S).
\end{equation}
Here \begin{equation}\label{eq:raw-forcing-functional}
\begin{aligned}
  \mathfrak F_{R,R_1}(L,S):={}&1+2^{(1+\kappa+\varepsilon)R}+2^{-(1-\kappa-\varepsilon)R}S
  +2^{(2\kappa+\varepsilon)R_1} \\
  &+T_\eta+2^{-(\alpha-2\kappa-3\varepsilon)R_1}U_\alpha
  +K^{\kappa+5\varepsilon}(1+T_\eta^{\theta_0}),
\end{aligned}
\end{equation}
where \begin{equation}\label{eq:theta-theta0}
  \theta_0:=\frac{\kappa+\varepsilon}{r-1+\kappa}.
\end{equation}
\end{proposition}

\begin{proof}
We first have
\begin{equation}\label{eq:cat-low-noise}
\begin{aligned}
  \norm{F(u)\eta_R^\leqslant}_{\mathfrak N}
  &\lesssim 2^{(1+\kappa+\varepsilon)R},
  \qquad
  \norm{\eta_R^\leqslant}_{C^\varepsilon(\rho_{d_\eta})}
  \lesssim2^{(1+\kappa+\varepsilon)R}.
\end{aligned}
\end{equation}
Here we used  $d_\eta+s_0<D_L$ from \eqref{eq:DL}
and  $D_L<D$. 

\smallskip
\noindent\emph{Paracontrolled resonant part.}
Split
\[
  \mathcal Q_R^{u^\sharp_R}
  =(F'(u)\para u_1^\sharp)\res\eta_R^>
   +(F'(u)\para w)\res\eta_R^>
  =:\mathcal Q_R^{u_1^\sharp}+\mathcal Q_R^w .
\]
The first term uses \(u_1^\sharp\in C^r(\rho_{D_1})\) and the rough high noise
\(\eta_R^>\in C^{-1-\kappa}(\rho_{s_0})\).  By the paraproduct estimate
\eqref{eq:para-Linfty-left}, the resonant product estimate \eqref{eq:res-est},
Proposition~\ref{prop:first-layer}, and the high-noise bound \eqref{eq:eta-high-lambda},
\begin{equation}\label{eq:cat-Q-u1}
\begin{aligned}
  \norm{\mathcal Q_R^{u_1^\sharp}}_{\mathfrak N}
  &\lesssim
  \norm{F'(u)\para u_1^\sharp}_{L_T^{\infty,\gamma}C^r(\rho_{D_1})}
  \norm{\eta_R^>}_{L_T^\infty C^{-1-\kappa}(\rho_{s_0})} \\
  &\lesssim
  \norm{u_1^\sharp}_{\mathbb S_T^{r,\gamma;\alpha/2}(\rho_{D_1})}
  \norm{\eta_R^>}_{L_T^\infty C^{-1-\kappa}(\rho_{s_0})} \\
  &\lesssim
  1+2^{-(2\alpha-1-\kappa-2\varepsilon)R}U_\alpha
    +2^{-(\alpha-2\kappa-3\varepsilon)R_1}
  \lesssim 1+T_\eta .
\end{aligned}
\end{equation}
Here we used  $D_1+s_0<D_L$ by \eqref{eq:initial-weight-margin},
$r-1-\kappa=7\varepsilon>0$, and
$2^{-(\alpha-2\kappa-3\varepsilon)R_1}\leqslant 1$ for the admissible cutoffs.
By \eqref{eq:eta-high-alpha-time}, \eqref{eq:para-Linfty-left} for \(F'(u)\para w\),  \eqref{eq:res-est}, the definition of \(S\) in \eqref{eq:L-S-def}, and \(s-2+\varepsilon=\varepsilon/2>0\), we get
\[
\begin{aligned}
  \norm{\mathcal Q_R^w}_{\mathfrak N}
  &\lesssim
  \norm{F'(u)\para w}_{L_T^{\infty,\gamma}C^s(\rho_D)}
  \norm{\eta_R^>}_{L_T^\infty C^{-2+\varepsilon}(\rho_{-\sigma_{\eta,\varepsilon}})}  \\
  &\lesssim 2^{-(1-\kappa-\varepsilon)R}S .
\end{aligned}
\]
Thus
\begin{equation*}
  \norm{\mathcal Q_R^w}_{\mathfrak N}
  \lesssim2^{-(1-\kappa-\varepsilon)R}S,
\end{equation*}
where we used  $D-\sigma_{\eta,\varepsilon}<D_L$ by \eqref{eq:transport-error-weight-margin}.

\smallskip
\noindent\emph{Commutators.}
The two commutators in \eqref{eq:pc-resonant-expanded} are
\[
  \mathcal Q_R^{\rm com,1}=\mathrm{Com}(F'(u),u_1,\eta_R^>),
  \qquad
  \mathcal Q_R^{\rm com,2}=F'(u)\mathrm{Com}(F(u),\sI(\eta_R^>),\eta_R^>).
\]

First, boundedness of \(F\),  \eqref{eq:para-Linfty-left}, Lemma~\ref{lem:mild-time-increments}, and  \eqref{eq:eta-high-lambda} give
\[
  \norm{u_1}_{C^{2-2\alpha+2\varepsilon}(\rho_{-\sigma_\eta})}
  =\norm{F(u)\para\sI\eta_R^>}_{C^{2-2\alpha+2\varepsilon}(\rho_{-\sigma_\eta})}
  \lesssim 2^{-(2\alpha-1-\kappa-2\varepsilon)R}.
\]
By \eqref{eq:u-alpha} we obtain \(F'(u),F(u)\in C^\alpha(\rho_{D_\alpha})\) with bound \(1+U_\alpha\).
Using
\[
  F'(u)\in C^\alpha(\rho_{D_\alpha}),\qquad
  u_1\in C^{2-2\alpha+2\varepsilon}(\rho_{-\sigma_\eta}),\qquad
  \eta_R^>\in C^{-1-\kappa}(\rho_{s_0}),
\]
we have
\[
  \alpha+(2-2\alpha+2\varepsilon)+(-1-\kappa)=r-1-\kappa>0,
  \qquad
  (2-2\alpha+2\varepsilon)+(-1-\kappa)<0.
\]
Hence,  \eqref{eq:weighted-bony-commutator-catalogue},  and \eqref{eq:u-alpha} give
\begin{equation*}
  \norm{\mathcal Q_R^{\rm com,1}}_{\mathfrak N}
  \lesssim 2^{-(2\alpha-1-\kappa-2\varepsilon)R}(1+U_\alpha)\lesssim 1+T_\eta .
\end{equation*}
Here the last inequality uses the definition of \(T_\eta\) in \eqref{eq:Ta-Tb-def}, and the weight placement uses \(D_\alpha-\sigma_\eta+s_0<D_1+s_0<D_L\), by \eqref{eq:first-level-high-product-margin} and \eqref{eq:initial-weight-margin}.
Similarly,
\[
  F(u)\in C^\alpha(\rho_{D_\alpha}),\qquad
  \sI\eta_R^>\in C^{2-2\alpha+2\varepsilon}(\rho_{-\sigma_\eta}),\qquad
  \eta_R^>\in C^{-1-\kappa}(\rho_{s_0}),
\]
and the same regularity and weight conditions allow another application of the commutator estimate \eqref{eq:weighted-bony-commutator-catalogue}.  Using again  \eqref{eq:u-alpha}, \eqref{eq:Ta-Tb-def}, \eqref{eq:first-level-high-product-margin}, and \eqref{eq:initial-weight-margin}, we obtain
\begin{equation}\label{eq:cat-com2}
  \norm{\mathcal Q_R^{\rm com,2}}_{\mathfrak N}
  \lesssim 2^{-(2\alpha-1-\kappa-2\varepsilon)R}(1+U_\alpha)\lesssim 1+T_\eta .
\end{equation}

\smallskip
\noindent\emph{Other terms.}
 By Lemma~\ref{lem:inherited-second-level} we obtain
\begin{equation}\label{eq:cat-second-total}
  \norm{\mathcal Q_R^\Theta}_{\mathfrak N}
  =\norm{H(u)\Theta_{R,R_1}}_{\mathfrak N}
  \lesssim 2^{(1+\kappa+\varepsilon)R}+2^{(2\kappa+\varepsilon)R_1}.
\end{equation}

Using \(H(u)\in C^\alpha(\rho_{D_\alpha})\) and the first estimate in \eqref{eq:Psi-high}, namely \(\Psi_{R_1}^>\in C^{-\alpha+2\varepsilon}(\rho_{-\sigma_\Psi})\) with gain $2^{-(\alpha-2\kappa-3\varepsilon)R_1}$, the resonant part is admissible since
\(\alpha+(-\alpha+2\varepsilon)=2\varepsilon>0\).   Hence
\begin{equation}\label{eq:cat-second-high-remainder}
  \norm{\mathcal Q_R^{\Psi,>}}_{\mathfrak N}
  \lesssim 2^{-(\alpha-2\kappa-3\varepsilon)R_1}(1+U_\alpha)
  \lesssim 1+2^{-(\alpha-2\kappa-3\varepsilon)R_1}U_\alpha .
\end{equation}
Here we used
\begin{equation*}
  D_\alpha-\sigma_\Psi<D_1<D_L.
\end{equation*}
Indeed, \(\sigma_\Psi\) is the midpoint of a non-empty interval and \(\sigma_{\Psi,-}=D_\alpha-D_1\), hence \(\sigma_\Psi>D_\alpha-D_1\); the last inequality follows from \eqref{eq:initial-weight-margin}.

From Lemma \ref{lem:G-bounds},
\begin{equation*}
  \norm{B_R^{}(t)}_{C^\varepsilon(\rho_{D_{\rm dr}})}
  \lesssim t^{-\omega_B}\bigl(1+K^{\kappa+5\varepsilon}\bigr),
  \qquad
  \omega_B=\gamma(\kappa+5\varepsilon)\frac{1+\varepsilon}{s}.
\end{equation*}
Interpolation between \eqref{eq:u1-low} and \eqref{eq:u1-r} gives
\begin{equation*}
\begin{gathered}
  1+\varepsilon=(1-\theta_0)(1-\kappa)+\theta_0r,\\
  \norm{\nabla u_1^\sharp}_{C_T^{\varepsilon,\gamma\theta_0}(\rho_{(1-\theta_0)d_0+\theta_0D_1})}
  \lesssim1+T_\eta^{\theta_0}.
\end{gathered}
\end{equation*}
 Therefore
\begin{align}
  \norm{B_R\cdot\nabla u_1^\sharp}_{\mathfrak N}
  &\lesssim \bigl(1+K^{\kappa+5\varepsilon}\bigr)(1+T_\eta^{\theta_0}),
  \label{eq:cat-B2-u1}
\end{align}
where we used \(0<\theta_0<1\).
For the time singularities and weights here we used
\begin{equation}\label{eq:time-product-margin-a}
  \bigl(1+(\kappa+5\varepsilon)\bigr)\frac{1+\varepsilon}{s}<1,
  \qquad
  (\kappa+5\varepsilon)\frac{1+\varepsilon}{s}+\theta_0<1.
\end{equation}
They follow by continuity from their strict limits at \(\varepsilon=0\):
\[
  (1+\kappa)\frac12<1,
  \qquad \frac\kappa2+\frac12<1.
\]
Here for the weight we also used
$D_{\rm dr}+\theta_0D_1+d_0<D_L $.
Indeed, after dividing by \(D_L\), we obtain
\[
  (\kappa+5\varepsilon)
  \left(1-\frac{1+\varepsilon}{s}+\frac{1+\varepsilon}{s\ell}\right)
  +\theta_0<1 .
\]
Multiplying by \(\ell\), and then setting \(\varepsilon=0\), gives
\[
  \kappa\frac{1+\ell}{2}+\frac\ell2<\ell,
  \qquad\text{equivalently}\qquad
  \ell>\frac{\kappa}{1-\kappa}.
\]
This lower bound is weaker than \(\ell>\theta\) on \(0<\kappa<\sqrt5-2\), because \(\theta\to(1+\kappa)/2\) and \(\kappa/(1-\kappa)<(1+\kappa)/2\).  Hence the case $\varepsilon=0$ leaves a strict margin, allowing us to take \(\varepsilon\) sufficiently small; the perturbative terms \(s_0/D_L\), \((D_L-D_1)/D_L\), and \(d_0/D_L\) vanish as \(\mu_{\rm wt}\downarrow0\).  This is the requirement for the  weight  of the transport term $B_R\cdot\nabla u_1^\sharp$ in \eqref{eq:cat-B2-u1}.

Combining Proposition~\ref{prop:transport-error}, \eqref{eq:cat-low-noise},
\eqref{eq:cat-Q-u1}--\eqref{eq:cat-com2}, \eqref{eq:cat-second-total},
\eqref{eq:cat-second-high-remainder}, and the transport estimate \eqref{eq:cat-B2-u1} gives exactly \eqref{eq:term-before-cutoff}.
\end{proof}

\begin{proposition}\label{prop:max-step}
Let
\begin{equation*}
  p_{\rm mp}:=\frac{D_L}{1-D_{\rm dr}}.
\end{equation*}
Then the solution of \eqref{eq:w-transport} satisfies
\begin{equation}\label{eq:max-before-cutoff}
  L\leqslant C\bigl(1+K^{\kappa+5\varepsilon}\bigr)^{p_{\rm mp}}
  \bigl(1+\norm{f_R}_{\mathfrak N}\bigr).
\end{equation} Here we used $0\leqslant \omega_B=\gamma(\kappa+5\varepsilon)\frac{1+\varepsilon}{s}<1$.
Moreover, taking the cutoff by \eqref{eq:R-dynamic}, \eqref{eq:R1-dynamic} below, we obtain
\begin{equation}\label{eq:term-after-cutoff}
\begin{aligned}
  \norm{f_R}_{\mathfrak N}
  \leqslant C\Big(&1+S^\Theta+L^{A_\alpha}S^{B_\alpha}
  +L^{A_K}S^{B_K}\bigl(1+L^{\theta_0A_\alpha}S^{\theta_0B_\alpha}\bigr)\Big).
\end{aligned}
\end{equation}
Here all the parameters $\Theta, A_\alpha, B_\alpha, A_K, B_K$ are given in the proof.
\end{proposition}

\begin{proof}
First we have \(D_{\rm dr}+\theta_0D_1+d_0<D_L\), proved after \eqref{eq:time-product-margin-a}, gives \(D_{\rm dr}<D_L\).  Moreover, after shrinking the uniform weight bound if necessary, \(D_{\rm dr}<1\).   Applying Lemma~\ref{lem:max} and \eqref{eq:B-time-bound} to \eqref{eq:w-transport} gives \eqref{eq:max-before-cutoff}.

We now choose the cutoffs. First we consider terms containing $R$. Set
\[
  a_R:=1+\kappa+\varepsilon,
  \qquad
  b_R:=1-\kappa-\varepsilon,
  \qquad
  c_R:=2\alpha-1-\kappa-2\varepsilon.
\]
Recall \eqref{eq:raw-forcing-functional} contains
\[
  2^{a_RR},\qquad
  2^{-b_RR}S,\qquad
  T_\eta=2^{-c_RR}U_\alpha.
\] If a maximum-principle inequality contains a monomial
\[
  \lambda L^A S^B,\qquad A<1,
\]
then, after eliminating the $L$-power, this monomial contributes the effective term
\[
  \lambda^{1/(1-A)}S^{B/(1-A)} .
\]
For the term
\(
  T_\eta=2^{-c_RR}U_\alpha,
\)
recall from \eqref{eq:Ualpha-def} that
\(U_\alpha=L^{1-\alpha/s}S^{\alpha/s}\).
Therefore, after eliminating the $L$-power, its effective contribution is
\[
  2^{-(s/\alpha)c_RR}S.
\]
Since the loss ${p_{\rm mp}}$ is small, we do not include it in this preliminary cutoff balance, for notational simplicity.

Thus the first cutoff should balance $2^{a_RR}$ against the worse of the two high-frequency gains
\[
  2^{-b_RR}S,
  \qquad
  2^{-(s/\alpha)c_RR}S.
\]
This leads to
\begin{equation}\label{eq:GammaR-def}
\begin{gathered}
  m_\varepsilon=b_R\wedge\frac{s}{\alpha}c_R,\qquad
  \Gamma_R:=a_R+m_\varepsilon=1+\kappa+\varepsilon+m_\varepsilon,
\end{gathered}
\end{equation}
and
\begin{align}
  2^{\Gamma_RR}=2^{(1+\kappa+\varepsilon+m_\varepsilon)R}&\simeq C_0(1+S).
  \label{eq:R-dynamic}
\end{align}
Here $R_\eta$ is the deterministic lower cutoff fixed in Proposition~\ref{prop:first-layer}, and we choose $C_0$ large enough so that $R\geqslant R_\eta$.  Recall also from \eqref{eq:m-GammaR} that
\[
  \theta=\frac{1+\kappa+\varepsilon}{1+\kappa+\varepsilon+m_\varepsilon}
  =\frac{1+\kappa+\varepsilon}{\Gamma_R}.
\]
  From \eqref{eq:R-dynamic},
\begin{align}
  2^{(1+\kappa+\varepsilon)R}
  &\lesssim 1+S^{(1+\kappa+\varepsilon)/\Gamma_R}=1+S^\theta,
  \label{eq:cut-low-R}\\
  2^{-(1-\kappa-\varepsilon)R}S
  &\lesssim 1+S^{1-(1-\kappa-\varepsilon)/\Gamma_R}
  \leqslant 1+S^{1-m_\varepsilon/\Gamma_R}=1+S^\theta, \notag\\
  T_\eta=2^{-(2\alpha-1-\kappa-2\varepsilon)R}U_\alpha
  &\lesssim 1+L^{1-\alpha/s}S^{\alpha/s-(2\alpha-1-\kappa-2\varepsilon)/\Gamma_R}
  =1+L^{A_\alpha}S^{B_\alpha}, \notag
\end{align}
where
\begin{equation}\label{eq:Aalpha-Balpha}
  A_\alpha:=1-\frac{\alpha}{s},
  \qquad
  B_\alpha:=\frac{\alpha}{s}-\frac{2\alpha-1-\kappa-2\varepsilon}{\Gamma_R}.
\end{equation}

Set \[
  C_\Psi:=2\kappa+\varepsilon,
  \qquad
  A_\Psi:=\alpha-2\kappa-3\varepsilon.
\]
Similarly, the terms involving $R_1$ are
\[
  2^{C_\Psi R_1},\qquad
  2^{-A_\Psi R_1}U_\alpha.
\]
The term
\[
  2^{-A_\Psi R_1}U_\alpha
\]
has the effective cutoff form
\[
  2^{-(s/\alpha)A_\Psi R_1}S.
\]
We do not rebalance $R_1$ with this sharper effective power, and instead use the simpler choice
\begin{align}
  2^{(\alpha+\varepsilon)R_1}&\simeq C_1(1+S).
  \label{eq:R1-dynamic}
\end{align}
Here $R_\Psi$ is the deterministic lower cutoff fixed in Proposition~\ref{prop:first-layer}, and we choose $C_1$ large enough so that $R_1\geqslant R_\Psi$.
Thus, we obtain
\begin{equation*}
  2^{(2\kappa+\varepsilon)R_1}
  \lesssim1+S^{(2\kappa+\varepsilon)/(\alpha+\varepsilon)}
  =1+S^{\theta_\Psi}
  \leqslant 1+S^\Theta.
\end{equation*}
Here \begin{equation}\label{eq:thetaPsi}
  \theta_\Psi:=\frac{2\kappa+\varepsilon}{\alpha+\varepsilon},
\end{equation}
and \begin{equation*}
  \Theta:=\theta\vee\theta_\Psi.
\end{equation*}
 Moreover, $ 2^{-(\alpha-2\kappa-3\varepsilon)R_1}U_\alpha$ is controlled by the same mixed monomial as \(T_\eta\):
\begin{equation*}
  2^{-(\alpha-2\kappa-3\varepsilon)R_1}U_\alpha
  \lesssim1+L^{A_\alpha}S^{B_\alpha}.
\end{equation*}
Indeed, after absorbing the fixed cutoff constant, the left-hand side is bounded by
\[
  L^{1-\alpha/s}S^{\alpha/s}(1+S)^{-(\alpha-2\kappa-3\varepsilon)/(\alpha+\varepsilon)}.
\]
For \(0\leqslant S\leqslant1\), this is bounded by \(1+L^{A_\alpha}S^{B_\alpha}\), since \(0<B_\alpha<\alpha/s\) for the admissible parameters.  For \(S\geqslant1\), it is bounded by the same monomial because
\[
  \frac{\alpha-2\kappa-3\varepsilon}{\alpha+\varepsilon}
  >
  \frac{2\alpha-1-\kappa-2\varepsilon}{\Gamma_R},
\]
which follows from \(\alpha-2\kappa-3\varepsilon>2\alpha-1-\kappa-2\varepsilon\), \(\alpha+\varepsilon<1\), and \(\Gamma_R>1\), after decreasing \(\varepsilon\).

Finally,
\begin{align}
  K^{\kappa+5\varepsilon}&=L^{A_K}S^{B_K}, \notag\\
  K^{\kappa+5\varepsilon} T_\eta^{\theta_0}
  &\lesssim
  L^{A_K}S^{B_K}\bigl(1+L^{\theta_0A_\alpha}S^{\theta_0B_\alpha}\bigr).
  \label{eq:cut-K-Ta}
\end{align}
Here \begin{equation}\label{eq:AK-BK}
  A_K:=(\kappa+5\varepsilon)\left(1-\frac{1+\varepsilon}{s}\right),
  \qquad
  B_K:=(\kappa+5\varepsilon)\frac{1+\varepsilon}{s}.
\end{equation}
Combining \eqref{eq:term-before-cutoff} with \eqref{eq:cut-low-R}--\eqref{eq:cut-K-Ta} gives \eqref{eq:term-after-cutoff}.

\end{proof}

\begin{remark}
The choice of $R_1$ above is not optimized to balance the two  contributions after the maximum-principle elimination of $L$.  Rather, it reduces the estimate before the maximum-principle step while keeping the same finite family of monomials.  A sharper balance will not improve the range of $\kappa$.  The present choice of $R_1$ is therefore a non-optimal but technically convenient cutoff choice.
\end{remark}

\begin{proposition}\label{prop:schauder-step}
For every \(R\geqslant R_\eta\) and \(R_1\geqslant R_\Psi\), one has
\begin{equation}\label{eq:schauder-before-cutoff}
  S\leqslant C\Bigl(\norm{f_R}_{\mathfrak N}
  +\bigl(1+K^{\kappa+5\varepsilon}\bigr)K\Bigr).
\end{equation}
\end{proposition}

\begin{proof}
Applying Lemma~\ref{lem:mild-time-increments} to \eqref{eq:w-transport} gives
\[
  S\leqslant C\Bigl(
  \|B_R(u)\cdot\nabla w\|_{L_T^{\infty,\gamma}C^{-\varepsilon/2}(\rho_D)}
  +\|f_R\|_{L_T^{\infty,\gamma}C^{-\varepsilon/2}(\rho_D)}\Bigr).
\]
Since \(D_L<D\),  the last term is controlled by \(\norm{f_R}_{\mathfrak N}\).  For the transport product, Lemma~\ref{lem:para} and the weight margin \(D_{\rm dr}+D_\varepsilon<D\), already checked in \eqref{eq:G1-weight-margin}, give
\[
  \|B_R(u)(t)\cdot\nabla w(t)\|_{C^{-\varepsilon/2}(\rho_D)}
  \lesssim
  \|B_R(u)(t)\|_{C^\varepsilon(\rho_{D_{\rm dr}})}
  \|w(t)\|_{C^{1+\varepsilon}(\rho_{D_\varepsilon})}.
\]
  Using \eqref{eq:B-time-bound} and \eqref{eq:K-def} yields
\[
  \|B_R(u)\cdot\nabla w\|_{L_T^{\infty,\gamma}C^{-\varepsilon/2}(\rho_D)}
  \lesssim \bigl(1+K^{\kappa+5\varepsilon}\bigr)K.
\]
This proves \eqref{eq:schauder-before-cutoff}. Here we used
\(
  \omega_B=\gamma(\kappa+5\varepsilon)\frac{1+\varepsilon}{s},
\)
and the time condition needed to multiply the two weighted factors is
\[
  \omega_B+\gamma\frac{1+\varepsilon}{s}\leqslant\gamma,
  \qquad\text{i.e.}\qquad
  \bigl(1+(\kappa+5\varepsilon)\bigr)\frac{1+\varepsilon}{s}\leqslant1,
\]
which follows from  \eqref{eq:time-product-margin-a}.
\end{proof}

\begin{theorem}\label{thm:apriori}
It holds that
\begin{equation}\label{eq:apriori-core}
  \norm{w}_{L_T^\infty L^\infty(\rho_{D_L})}
  +\norm{w}_{\mathbb S_T^{s,\gamma;s/2}(\rho_D)}
  +\norm{u_1^\sharp}_{\mathbb S_T^{r,\gamma;\alpha/2}(\rho_{D_1})}
  \leqslant C.
\end{equation}

\end{theorem}

\begin{proof}
Using Proposition~\ref{prop:max-step}, we obtain
\[
\begin{aligned}
  L\leqslant C(1+L^{A_K}S^{B_K})^{p_{\rm mp}}
  \Big(&1+S^\Theta+L^{A_\alpha}S^{B_\alpha}
  +L^{A_K}S^{B_K}
  \bigl(1+L^{\theta_0A_\alpha}S^{\theta_0B_\alpha}\bigr)\Big).
\end{aligned}
\]

At this point the analytic estimates have been reduced to finitely many monomials $L^AS^B$. For \(A<1\),
\begin{equation*}
  X\leqslant C(1+X^AY^B)
  \quad\Longrightarrow\quad
  X\leqslant C(1+Y^{B/(1-A)}).
\end{equation*}
Each monomial $L^AS^B$
 is represented by its exponent pair $e=(A,B)$.
For a  pair $e=(A,B)$ with $A<1$, write
\begin{equation}\label{eq:Pi-def}
  \Pi(e):=\frac{B}{1-A}.
\end{equation}
The exponent pairs occurring in the maximum-principle estimate are collected in the lists
\begin{equation}\label{eq:E0-set}
\begin{aligned}
  \mathscr E_0:=\bigl\{ &(0,0),\ (0,\Theta),\ (A_\alpha,B_\alpha),\ (A_K,B_K),\\
  &(A_K+\theta_0A_\alpha,\ B_K+\theta_0B_\alpha)\bigr\},
\end{aligned}
\end{equation}
\begin{equation}\label{eq:EL-set}
  \mathscr E_L(p):=\mathscr E_0\cup\{e+(pA_K,pB_K):e\in\mathscr E_0\}.
  \qquad \mbox{Set}\;\;
  \Theta_{\rm mp}(p):=\max_{e\in\mathscr E_L(p)}\Pi(e).
\end{equation}
Therefore, by the definitions of \( \mathscr E_L(p)\), \( \Theta_{\rm mp}(p)\), and \( \Pi \) in \eqref{eq:Pi-def}--\eqref{eq:EL-set}, any inequality of the form
\begin{equation*}
  L\leqslant C\sum_j(1+L^{A_j}S^{B_j}),
  \quad (A_j,B_j)\in\mathscr E_L(p_{\rm mp}),
  \quad A_j<1,
\end{equation*}
implies
\begin{equation}\label{eq:scalar-L-output}
  L\leqslant C(1+S^{\Theta_{\rm mp}}),
  \qquad
  \Theta_{\rm mp}=\max_{(A,B)\in\mathscr E_L(p_{\rm mp})}\frac{B}{1-A}.
\end{equation}

Combining \eqref{eq:schauder-before-cutoff} with \eqref{eq:term-after-cutoff} gives
\begin{equation*}
\begin{aligned}
S\leqslant C\Big(&1+S^\Theta
+L^{A_\alpha}S^{B_\alpha}
+L^{A_K}S^{B_K}
+L^{A_K+\theta_0A_\alpha}S^{B_K+\theta_0B_\alpha}\\
&+(1+K^{\kappa+5\varepsilon})K\Big).
\end{aligned}
\end{equation*}
Using \eqref{eq:scalar-L-output}, the  monomials satisfy
\begin{equation*}
\begin{gathered}
  L^{A_\alpha}S^{B_\alpha}\lesssim1+S^{A_\alpha\Theta_{\rm mp}+B_\alpha},\\
  L^{A_K}S^{B_K}\lesssim1+S^{A_K\Theta_{\rm mp}+B_K},
  \qquad K\lesssim1+S^{(1-\frac{1+\varepsilon}{s})\Theta_{\rm mp}+\frac{1+\varepsilon}{s}},\\
  L^{A_K+\theta_0A_\alpha}S^{B_K+\theta_0B_\alpha}
  \lesssim
  1+S^{(A_K+\theta_0A_\alpha)\Theta_{\rm mp}+B_K+\theta_0B_\alpha}.
\end{gathered}
\end{equation*}
So the mixed terms $L^A S^B$
 are bounded by pure powers of $S$. We collect the resulting exponents:
\begin{equation}\label{eq:ES-set}
\begin{aligned}
  \mathscr E_S(\Theta_*):=\bigl\{&\Theta,\;
  A_\alpha\Theta_*+B_\alpha,\;
  A_K\Theta_*+B_K,\\
  &\left(1-\frac{1+\varepsilon}{s}\right)\Theta_*+\frac{1+\varepsilon}{s},\;
  (A_K+\theta_0A_\alpha)\Theta_*+B_K+\theta_0B_\alpha\bigr\}.
\end{aligned}
\end{equation}
By Young's inequality (using \eqref{eq:time-product-margin-a})
\begin{equation*}
\begin{aligned}
  K^{1+(\kappa+5\varepsilon)}
  &=L^{(1+(\kappa+5\varepsilon))(1-\frac{1+\varepsilon}{s})}S^{(1+(\kappa+5\varepsilon))\frac{1+\varepsilon}{s}} \\
  &\leqslant \widetilde\varepsilon S+C_{\widetilde\varepsilon}L^{M_{\kappa,\varepsilon}},\\
  &\lesssim \widetilde\varepsilon S
  +C_{\widetilde\varepsilon}\bigl(1+S^{\Theta_{\rm mp}M_{\kappa,\varepsilon}}\bigr).
\end{aligned}
\end{equation*}
Here
\begin{equation}\label{eq:qeps-M}
  M_{\kappa,\varepsilon}:=
  \frac{(1+(\kappa+5\varepsilon))\left(1-\frac{1+\varepsilon}{s}\right)}{1-(1+(\kappa+5\varepsilon))\frac{1+\varepsilon}{s}}.
\end{equation}

So we obtain
\begin{equation*}
  S\leqslant C\left(1+
  \sum_{q\in\mathscr E_S(\Theta_{\rm mp})}S^q
  +S^{\Theta_{\rm mp}M_{\kappa,\varepsilon}}\right).
\end{equation*}
\noindent\textit{Claim (scalar closure).}
There exists \(p_*>0\) such that, after possibly decreasing \(\varepsilon, \mu_{\rm wt}\), the scalar inequalities hold:
\begin{equation}\label{eq:core-close}
  \Theta M_{\kappa,\varepsilon}<1,
\end{equation}
\begin{equation}\label{eq:full-close}
\begin{gathered}
  A<1\quad\hbox{for every }(A,B)\in\mathscr E_L(p),\\
  \Theta_{\rm mp}(p)M_{\kappa,\varepsilon}<1,
  \qquad
  \max \mathscr E_S(\Theta_{\rm mp}(p))<1,
  \qquad 0<p\leqslant p_*.
\end{gathered}
\end{equation}

 Thus,
\begin{equation*}
  S\leqslant C, \qquad L\leqslant C.
\end{equation*}
Using these bounds, and by \eqref{eq:u1-r} in Proposition~\ref{prop:first-layer}, \eqref{eq:cut-low-R}, we have
\[
  \norm{u_1^\sharp}_{\mathbb S_T^{r,\gamma;\alpha/2}(\rho_{D_1})}\leqslant C.
\]
This proves \eqref{eq:apriori-core}.

Below we verify  the above scalar inequalities.

\smallskip

 As \(\varepsilon\downarrow0\), the definitions \eqref{eq:s-r-gamma}, \eqref{eq:theta-theta0}, \eqref{eq:m-GammaR}, and \eqref{eq:GammaR-def} give
\[
  \alpha\uparrow1-\kappa,
  \qquad s\to2,
  \qquad \frac{1+\varepsilon}{s}\to\frac12,
  \qquad \theta_0\to\frac12,
  \qquad \gamma\to\frac{1-\kappa}{2},
  \qquad r\to1+\kappa.
\]
Moreover,
\[
  2\alpha-1-\kappa-2\varepsilon\to1-3\kappa,
  \qquad
  \frac{s}{\alpha}(2\alpha-1-\kappa-2\varepsilon)\to\frac{2(1-3\kappa)}{1-\kappa}.
\]
As in the proof of Lemma~\ref{lem:split-weights}, for \(\kappa<\sqrt5-2\),
\[
  m_\varepsilon\to1-\kappa,
  \qquad
  \Gamma_R=1+\kappa+\varepsilon+m_\varepsilon\to2.
\]
Using \eqref{eq:theta-ell-def}, \eqref{eq:thetaPsi}, and \eqref{eq:qeps-M}, we obtain
\begin{equation}\label{eq:prop61-limits-main}
  \theta\to\frac{1+\kappa}{2},
  \qquad
  \theta_\Psi\to\frac{2\kappa}{1-\kappa},
  \qquad
  M_{\kappa,\varepsilon}\to\frac{1+\kappa}{1-\kappa}.
\end{equation}

\smallskip
Now we prove  \eqref{eq:core-close}. Recall $ \Theta:=\theta\vee\theta_\Psi$.
Note that
\[
  \frac{2\kappa}{1-\kappa}<\frac{1+\kappa}{2}
  \quad\Longleftrightarrow\quad
  \kappa^2+4\kappa-1<0.
\]
Thus \(\Theta\to(1+\kappa)/2\).  By \eqref{eq:prop61-limits-main}, the limiting form of \eqref{eq:core-close} is
\[
  \frac{1+\kappa}{2}\cdot\frac{1+\kappa}{1-\kappa}<1,
\]
which is equivalent to \(\kappa^2+4\kappa-1<0\).  So \eqref{eq:core-close} holds after decreasing \(\varepsilon\) if necessary.

\smallskip
Now we turn to proving \eqref{eq:full-close}. We first consider the set $\mathscr E_0$.
  By \eqref{eq:Aalpha-Balpha}--\eqref{eq:AK-BK},
\[
  A_\alpha\to\frac{1+\kappa}{2},
  \qquad B_\alpha\to\kappa,
  \qquad A_K\to\frac\kappa2,
  \qquad B_K\to\frac\kappa2.
\]
Applying \(\Pi(A,B)=B/(1-A)\), see \eqref{eq:Pi-def}, gives 
\begin{align*}
  \Pi(A_\alpha,B_\alpha)&\to\frac{2\kappa}{1-\kappa},\\
  \Pi(A_K,B_K)&\to\frac{\kappa}{2-\kappa},\\
  \Pi(A_K+\theta_0A_\alpha,B_K+\theta_0B_\alpha)&\to\frac{4\kappa}{3(1-\kappa)}.
\end{align*}
Each limit is strictly smaller than \((1+\kappa)/2\) on the present range of $\kappa$.  The last comparison follows from \(3\kappa^2+8\kappa-3<0\), which holds for \(0<\kappa<\sqrt5-2\).  Hence all pairs in the base set \(\mathscr E_0\) of \eqref{eq:E0-set} have \(A<1\) and have \(\Pi(A,B)\) below
\((1+\kappa)/2\), 
 after decreasing \(\varepsilon\) if necessary.

Next consider the set \(\mathscr E_L(p)\)
 in \eqref{eq:EL-set} where one  replaces each  base pair \((A,B)\) by \((A+pA_K,B+pB_K)\).
For each fixed base pair, the map
\[
  (p,A,B)\longmapsto \frac{B+pB_K}{1-A-pA_K}
\]
is continuous as long as the denominator is positive.  Since the denominator is positive at \(p=0\) and the limiting inequalities are strict, there is a positive number \(p_*\) such that, uniformly for all sufficiently small \(\varepsilon\) and all \(0<p\leqslant p_*\), every pair in \(\mathscr E_L(p)\) has \(A<1\), and
\[
  \Theta_{\rm mp}(p)M_{\kappa,\varepsilon}<1.
\]
The actual maximum-principle exponent is \(p_{\rm mp}=D_L/(1-D_{\rm dr})\), which tends to zero as \(\mu_{\rm wt}\downarrow0\).  We choose the weight  so that \(0<p_{\rm mp}\leqslant p_*\).  This proves the first assertion in \eqref{eq:full-close} for the value of \(p\) used in \eqref{eq:max-before-cutoff}.

\smallskip
It remains to check the exponents in \(\mathscr E_S\).  At the limiting value \(\Theta_*=(1+\kappa)/2\), 
\[
  A_\alpha\Theta_*+B_\alpha\to\left(\frac{1+\kappa}{2}\right)^2+\kappa<1,
  \qquad
  A_K\Theta_*+B_K\to\frac\kappa2\frac{1+\kappa}{2}+\frac\kappa2<1,
\]
\[
  \left(1-\frac{1+\varepsilon}{s}\right)\Theta_*+\frac{1+\varepsilon}{s}
  \to\frac12\frac{1+\kappa}{2}+\frac12=\frac{3+\kappa}{4}<1,
\]
\[
  (A_K+\theta_0A_\alpha)\Theta_*+B_K+\theta_0B_\alpha
  \to\left(\frac\kappa2+\frac{1+\kappa}{4}\right)\frac{1+\kappa}{2}+\kappa<1.
\]
All these inequalities are strict for \(\kappa<\sqrt5-2\).  Since \(\Theta_{\rm mp}(p)\) stays uniformly close to \((1+\kappa)/2\) for \(0<p\leqslant p_*\), continuity gives, after decreasing \(\varepsilon\) and \(p_*\), if necessary,
\[
  \max \mathscr E_S(\Theta_{\rm mp}(p))<1
\]
for all admissible \(\varepsilon\) and \(0<p\leqslant p_*\).  This proves the last assertion in \eqref{eq:full-close}.

\end{proof}


\begin{corollary}\label{cor:apriori-low-w}
It holds that
\begin{equation}\label{eq:apriori-low-w}
  \norm{w}_{C_T C^{1-\kappa}(\rho_D)}\leqslant C.
\end{equation}
Consequently,
\begin{equation}\label{eq:apriori}
\begin{aligned}
  &\norm{w}_{C_T C^{1-\kappa}(\rho_D)}
  +\norm{w}_{L_T^\infty L^\infty(\rho_{D_L})}
  +\norm{w}_{\mathbb S_T^{s,\gamma;s/2}(\rho_D)} \\
  &\quad
  +\norm{u_1^\sharp}_{C_T C^{1-\kappa}(\rho_{d_0})}
  +\norm{u_1^\sharp}_{\mathbb S_T^{r,\gamma;\alpha/2}(\rho_{D_1})}
  \leqslant C.
\end{aligned}
\end{equation}
\end{corollary}

\begin{proof}
 The estimate \eqref{eq:term-after-cutoff} bounds \(f_R\) in \(\mathfrak N\). By Lemma~\ref{lem:para}, we obtain
\begin{equation}\label{eq:BR-grad-w-NS-proof}
  \norm{B_R(u)(t)\cdot\nabla w(t)}_{C^{-\varepsilon/2}(\rho_D)}
  \lesssim
  \norm{B_R(u)(t)}_{C^\varepsilon(\rho_{D_{\rm dr}})}
  \norm{w(t)}_{C^{1+\varepsilon}(\rho_{D_\varepsilon})},
\end{equation}
because \(D_{\rm dr}+D_\varepsilon<D\) by \eqref{eq:G1-weight-margin}.  Recall
\(
  \omega_B=\gamma(\kappa+5\varepsilon)\frac{1+\varepsilon}{s}.
\)
 Multiplying \eqref{eq:BR-grad-w-NS-proof} by \(t^\gamma\) and splitting the time weights gives
\begin{align*}
  t^\gamma
  \norm{B_R(u)(t)\cdot\nabla w(t)}_{C^{-\varepsilon/2}(\rho_D)}
  &\lesssim
  t^{\gamma-\omega_B-\gamma\frac{1+\varepsilon}{s}}
  \bigl(t^{\omega_B}\norm{B_R(u)(t)}_{C^\varepsilon(\rho_{D_{\rm dr}})}\bigr)
  \bigl(t^{\gamma\frac{1+\varepsilon}{s}}\norm{w(t)}_{C^{1+\varepsilon}(\rho_{D_\varepsilon})}\bigr)\\
  &\lesssim (1+K^{\kappa+5\varepsilon})K.
\end{align*}
Here we used \eqref{eq:B-time-bound} and \eqref{eq:K-def}, and \(\gamma-\omega_B-\gamma(1+\varepsilon)/s\geqslant0\) follows from \eqref{eq:time-product-margin-a}.  Since \(L\) and \(S\) are bounded, the interpolation quantity \(K=L^{1-(1+\varepsilon)/s}S^{(1+\varepsilon)/s}\) is bounded.  Therefore
\begin{equation*}
  B_R(u)\cdot\nabla w\in L_T^{\infty,\gamma}C^{-\varepsilon/2}(\rho_D).
\end{equation*}
By  \cite[Lemma 2.8]{ZZZ22}, for \(0<t\leqslant T\),
\[
\begin{aligned}
  \norm{w(t)}_{C^{1-\kappa}(\rho_D)}
  &\lesssim \int_0^t
  (t-\tau)^{-\frac{1-\kappa+\varepsilon/2}{2}}
  \norm{ (B_R(u)\cdot\nabla w+ f_R)(\tau)}_{C^{-\varepsilon/2}(\rho_D)}\,\mathrm{d}\tau  \\
  &\leqslant
  \int_0^t
  (t-\tau)^{-\frac{1-\kappa+\varepsilon/2}{2}}\tau^{-\gamma}\,\mathrm{d}\tau .
\end{aligned}
\]
The parameter choice \eqref{eq:alpha-eps-choice} gives
\[
  \gamma=\frac{2\alpha+\kappa-1}{2}+\varepsilon
  =\frac{1-\kappa}{2}-4\varepsilon,
\]
and hence
\[
  1-\gamma-\frac{1-\kappa+\varepsilon/2}{2}
  =\kappa+\frac{15}{4}\varepsilon>0.
\]
Thus \eqref{eq:apriori-low-w} follows. Combining \eqref{eq:apriori-low-w} with \eqref{eq:apriori-core}, as well as \eqref{eq:u1-low} gives \eqref{eq:apriori}.
\end{proof}

\begin{remark}\label{rem:szz24-parameter-comparison}
The improvement of the admissible range of $\kappa$ 
can be explained  at the unweighted  level
(the polynomial weights only introduce the small maximum-principle loss $p_{\rm mp}$ and the weight constraints).
In the present argument we take
\begin{equation*}
  \alpha=1-\kappa-5\varepsilon,
  \qquad
  s=2-\frac{\varepsilon}{2},
  \qquad
  r=1+\kappa+7\varepsilon,
\end{equation*}
and estimate $u_1^\sharp$ in $\mathbb S_T^{r,\gamma;\alpha/2}$. 
(Thus, unlike  \cite{SZZ24}, the spatial regularity assigned to $u_1^\sharp$ is not fixed to be $2\alpha$ here.)
Note that regularity $r=2-\alpha+2\varepsilon$ can be deduced by Proposition~\ref{prop:first-layer}, while $r-1-\kappa=7\varepsilon>0$ guarantees that the resonant products involving $u_1^\sharp$
and $\eta_R^>$ are well-defined.

The high block $\eta_R^>$ is estimated in four regularity regimes:
\begin{equation*}
  C^{-(2\alpha-2\varepsilon)},
  \qquad C^{-1-\kappa},
  \qquad C^{\alpha-2},
  \qquad C^{-2+\varepsilon}.
\end{equation*}
These estimates are used at different points of the argument: the $C^{-(2\alpha-2\varepsilon)}$ bound gives the gain
$2^{-(2\alpha-1-\kappa-2\varepsilon)R}$
in the $u_1^\sharp$ estimate and in the commutator-type terms; the $C^{-1-\kappa}$ bound controls rough resonances; the $C^{\alpha-2}$ bound supplies small fixed-cutoff and equal-time estimates; and the $C^{-2+\varepsilon}$ bound is paired with the high regularity $s=2-\varepsilon/2$ of $w$.

By contrast, in \cite{SZZ24} $u_1^\sharp$ is estimated in spatial regularity $2\alpha$ with $\alpha$ close to $2/3$, and the corresponding scalar balance gives a stronger obstruction.
 The final scalar exponent is $2\kappa/(1-\kappa)$, which is below $(1+\kappa)/2$ in the present range.
\end{remark}

\section{Proof of uniqueness}\label{sec:difference-details}

This section proves uniqueness from the a priori construction of the previous sections.  We compare two paracontrolled solutions with the same initial condition in a time-dependent exponentially weighted topology.
We prove uniqueness directly in the solution space of Definition~\ref{def:pc-solution}. Let $(u,u^\sharp)$ and $(v,v^\sharp)$ be two solutions with finite norms in the corresponding spaces. Throughout this section, for any solution-dependent quantity \(A\), write
\begin{equation*}
	\dlt A:=A^u-A^v .
\end{equation*}
Now let \(\rho_{M_a}=\la x\ra ^{-M_a}\) be polynomial weights with \(M_a\) the maximum for all polynomial weights of the solutions. By decreasing \(\mu_{\rm wt}\), we may and do
assume \(M_a=O(\mu_{\rm wt}^{1/2})\) is arbitrarily small.

In this section, we use a new split of the noise. To this end, we introduce
\begin{equation*}
	\sigma_{\eta,\alpha}^{\rm br}:=
	\frac{1-\kappa-\alpha}{2\alpha-1-\kappa-2\varepsilon}(s_0+\sigma_\eta)-s_0.
\end{equation*}
We could take \(\sigma_{\eta,\alpha}^{\rm br}>0\), because \(\sigma_\eta\) is a fixed positive multiple of \(D_L=\mu_{\rm wt}^{1/2}\), while \(s_0=O(\mu_{\rm wt})\).
 We choose
\[
\sigma_{\eta,\alpha}:=\sigma^{\rm br}_{\eta,\alpha}+4M_a.
\]
The extra \(4M_a\) is reserved for the terms where no
exponential time gap is available.
Now we replace the first-level annular threshold by
\begin{equation*}
	J_{k}^{\eta,u}
	:=\left\lceil R+\frac{s_0+\sigma_{\eta,\alpha}}{1-\kappa-\alpha}k\right\rceil,
	\qquad k\geqslant0,
\end{equation*}
whose primary high target is \(\alpha-2\).  In this section only, we keep the notation
\begin{equation*}
	\eta_R^>:=\sum_{k\geqslant0}\widetilde\chi_k\Delta_{>J_k^{\eta,u}}(\chi_k\eta),
	\qquad
	\eta_R^\leqslant:=\eta-\eta_R^>,
\end{equation*}
for this uniqueness representative.   The corresponding low block has a fixed polynomial growth exponent.  We name it explicitly:
\begin{equation*}
	d_{\eta,u}
	:=s_0+(1+\kappa+\varepsilon)\frac{s_0+\sigma_{\eta,\alpha}}{1-\kappa-\alpha}.
\end{equation*}
For the two additional rough regularities used later, set
\begin{equation*}
	\begin{aligned}
		\sigma_{\eta,u}^{(2\alpha)}
		&:=\frac{2\alpha-1-\kappa-2\varepsilon}{1-\kappa-\alpha}(s_0+\sigma_{\eta,\alpha})-s_0,\\
		\sigma_{\eta,u}^{(-2+\varepsilon)}
		&:=\frac{1-\kappa-\varepsilon}{1-\kappa-\alpha}(s_0+\sigma_{\eta,\alpha})-s_0.
	\end{aligned}
\end{equation*}
By Lemma~\ref{lem:bridge-localization}, it holds that
\begin{equation}\label{eq:eta-uniq-split-bounds}
	\begin{aligned}
		\|\eta_R^>\|_{L_T^\infty C^{\alpha-2}(\rho_{-\sigma_{\eta,\alpha}})}
		&\lesssim 2^{-(1-\kappa-\alpha)R},\\
		\|\eta_R^>\|_{L_T^\infty C^{-(2\alpha-2\varepsilon)}(\rho_{-\sigma_{\eta,u}^{(2\alpha)}})}
		&\lesssim 2^{-(2\alpha-1-\kappa-2\varepsilon)R},\\
		\|\eta_R^>\|_{L_T^\infty C^{-2+\varepsilon}(\rho_{-\sigma_{\eta,u}^{(-2+\varepsilon)}})}
		&\lesssim 2^{-(1-\kappa-\varepsilon)R},\\
		\|\eta_R^>\|_{L_T^\infty C^{-1-\kappa}(\rho_{s_0})}
		+\|\eta_R^\leqslant\|_{L_T^\infty C^{-1-\kappa}(\rho_{s_0})}
		&\lesssim 1,\\
		\|\eta_R^\leqslant\|_{L_T^\infty C^\varepsilon(\rho_{d_{\eta,u}})}
		&\lesssim 2^{(1+\kappa+\varepsilon)R}.
	\end{aligned}
\end{equation}
	Consequently, by Lemma~\ref{lem:mild-time-increments},
	\begin{equation}\label{eq:Ieta-high-alpha-decay}
		\|\sI\eta_R^>\|_{\mathbb{S}_T^{\alpha,0;\alpha/2}(\rho_{-\sigma_{\eta,\alpha}})}
		\lesssim 2^{-(1-\kappa-\alpha)R}.
	\end{equation}


We still use the following notation:
\[
	\cR(u):=F(u)-F'(u)\para u,
	\qquad H(u):=F(u)F'(u),\]
\begin{equation*}
	u_R^\sharp:=u-F(u)\para\sI\eta_R^>
	=u^\sharp+F(u)\para\sI(\eta_R^\leqslant).
\end{equation*}
We also define
\begin{equation}\label{eq:single-term-def}
	\begin{aligned}
		\mathcal N_R(u,u_R^\sharp):={}&F(u)\Par\eta_R^>
		+F(u)\eta_R^\leqslant
		+\cR(u)\res\eta_R^>
		+(F'(u)\para u_R^\sharp)\res\eta_R^> \\
		&+\operatorname{Com}(F'(u),F(u)\para\sI\eta_R^>,\eta_R^>)
		+F'(u)\operatorname{Com}(F(u),\sI\eta_R^>,\eta_R^>) \\
		&+H(u)\bigl(\Psi_{R_1}^>+\Theta_{R,R_1}\bigr).
	\end{aligned}
\end{equation}

\begin{lemma}\label{lem:single-cutoff-product-identity}
	Fix  cutoffs \((R,R_1)\), and let \((u,u^\sharp)\) be a  solution in the sense of Definition~\ref{def:pc-solution}.  Put
	\[
	u_R^\sharp:=u-F(u)\para\sI\eta_R^>.
	\]
	Then, in distributions,
	\begin{equation}\label{eq:single-cutoff-product-identity}
		F(u)\eta-F(u)\para\eta_R^>
		=\mathcal N_R(u,u_R^\sharp).
	\end{equation}
	Consequently \(u_R^\sharp\) satisfies
	\begin{equation}\label{eq:single-cutoff-mild}
		u_R^\sharp=P_tu_0+[\sI,F(u)\para]\eta_R^>+\sI\mathcal N_R(u,u_R^\sharp).
	\end{equation}
\end{lemma}

\begin{proof}
  Split \(\eta=\eta_R^>+\eta_R^\leqslant\).  Since
	\(
	u=F(u)\para\sI\eta_R^>+u_R^\sharp,
	\)
we have
	\begin{align*}
		F(u)\eta_R^>={}&F(u)\para\eta_R^>+F(u)\Par\eta_R^>
		+\cR(u)\res\eta_R^> +(F'(u)\para u_R^\sharp)\res\eta_R^> \notag\\
		&+\operatorname{Com}(F'(u),F(u)\para\sI\eta_R^>,\eta_R^>)
		+F'(u)\operatorname{Com}(F(u),\sI\eta_R^>,\eta_R^>) \\[-0.4ex]
		&+H(u)\Psi_R^{>>} .\notag
	\end{align*}
	Here \(\Psi_R^{>>}\) is  the high--high block with the new split, and the same argument as Lemma~\ref{lem:inherited-second-level} gives
	\[
	\Psi_R^{>>}=\Psi_{R_1}^>+\Theta_{R,R_1}.
	\]
	As
	\(
	F(u)\eta=F(u)\eta_R^\leqslant+F(u)\eta_R^>.
	\)
	Subtracting the  paraproduct \(F(u)\para\eta_R^>\) gives exactly \eqref{eq:single-cutoff-product-identity} with the definition \eqref{eq:single-term-def}.   Finally, the mild formulation of the equation gives
	\[
	u_R^\sharp=P_tu_0+\sI F(u)\eta-F(u)\para\sI\eta_R^>
	=P_tu_0+[\sI,F(u)\para]\eta_R^>+\sI\bigl(F(u)\eta-F(u)\para\eta_R^>\bigr),
	\]
	which is \eqref{eq:single-cutoff-mild} by \eqref{eq:single-cutoff-product-identity}.
\end{proof}

We use the time-dependent exponential weight, as in \cite{HL2d,HL}:
\begin{equation*}
	\mathfrak e_t(x):=\exp\{-(1+t)\langle x\rangle^\varrho\},
	\qquad t\geqslant0,
	\qquad 0<\varrho<1 .
\end{equation*}
Then, for every \(0<a<t\) and \(M\geqslant0\),
\begin{equation}\label{eq:single-exp-absorbs-product-form}
	\|f\|_{C^\beta(\mathfrak e_t)}
	\leqslant C(t-a)^{-M/\varrho}
	\|f\|_{C^\beta(\mathfrak e_a\rho_M)} .
\end{equation}



For \(0<\tau\leqslant T\), write
\begin{equation*}
	\|f\|_{\mathbb S_{\tau,\mathfrak e}^{\beta,\omega;\nu}}
	:=\sup_{0<t\leqslant\tau}t^\omega\|f(t)\|_{C^\beta(\mathfrak e_t)}
	+\sup_{0<s<t\leqslant\tau}s^\omega
	\frac{\|f(t)-f(s)\|_{L^\infty(\mathfrak e_t)}}{|t-s|^\nu}.
\end{equation*}
As in the notation of Section~\ref{sec:target}, we also write
\[
	\|f\|_{C_{t,\mathfrak e}^{\nu,\omega}L^\infty}
	:=\sup_{0<\sigma<\tau\leqslant t}
	\sigma^\omega
	\frac{\|f(\tau)-f(\sigma)\|_{L^\infty(\mathfrak e_\tau)}}{|\tau-\sigma|^\nu}.
\]
The comparison norm is
\begin{equation}\label{eq:single-X-def}
	\begin{aligned}
		X(t):={}&\|\dlt u\|_{\mathbb S_{t,\mathfrak e}^{\alpha,\omega_\alpha;\nu_\alpha}}
		+\sup_{0<a\leqslant t}\|\dlt u(a)\|_{L^\infty(\mathfrak e_a)} \\
		&+\|\dlt u_R^\sharp\|_{\mathbb S_{t,\mathfrak e}^{r_\circ,\gamma;\nu_\alpha}}
		+\sup_{0<a\leqslant t}\|\dlt u_R^\sharp(a)\|_{C^{1-\kappa}(\mathfrak e_a)} .
	\end{aligned}
\end{equation}
Here
\begin{equation*}
r_\circ:=1+\kappa+\frac72\varepsilon,
\qquad
\omega_\alpha:=\gamma\alpha/s.
\end{equation*}
and
\begin{equation*}
	\nu_\alpha:=\frac{\alpha}{2}-\frac{4M_*}{\varrho}>0,
\end{equation*}
where \(M_*>0\) depending on $M_a$ is a  polynomial weight exponent dominating all finite losses
generated in the uniqueness estimates. Moreover, $M_*$ can be chosen arbitrarily small by taking $\mu_{\rm wt}\to0$.

Throughout the rest of this section, \(C_{\rm abs}\) denotes a constant depending only on the fixed parameters, on the enhanced-noise norm, and on the a priori bounds of the two solutions in the spaces of Definition~\ref{def:pc-solution}.  It is independent of the time variable and of the cutoff levels \(R,R_1\), except through explicitly displayed factors such as \(2^{-(1-\kappa-\alpha)R}\).  Constants allowed to depend on the fixed cutoff levels and on the background bounds are denoted by \(C_{R,R_1,\mathcal B}\).

\begin{lemma}\label{lem:single-sharp-composition}
	It holds that
	\begin{equation}\label{eq:single-ansatz-small-lemma}
		\|\dlt F\para\sI\eta_R^>\|_{\mathbb S_{t,\mathfrak e}^{\alpha,\omega_\alpha;\nu_\alpha}}
		+\sup_{0<a\leqslant t}\|\dlt F(a)\para\sI\eta_R^>(a)\|_{L^\infty(\mathfrak e_a)}
		\leqslant C_{\rm abs}2^{-(1-\kappa-\alpha)R}X(t),
	\end{equation}
	and
\begin{equation}\label{eq:single-sharp-composition-time}
		\|\dlt F\|_{C_{t,\mathfrak e}^{\nu_\alpha,\omega_\alpha}L^\infty}
		\leqslant C_{\rm abs}X(t).
	\end{equation}

\end{lemma}

\begin{proof}
	We first prove \eqref{eq:single-ansatz-small-lemma}.  By \eqref{eq:Ieta-high-alpha-decay}, we obtain
	\begin{equation*}
		\sup_{0<a\leqslant t}a^{\omega_\alpha}\|\dlt F(a)\para\sI\eta_R^>(a)\|_{C^\alpha(\mathfrak e_a)}
		+\sup_{0<a\leqslant t}\|\dlt F(a)\para\sI\eta_R^>(a)\|_{L^\infty(\mathfrak e_a)}
		\leqslant C_{\rm abs}2^{-(1-\kappa-\alpha)R}X(t).
	\end{equation*}
	For \(0<s<t\), write
	\begin{align*}
		&\dlt F(t)\para\sI\eta_R^>(t)-\dlt F(s)\para\sI\eta_R^>(s) \\
		&\quad=\bigl(\dlt F(t)-\dlt F(s)\bigr)\para\sI\eta_R^>(t)
		+\dlt F(s)\para\bigl(\sI\eta_R^>(t)-\sI\eta_R^>(s)\bigr).
	\end{align*}
	The second term is immediate from \eqref{eq:single-X-def}, \eqref{eq:Ieta-high-alpha-decay}, Lemma \ref{lem:para} and the monotonicity \(\mathfrak e_t\leqslant \mathfrak e_s\):
	\[
	s^{\omega_\alpha}
	\frac{\|\dlt F(s)\para(\sI\eta_R^>(t)-\sI\eta_R^>(s))\|_{L^\infty(\mathfrak e_t)}}
	{|t-s|^{\nu_\alpha}}
	\leqslant C_{\rm abs}2^{-(1-\kappa-\alpha)R}X(t).
	\]
	 We write
	\[
	\dlt F=\mathcal F(u,v)\dlt u,
	\qquad
	\mathcal F(x,y):=\int_0^1F'(y+\theta(x-y))\,\mathrm{d}\theta .
	\]
	Then
	\begin{align*}
		\dlt F(t)-\dlt F(s)
		={}&\mathcal F(u(t),v(t))\bigl(\dlt u(t)-\dlt u(s)\bigr) \\
		&+\bigl(\mathcal F(u(t),v(t))-\mathcal F(u(s),v(s))\bigr)\dlt u(s).
	\end{align*}
	The first line is bounded by the \(\nu_\alpha\)-time component of \(X(t)\), the boundedness of \(\mathcal F\) and \eqref{eq:Ieta-high-alpha-decay}.  For the second line, we obtain
	\[
	\|\mathcal F(u(t),v(t))-\mathcal F(u(s),v(s))\|_{L^\infty(\rho_{M_a})}
	\leqslant C_{\rm abs}s^{-\omega_\alpha}|t-s|^{\nu_\alpha}.
	\]
	  Since \(\sigma_{\eta,\alpha}\geqslant4M_a\), the weight of \(\sI\eta_R^>\) compensates for this  polynomial weight.  Thus
	\[
	\|\bigl(\mathcal F(t)-\mathcal F(s)\bigr)\dlt u(s)\para\sI\eta_R^>(t)\|_{L^\infty(\mathfrak e_t)}
	\leqslant C_{\rm abs}2^{-(1-\kappa-\alpha)R}s^{-\omega_\alpha}|t-s|^{\nu_\alpha}X(t).
	\]
 Combining the three estimates proves the time-H\"older part of \eqref{eq:single-ansatz-small-lemma}.

 We now prove \eqref{eq:single-sharp-composition-time}. The argument is the same as above, except for the following point:
	Using \eqref{eq:single-exp-absorbs-product-form} with the time gap \(t-s\), and using
	\(\|\dlt u(s)\|_{L^\infty(\mathfrak e_s)}\leqslant X(t)\), we get
	\[
	\begin{aligned}
		\|\bigl(\mathcal F(u(t),v(t))-\mathcal F(u(s),v(s))\bigr)
		  \dlt u(s)\|_{L^\infty(\mathfrak e_t)}
		&\leqslant C(t-s)^{-M_*/\varrho}
		s^{-\omega_\alpha}|t-s|^{\alpha/2}X(t).
	\end{aligned}
	\]
As \(\alpha/2-\nu_\alpha-M_*/\varrho=3M_*/\varrho>0\), \eqref{eq:single-sharp-composition-time} follows.
\end{proof}

Set
\begin{equation*}
\varepsilon_R:=2^{-(1-\kappa-\alpha)R}
+2^{-(2\alpha-1-\kappa-2\varepsilon)R}
+2^{-(\alpha-2\kappa-3\varepsilon)R_1}.
\end{equation*}

\begin{lemma}\label{lem:single-remainder-diff}
	It holds that
	\begin{equation}\label{eq:single-remainder-diff-output}
		\begin{aligned}
			&\|\dlt u_R^\sharp\|_{\mathbb S_{t,\mathfrak e}^{r_\circ,\gamma;\nu_\alpha}}
			+\sup_{0<a\leqslant t}\|\dlt u_R^\sharp(a)\|_{C^{1-\kappa}(\mathfrak e_a)} +\|\dlt u_R^\sharp\|_{\mathbb S_{t,\mathfrak e}^{\alpha,\omega_\alpha;\nu_\alpha}}\\
			&\qquad\leqslant C_{\rm abs}\varepsilon_R X(t)
			+\sum_m C_{m}\bigl(\mathcal V_{\mu_m,\lambda_m}X\bigr)(t),
		\end{aligned}
	\end{equation}
	where \(\mathcal V_{\mu,\lambda}\) is defined in \eqref{eq:volterra-envelope-def}, and each kernel satisfies \(0\leqslant\mu_m,\lambda_m<1\) and \(\mu_m+\lambda_m<1\).
\end{lemma}

\begin{proof}
	Subtract \eqref{eq:single-cutoff-mild} for the two solutions:
	\begin{equation*}
		\dlt u_R^\sharp=[\sI,\dlt F\para]\eta_R^>
		+\sI\bigl(\dlt\mathcal N_R\bigr).
	\end{equation*}

\emph{Estimate of $\sI\bigl(\dlt\mathcal N_R\bigr)$.}
 Using
	\(\eta_R^>\in C^{-1-\kappa}(\rho_{s_0})\), Lemma~\ref{lem:para}, and then \eqref{eq:single-exp-absorbs-product-form},
	\begin{equation*}
		\|\dlt F(a)\Par\eta_R^>(a)\|_{C^{\alpha-1-\kappa}(\mathfrak e_t)}
		\lesssim (t-a)^{-M_*/\varrho}a^{-\omega_\alpha}X(a).
	\end{equation*}
	By \eqref{eq:eta-uniq-split-bounds} and \eqref{eq:single-exp-absorbs-product-form}, we obtain
	\begin{equation*}
		\|\dlt F(a)\eta_R^\leqslant(a)\|_{L^\infty(\mathfrak e_t)}
		\lesssim(t-a)^{-M_*/\varrho}a^{-\omega_\alpha}X(a).
	\end{equation*}
	Using \(\Psi_{R_1}^>\in C^{-2\kappa}\) from \eqref{eq:Psi-rough-bound}, and treating \(\Theta_{R,R_1}\) as a fixed classical factor, gives
	\begin{equation}\label{eq:single-Psi-term}
		\begin{aligned}
			\|\dlt H(a)(\Psi_{R_1}^>+\Theta_{R,R_1})(a)\|_{C^{-2\kappa}(\mathfrak e_t)}
			&\lesssim(t-a)^{-M_*/\varrho}a^{-\omega_\alpha}X(a).
		\end{aligned}
	\end{equation}
	
Applying Lemma~\ref{lem:weighted-paralinearization} with \(G=F\), \(f=u(a)\), and \(g=v(a)\), we obtain
	\[
	\begin{aligned}
		\|\dlt\cR(a)\|_{C^{2\alpha}(\mathfrak e_a\rho_{M_*})}
		&\lesssim a^{-2\omega_\alpha}X(a),
	\end{aligned}
	\]
where we used	\[
		\|\dlt u(a)\|_{C^\alpha(\mathfrak e_a)}\leqslant a^{-\omega_\alpha}X(a),
		\qquad
		\|\dlt u(a)\|_{L^\infty(\mathfrak e_a)}\leqslant X(a).
	\]
	Using
	\(
	\eta_R^>\in C^{-(2\alpha-2\varepsilon)}(\rho_{-\sigma_{\eta,u}^{(2\alpha)}}),
	\)
	 we obtain by  \eqref{eq:single-exp-absorbs-product-form},
	\begin{equation*}
		\|\dlt\cR(a)\res\eta_R^>(a)\|_{C^{2\varepsilon}(\mathfrak e_t)}
		\lesssim (t-a)^{-M_*/\varrho}a^{-2\omega_\alpha}X(a).
	\end{equation*}

  Expand
	\begin{equation*}
		\begin{aligned}
			\dlt\bigl((F'(u)\para u_R^\sharp)\res\eta_R^>\bigr)
			={}&\bigl((F'(u)-F'(v))\para u_R^\sharp\bigr)\res\eta_R^> \\
			&+\bigl(F'(v)\para\dlt u_R^\sharp\bigr)\res\eta_R^> .
		\end{aligned}
	\end{equation*}
For the first term, bound \(\dlt F'\) in \(L^\infty(\mathfrak e_a)\) and  \(u_R^\sharp\in C^r\); this gives a term controlled by \(a^{-\gamma}X(a)\) and has an admissible resonant product with \(\eta_R^>\in C^{-1-\kappa}\), because \(r-1-\kappa=7\varepsilon>0\).  The second term uses \(\dlt u_R^\sharp\in C^{r_\circ}\) and the same rough high-noise norm. This is valid since \(r_\circ-1-\kappa=7\varepsilon/2>0\).    Hence by \eqref{eq:single-exp-absorbs-product-form}, we obtain
	\begin{equation*}
		\|\dlt((F'\para u_R^\sharp)\res\eta_R^>)(a)\|_{L^\infty(\mathfrak e_t)}
		\lesssim(t-a)^{-M_*/\varrho}
		a^{-\gamma}X(a).
	\end{equation*}

 Let \(\mathcal Q_R^{{\rm com},1}\) and \(\mathcal Q_R^{{\rm com},2}\) denote the two commutator terms in \eqref{eq:single-term-def}.   Expanding every solution-dependent factor gives
	\begin{align*}
		\dlt\mathcal Q_R^{{\rm com},1}
		={}&\operatorname{Com}(\dlt F',F(u)\para\sI\eta_R^>,\eta_R^>)
		+\operatorname{Com}(F'(v),\dlt F\para\sI\eta_R^>,\eta_R^>),\\
		\dlt\mathcal Q_R^{{\rm com},2}
		={}&\dlt F'\operatorname{Com}(F(u),\sI\eta_R^>,\eta_R^>)
		+F'(v)\operatorname{Com}(\dlt F,\sI\eta_R^>,\eta_R^>).
	\end{align*}
	The four terms are controlled as follows.
	\begin{itemize}[leftmargin=1.5em]
		\item \(\operatorname{Com}(\dlt F',F(u)\para\sI\eta_R^>,\eta_R^>)\): use \(\dlt F'\in C^\alpha\), \(F(u)\para\sI\eta_R^>\in C^{2-2\alpha+2\varepsilon}\), and \(\eta_R^>\in C^{-1-\kappa}\).
		\item \(\operatorname{Com}(F'(v),\dlt F\para\sI\eta_R^>,\eta_R^>)\): use  \(F'(v)\in C^\alpha\),  \(\dlt F\para\sI\eta_R^>\in C^{2-2\alpha+2\varepsilon}\), and \(\eta_R^>\in C^{-1-\kappa}\).
		\item \(\dlt F'\operatorname{Com}(F(u),\sI\eta_R^>,\eta_R^>)\): first estimate the  commutator in the positive regularity \(C^{r-1-\kappa}\), then multiply by the \(L^\infty\) difference \(\dlt F'\).
		\item \(F'(v)\operatorname{Com}(\dlt F,\sI\eta_R^>,\eta_R^>)\): use \(\dlt F\in C^\alpha\), \(\sI\eta_R^>\in C^{2-2\alpha+2\varepsilon}\), and \(\eta_R^>\in C^{-1-\kappa}\).
	\end{itemize}
	In the commutator estimates the singular pair has negative sum,
	\((2-2\alpha+2\varepsilon)+(-1-\kappa)<0\), while the total regularity is
	\[
	\alpha+(2-2\alpha+2\varepsilon)-1-\kappa=r-1-\kappa>0.
	\]
	Thus by Lemma~\ref{lem:weighted-bony-commutator-catalogue} and \eqref{eq:single-exp-absorbs-product-form}, we obtain
	\begin{equation*}
		\|\dlt\mathcal Q_R^{{\rm com},1}(a)\|_{L^\infty(\mathfrak e_t)}
		+\|\dlt\mathcal Q_R^{{\rm com},2}(a)\|_{L^\infty(\mathfrak e_t)}
		\lesssim(t-a)^{-M_*/\varrho}a^{-\omega_\alpha}X(a).
	\end{equation*}
Combining the above estimates, we may write
\(\dlt\mathcal N_R=h_1+h_2\) so that, for \(0<a<t\leqslant T\),
\begin{equation*}
\begin{aligned}
\|h_1(a)\|_{C^{\alpha-1-\kappa}(\mathfrak e_t)}
&\leqslant C_{R,R_1,\mathcal B}(t-a)^{-M_*/\varrho}
      a^{-\omega_\alpha}X(a),\\
\|h_2(a)\|_{L^\infty(\mathfrak e_t)}
&\leqslant C_{R,R_1,\mathcal B}(t-a)^{-M_*/\varrho}
      a^{-\gamma}X(a).
\end{aligned}
\end{equation*}
Here the \(C^{-2\kappa}\)-term in \eqref{eq:single-Psi-term} is included in
\(h_1\), since \(C^{-2\kappa}\hookrightarrow C^{\alpha-1-\kappa}\).  The
terms with time powers \(a^{-\omega_\alpha}\) and \(a^{-2\omega_\alpha}\) that are
placed in \(h_2\) are absorbed into \(a^{-\gamma}\), because
\(2\omega_\alpha<\gamma\) after decreasing \(\varepsilon\).

We now apply Lemma~\ref{lem:duhamel-volterra-reduction} separately to
\(h_1\) and \(h_2\).  Put
\[
 m_*:=\frac{M_*}{\varrho},\qquad
 \theta_1:=\alpha-1-\kappa,
 \quad \lambda_1:=\omega_\alpha,
 \qquad
 \theta_2:=0,
 \quad \lambda_2:=\gamma .
\]
For \(h_1\),
\((\beta,\Omega,\nu)=(r_\circ,\gamma,\nu_\alpha)\) gives
\[
 \mu_{1,r}:=\frac{r_\circ-\theta_1}{2}+m_*
 =\frac{r_\circ-\alpha+1+\kappa}{2}+m_*,
 \qquad
 \mu_{1,0}:=\frac{-\theta_1}{2}+m_*
 =\frac{1+\kappa-\alpha}{2}+m_* .
\]
The four conditions of Lemma~\ref{lem:duhamel-volterra-reduction} are
\[
 \mu_{1,r}<1,
 \qquad \mu_{1,r}+\omega_\alpha-\gamma<1,
 \qquad \mu_{1,0}+\nu_\alpha<1,
 \qquad \mu_{1,0}+\nu_\alpha+\omega_\alpha-\gamma<1 .
\]
At \(\varepsilon=m_*=0\) these correspond to
\[
 \frac{1-3\kappa}{2}>0,\qquad
 \frac{3-6\kappa-\kappa^2}{4}>0,\qquad
 \frac{1-\kappa}{2}>0,\qquad
 1-\frac{(1+\kappa)^2}{4}>0,
\]
which holds for \(0<\kappa<\sqrt5-2\).  Hence the
conditions hold after decreasing \(\varepsilon\) and the polynomial weight.
For \((\beta,\Omega,\nu)=(\alpha,\omega_\alpha,\nu_\alpha)\),
\[
 \mu_{1,\alpha}:=\frac{\alpha-\theta_1}{2}+m_*
 =\frac{1+\kappa}{2}+m_* ,
 \qquad
 \mu_{1,0}=\frac{1+\kappa-\alpha}{2}+m_* ,
\]
and since \(\lambda_1=\Omega=\omega_\alpha\), the spatial conditions reduce to
\(\mu_{1,\alpha}<1\), while the time-H\"{o}lder conditions are the two already checked
conditions involving \(\mu_{1,0}\).  For
\((\beta,\Omega)=(1-\kappa,0)\), only the spatial part of the lemma is used and
\[
 \mu_{1,1-\kappa}:=\frac{1-\kappa-\theta_1}{2}+m_*
 =1-\frac{\alpha}{2}+m_* .
\]
Thus the required inequalities hold.

For \(h_2\), the parameters are \(\theta_2=0\), \(\lambda_2=\gamma\), and
\(M=M_*\).  For
\((\beta,\Omega,\nu)=(r_\circ,\gamma,\nu_\alpha)\),
\[
 \mu_{2,r}:=\frac{r_\circ}{2}+m_* ,
 \qquad \mu_{2,0}:=m_* .
\]
Since \(\lambda_2=\Omega=\gamma\), the conditions hold.  For
\((\beta,\Omega,\nu)=(\alpha,\omega_\alpha,\nu_\alpha)\),
\[
 \mu_{2,\alpha}:=\frac{\alpha}{2}+m_* ,
 \qquad \mu_{2,0}=m_* .
\]
The only required new condition is
\[
 1-\frac{\alpha}{2}-\gamma+\omega_\alpha-O(m_*)>0,
\]
which follows from the parameter conditions.  For
\((\beta,\Omega)=(1-\kappa,0)\),
\[
 \mu_{2,1-\kappa}:=\frac{1-\kappa}{2}+m_* ,
\]
and the spatial condition is
\[
 \frac{1-\kappa}{2}+\gamma+m_*<1,
\]
which follows from the parameter conditions.

Consequently Lemma~\ref{lem:duhamel-volterra-reduction} gives
\begin{equation}\label{eq:single-INR-volterra-output}
\begin{aligned}
&\|\sI(\dlt\mathcal N_R)\|_{\mathbb S_{t,\mathfrak e}^{r_\circ,\gamma;\nu_\alpha}}
 +\sup_{0<a\leqslant t}\|\sI(\dlt\mathcal N_R)(a)\|_{C^{1-\kappa}(\mathfrak e_a)}
 +\|\sI(\dlt\mathcal N_R)\|_{\mathbb S_{t,\mathfrak e}^{\alpha,\omega_\alpha;\nu_\alpha}}  \\
&\qquad\leqslant
 \sum_m C_{m}\bigl(\mathcal V_{\mu_m,\lambda_m}X\bigr)(t),
\end{aligned}
\end{equation}
where each resulting pair satisfies
\(0\leqslant \mu_m,\lambda_m<1\) and \(\mu_m+\lambda_m<1\).
	
\emph{The parabolic commutator.} Let
\[
\mathcal C_R:=[\sI,\dlt F\para]\eta_R^> .
\]
We apply Lemma~\ref{lem:weighted-time-freezing-commutator} with
\[
f=\dlt F,
\qquad
g=\eta_R^>,
\qquad
\beta=\alpha-2,
\qquad
\nu=\nu_\alpha,
\qquad
\omega=\omega_\alpha.
\]
Recall that by Lemma~\ref{lem:single-sharp-composition}, \eqref{eq:single-sharp-composition-time},
\begin{equation}\label{eq:single-fixed-terminal-DF-bound}
	\begin{aligned}
		&\sup_{0<a\leqslant\tau}a^{\omega_\alpha}
		\|\dlt F(a)\|_{C^\alpha(\mathfrak e_\tau)}  +
			\|\dlt F\|_{C_{t,\mathfrak e}^{\nu_\alpha,\omega_\alpha}}
		\lesssim X(t),
	\end{aligned}
\end{equation}
and by {the first bound in \eqref{eq:eta-uniq-split-bounds}}
\begin{equation}\label{eq:single-comm-model-decay}
	\|\eta_R^>\|_{C^{\alpha-2}(\rho_{-\sigma_{\eta,\alpha}})}
	\lesssim 2^{-(1-\kappa-\alpha)R} .
\end{equation}
 For the high regularity \(q=r_\circ\), set
\[
\theta_h:=r_\circ-(\alpha+\beta)=r_\circ-2\alpha+2,
\qquad \Omega=\gamma .
\]
It follows that \(0\leqslant \theta_h\leqslant 2+2\nu_\alpha-\alpha\).  For \eqref{eq:app-lemma313-style-margins}, we have
\[
\gamma+\frac{2\alpha-r_\circ}{2}-\omega_\alpha\longrightarrow \frac{3-6\kappa-\kappa^2}{4}>0.
\]
 Hence
\[
\tau^\gamma\|\mathcal C_R(\tau)\|_{C^{r_\circ}(\mathfrak e_\tau)}
\leqslant C_{\rm abs}2^{-(1-\kappa-\alpha)R}X(t).
\]
For the low regularity \(q=1-\kappa\), take
\[
\theta_{1-\kappa}:=1-\kappa-(\alpha+\beta)=3-\kappa-2\alpha,
\qquad \Omega=0 .
\]
The two  conditions become
\[
\frac{2-\theta_{1-\kappa}}2-\omega_\alpha>0,
\qquad
\frac{2+2\nu_\alpha-\alpha-\theta_{1-\kappa}}2-\omega_\alpha>0.
\]
Both reduce to the following, up to the already reserved \(M_*/\varrho\) loss:
\[
\frac{2-\theta_{1-\kappa}}2-\omega_\alpha\longrightarrow \frac{(1-\kappa)(1+\kappa)}4>0.
\]
 Hence
\[
\tau^{\omega_\alpha}\|\mathcal C_R(\tau)\|_{C^\alpha(\mathfrak e_\tau)}\leqslant \sup_{0<\tau\leqslant t}\|\mathcal C_R(\tau)\|_{C^{1-\kappa}(\mathfrak e_\tau)}
\leqslant C_{\rm abs}2^{-(1-\kappa-\alpha)R}X(t).
\]

Finally, the \(L^\infty\)-valued time-H\"older part is estimated directly here.
Using \eqref{eq:single-fixed-terminal-DF-bound}, \eqref{eq:single-comm-model-decay}, and \(2\nu_\alpha\leqslant\alpha\) gives
\[
  \|\dlt F\para\eta_R^>\|_{ C^{2\nu_\alpha-2}
       (\mathfrak e\rho_{-\sigma_{\eta,\alpha}})}
  \leqslant C_{\rm abs}2^{-(1-\kappa-\alpha)R}X(t).
\]
{Lemma~\ref{lem:mild-time-increments}\textup{(i)} therefore yields}
\[
  \|\sI(\dlt F\para\eta_R^>)\|_{C_t^{\nu_\alpha,\omega_\alpha}L^\infty
       (\mathfrak e)}
  \leqslant C_{\rm abs}2^{-(1-\kappa-\alpha)R}X(t).
\]
{Lemma~\ref{lem:single-sharp-composition}, specifically \eqref{eq:single-ansatz-small-lemma}, gives the time regularity for the other term.} Altogether,
\begin{equation}\label{eq:single-commutator-appendix-term}
	\begin{aligned}
		&\|\mathcal C_R\|_{\mathbb S_{t,\mathfrak e}^{r_\circ,\gamma;\nu_\alpha}}
		+\sup_{0<a\leqslant t}\|\mathcal C_R(a)\|_{C^{1-\kappa}(\mathfrak e_a)}
		+\|\mathcal C_R\|_{\mathbb S_{t,\mathfrak e}^{\alpha,\omega_\alpha;\nu_\alpha}} \\
		&\qquad\leqslant C_{\rm abs}2^{-(1-\kappa-\alpha)R}X(t).
	\end{aligned}
\end{equation}

Combining
\eqref{eq:single-INR-volterra-output}, with
\eqref{eq:single-commutator-appendix-term}, and using
\(
 2^{-(1-\kappa-\alpha)R}\leqslant \varepsilon_R,
\)
gives exactly \eqref{eq:single-remainder-diff-output}.
\end{proof}

\begin{theorem}\label{thm:direct-uniqueness}
	Assume \eqref{eq:F-lip-hyp}.  Two paracontrolled solutions in the sense of Definition~\ref{def:pc-solution} with the same initial condition coincide on \([0,T]\).
\end{theorem}

\begin{proof}
	
	By Lemma~\ref{lem:single-sharp-composition}, more precisely by \eqref{eq:single-ansatz-small-lemma},
	\begin{equation}\label{eq:single-ansatz-small}
		\|\dlt F\para\sI\eta_R^>\|_{\mathbb S_{t,\mathfrak e}^{\alpha,\omega_\alpha;\nu_\alpha}}
		+\sup_{0<a\leqslant t}\|\dlt F(a)\para\sI\eta_R^>(a)\|_{L^\infty(\mathfrak e_a)}
		\leqslant C_{\rm abs}\varepsilon_R X(t).
	\end{equation}
	Combining \eqref{eq:single-ansatz-small} with Lemma~\ref{lem:single-remainder-diff} yields
	\begin{equation*}
		X(t)\leqslant C_{\rm abs}\varepsilon_R X(t)
		+\sum_{m=1}^NC_{m,R,R_1,\mathcal B}\bigl(\mathcal V_{\mu_m,\lambda_m}X\bigr)(t),
	\end{equation*}
	where \(0\leqslant\mu_m,\lambda_m<1\) and \(\mu_m+\lambda_m<1\).  Choose \((R,R_1)\) so large that
\begin{equation*}
	C_{\rm abs}\varepsilon_R\leqslant \frac1{16}.
\end{equation*}  Lemma~\ref{lem:volterra-gronwall} gives \(X\equiv0\).  Hence \(\dlt u=0\) and \(\dlt u_R^\sharp=0\).
	
  Since
	\[
	u^\sharp=u_R^\sharp-F(u)\para\sI(\eta_R^\leqslant),
	\qquad
	v^\sharp=v_R^\sharp-F(v)\para\sI(\eta_R^\leqslant),
	\]
	and \(\dlt u=0\) implies \(\dlt F=0\), we also get
	\[
	\dlt u^\sharp=\dlt u_R^\sharp-\dlt F\para\sI(\eta_R^\leqslant)=0.
	\]
	Thus the two paracontrolled solutions in the sense of Definition~\ref{def:pc-solution} coincide on \([0,T]\).
\end{proof}

\begin{proof}[Proof of Theorem \ref{thm:main}]
The existence proof follows from the uniform bound obtained in
Theorem~\ref{thm:apriori} and the compactness argument of \cite{GH18}. Uniqueness follows from
 Theorem~\ref{thm:direct-uniqueness}.
\end{proof}

\appendix

\section{Paracontrolled calculus and auxiliary lemmas}\label{sec:tools}

\subsection{Characterization of weighted spaces} We fix the annular partition used throughout the paper as follows.  Let
$\chi_0\in C_c^\infty(\mathbb R^2)$ be supported in $\{|x|\leqslant2\}$ and let
$\chi\in C_c^\infty(\mathbb R^2)$ be supported in
$\{2^{-1}\leqslant |x|\leqslant2\}$, with
\[
  \chi_0(x)+\sum_{k\geqslant1}\chi(2^{-k}x)=1 .
\]
Set $\chi_k(x):=\chi(2^{-k}x)$ for $k\geqslant1$.  Choose enlarged cutoffs
$\widetilde\chi_k$ with $\widetilde\chi_k\chi_k=\chi_k$, supported in a fixed
finite enlargement of the annulus of $\chi_k$.  Then
\[
  \sum_{k\geqslant0}\chi_k=1,
  \qquad
  \widetilde\chi_k\chi_k=\chi_k,
\]
the families have finite overlap, and for every multiindex $m$,
\[
  \sup_{k\geqslant0}2^{k|m|}
  \bigl(\|\partial^m\chi_k\|_{L^\infty}
       +\|\partial^m\widetilde\chi_k\|_{L^\infty}\bigr)<\infty .
\]
Moreover, on the support of either cutoff one has $\langle x\rangle\simeq2^k$.

\begin{lemma}\label{lem:weighted-lp-annular}
Let \(q,\beta\in\R\), let \(\rho_q(x)=\langle x\rangle^{-q}\).  Then the dyadic weighted norm used in
\eqref{eq:weighted-holder} is equivalent to multiplication by the weight:
\begin{equation}\label{eq:weight-multiplier-equivalence}
  \norm{f}_{C^\beta(\rho_q)}
  \simeq
  \norm{\rho_q f}_{C^\beta}.
\end{equation}
It is also equivalent to the following norm
\begin{equation}\label{eq:weighted-annular-equivalence}
  \norm{f}_{C^\beta(\rho_q)}
  \simeq
  \sup_{\ell\geqslant0}2^{-q\ell}\norm{\chi_\ell f}_{C^\beta}.
\end{equation}
Moreover, if distributions \(f_k\) satisfy
\begin{equation*}
  \norm{f_k}_{C^\beta}\leqslant A2^{qk},
\end{equation*}
and
\begin{equation*}
  F=\sum_{k\geqslant0}\widetilde\chi_k f_k
\end{equation*}
locally in distributions, then
\begin{equation}\label{eq:annular-summation-conclusion}
  \norm{F}_{C^\beta(\rho_q)}\lesssim A .
\end{equation}
\end{lemma}

\begin{proof}
The equivalence \eqref{eq:weight-multiplier-equivalence} follows from \cite[Theorem 6.5]{Tri06}.

\smallskip
\noindent\emph{Step 1: Proof of \eqref{eq:weighted-annular-equivalence}.}
 Let
\((\varphi_m)_{m\in\mathbb Z^2}\subset C_c^\infty(\mathbb R^2)\) be a smooth
partition of unity by translates of one fixed cutoff, with finite overlap.  Then
\[
  \norm{g}_{C^\beta}
  \simeq
  \sup_{m\in\mathbb Z^2}\norm{\varphi_m g}_{C^\beta};
\]
see \cite[Theorem~2.4.7]{Tri92}.
We now pass from this unit-cube localization to the present annular partition.
Since the family \((\chi_\ell)_\ell\) is uniformly bounded as multipliers on
\(C^\beta\),
\[
  \sup_{\ell\geqslant0}\norm{\chi_\ell g}_{C^\beta}\lesssim \norm{g}_{C^\beta}.
\]
Conversely, for each \(m\in\mathbb Z^2\), the set
\[
  I(m):=\{\ell\geqslant0:\operatorname{supp}\varphi_m\cap
  \operatorname{supp}\chi_\ell\neq\varnothing\}
\]
has cardinality bounded uniformly in \(m\).  Therefore,
\[
  \norm{\varphi_m g}_{C^\beta}
  \leqslant
  \sum_{\ell\in I(m)}\norm{\varphi_m\chi_\ell g}_{C^\beta}
  \lesssim
  \sup_{\ell\geqslant0}\norm{\chi_\ell g}_{C^\beta}.
\]
Taking the supremum in \(m\)  yields
\[
  \norm{g}_{C^\beta}
  \simeq
  \sup_{\ell\geqslant0}\norm{\chi_\ell g}_{C^\beta}.
\]
We apply this annular localization estimate to \(g=\rho_q f\).

Let \(M>|\beta|+2\).  On the support of \(\widetilde\chi_\ell\) one has
\(\rho_q(x)=\la x\ra^{-q}\simeq 2^{-q\ell}\).  Define
\[
  a_\ell:=2^{q\ell}\widetilde\chi_\ell\rho_q,
  \qquad
  b_\ell:=2^{-q\ell}\widetilde\chi_\ell\rho_q^{-1}.
\]
The families \((a_\ell)_\ell\) and \((b_\ell)_\ell\) are uniformly bounded in
\(C^M\).  Indeed, writing \(\widetilde\chi_\ell(x)=\widetilde\chi(2^{-\ell}x)\)
on the annular part, with the usual harmless modification for \(\ell=0\),
Leibniz' rule and \(\la x\ra\simeq 2^\ell\) on the support give, for
\(|a|\leqslant M\),
\[
  \norm{\partial^a a_\ell}_{L^\infty}
  +\norm{\partial^a b_\ell}_{L^\infty}
  \lesssim 2^{-|a|\ell}\leqslant C .
\]
Hence multiplication by \(a_\ell\) and by \(b_\ell\) is uniformly bounded on
\(C^\beta\).  Since \(\widetilde\chi_\ell\chi_\ell=\chi_\ell\),
\[
  \chi_\ell\rho_q f
  =2^{-q\ell}a_\ell(\chi_\ell f),
  \qquad
  \chi_\ell f
  =2^{q\ell}b_\ell(\chi_\ell\rho_q f).
\]
Therefore
\[
  \norm{\chi_\ell\rho_q f}_{C^\beta}
  \lesssim
  2^{-q\ell}\norm{\chi_\ell f}_{C^\beta},
  \qquad
  2^{-q\ell}\norm{\chi_\ell f}_{C^\beta}
  \lesssim
  \norm{\chi_\ell\rho_q f}_{C^\beta},
\]
uniformly in \(\ell\).  Combining the preceding localization estimate with these two multiplier estimates and then using
\eqref{eq:weight-multiplier-equivalence} gives
\[
\begin{aligned}
  \norm{f}_{C^\beta(\rho_q)}
  &\simeq \norm{\rho_q f}_{C^\beta}
   \simeq \sup_{\ell\geqslant0}\norm{\chi_\ell\rho_q f}_{C^\beta} \\
  &\simeq \sup_{\ell\geqslant0}2^{-q\ell}\norm{\chi_\ell f}_{C^\beta}.
\end{aligned}
\]
This proves \eqref{eq:weighted-annular-equivalence}.

\smallskip
\noindent\emph{Step 2: Proof of \eqref{eq:annular-summation-conclusion}.}
  By
\eqref{eq:weighted-annular-equivalence}, it suffices to bound
\[
  \sup_\ell2^{-q\ell}\norm{\chi_\ell F}_{C^\beta}.
\]
Since the families \((\chi_\ell)_\ell\) and \((\widetilde\chi_k)_k\) have finite overlap, there exists
\(C_{\rm ov}\geqslant 1\), depending only on the fixed annular partition, such that
\[
    \chi_\ell \widetilde\chi_k \equiv 0
    \qquad\text{whenever } |k-\ell|>C_{\rm ov}.
\]
Hence,
\[
\begin{aligned}
    2^{-q\ell}\|\chi_\ell F\|_{\mathcal C^\beta}
    &\leqslant
    C 2^{-q\ell}
    \sum_{\substack{k\geqslant 0\\ |k-\ell|\leqslant C_{\rm ov}}}
    \|f_k\|_{\mathcal C^\beta}  \\
    &\leqslant
    CA
    \sum_{\substack{k\geqslant 0\\ |k-\ell|\leqslant C_{\rm ov}}}
    2^{q(k-\ell)}
    \leqslant C A .
\end{aligned}
\]
  This proves
\eqref{eq:annular-summation-conclusion}.
\end{proof}

\subsection{Paracontrolled calculus} Let $(\Delta_j)_{j\geqslant -1}$ be the inhomogeneous Littlewood--Paley blocks and
let $S_{j-1}:=\sum_{i<j-1}\Delta_i$.  We use Bony's decomposition
\[
  fg=f\para g+f\res g+f\Par g,
\]
where
\[
  f\para g:=\sum_{j\geqslant -1}S_{j-1}f\,\Delta_jg,
  \qquad
  f\Par g:=g\para f,
  \qquad
  f\res g:=\sum_{|i-j|\leqslant1}\Delta_if\,\Delta_jg .
\]
The estimates below are stated simultaneously for the polynomial weights
$\rho_D(x)=\langle x\rangle^{-D}$ and for the exponential-polynomial weights used
in Section~\ref{sec:difference-details}.  In the latter case
\[
  \mathfrak e_\tau(x):=\exp\{-(1+\tau)\langle x\rangle^\varrho\},
  \qquad 0<\varrho<1,
\]
with \(\tau\) restricted to a fixed bounded interval, as in Section~\ref{sec:difference-details}.  For the exponential weights we used the dyadic cutoffs in the compactly supported Gevrey class of order $1/\zeta$ with $0<\zeta<1$ (see \cite[ Chapter 1]{Rod93} and \cite{MW17}).

\begin{lemma}\label{lem:para}
Let $\rho_A,\rho_B$ be two weights. It holds that
for $\beta\in\mathbb R$,
\begin{equation}\label{eq:para-Linfty-left}
  \norm{f\para g}_{C^\beta(\rho_A\rho_B)}
  \lesssim \norm{f}_{L^\infty(\rho_A)}\norm{g}_{C^\beta(\rho_B)}.
\end{equation}
If $\alpha<0$, then
\begin{equation}\label{eq:para-neg}
  \norm{f\para g}_{C^{\alpha+\beta}(\rho_{A}\rho_B)}
  \lesssim \norm{f}_{C^\alpha(\rho_A)}\norm{g}_{C^\beta(\rho_B)},
\end{equation}
If
$\alpha+\beta>0$, then
\begin{equation}\label{eq:res-est}
  \norm{f\res g}_{C^{\alpha+\beta}(\rho_{A}\rho_{B})}
  \lesssim \norm{f}_{C^\alpha(\rho_A)}\norm{g}_{C^\beta(\rho_B)}.
\end{equation}
\end{lemma}

\begin{proof}
See \cite[Lemmas~2.10--2.11]{ZZZ22} and
\cite[Lemmas~2.14--2.16]{GH18}.
\end{proof}

\begin{lemma}\label{lem:weighted-bony-commutator-catalogue}
Let
\[
  \mathrm{Com}(f,g,h):=(f\para g)\res h-f\,(g\res h).
\]
Assume that $0<\alpha<1$, $\beta+\gamma<0$, and
$\alpha+\beta+\gamma>0$.  Then, for  weights $\rho_A,\rho_B,\rho_C$
\begin{equation}\label{eq:weighted-bony-commutator-catalogue}
  \norm{\mathrm{Com}(f,g,h)}_{C^{\alpha+\beta+\gamma}(\rho_{A}\rho_B\rho_{C})}
  \lesssim
  \norm{f}_{C^\alpha(\rho_A)}
  \norm{g}_{C^\beta(\rho_B)}
  \norm{h}_{C^\gamma(\rho_C)} .
\end{equation}
\end{lemma}

\begin{proof}
See \cite[Lemma~2.11]{ZZZ22} and
\cite[Lemma~2.16]{GH18}.
\end{proof}

\begin{lemma}\label{lem:parabolic-I-comm}\label{lem:weighted-time-freezing-commutator}
Let $0<\alpha<1$, $\beta\in\mathbb R$, $0<\nu\leqslant\alpha/2$, and
$0\leqslant\omega<1$.  Let $w_1,w_2$ be either polynomial weights or the
exponential-polynomial product weights.  For
\[
  [\sI,f\para]g(t):=\sI(f\para g)
  -f\para\sI(g),
\]
assume
\[
  0\leqslant\theta\leqslant2+2\nu-\alpha .
\]
%
it holds that for $\Omega\geqslant0$
\begin{equation}\label{eq:app-lemma313-style-comm-Omega}
  \sup_{0<t\leqslant T}t^\Omega
  \|[\sI,f\para]g(t)\|_{C^{\alpha+\beta+\theta}(w_1w_2)}
  \leqslant C_T
  \|f\|_{\mathbb S_T^{\alpha,\omega;\nu}(w_1)}
  \|g\|_{L_T^\infty C^\beta(w_2)},
\end{equation}
provided
\begin{equation}\label{eq:app-lemma313-style-margins}
  \Omega+\frac{2-\theta}{2}-\omega\geqslant0,
  \qquad
  \Omega+\frac{2+2\nu-\alpha-\theta}{2}-\omega\geqslant0 .
\end{equation}
In particular, if $w_1=\rho_A$, $w_2=\rho_B$, $\beta<0$,
$\zeta:=2+\alpha+\beta\in(0,2)$, and one takes $\nu=\alpha/2$,
$\theta=2$, $\Omega=\omega$, then
\begin{equation*}
  \norm{[\sI,f\para]g}_{C_T^{\zeta,\omega}(\rho_{A+B})}
  \lesssim
  \norm{f}_{\mathbb S_T^{\alpha,\omega;\alpha/2}(\rho_A)}
  \norm{g}_{L_T^\infty C^\beta(\rho_B)} .
\end{equation*}
\end{lemma}

\begin{proof}
The proof is the heat analogue of the commutator argument in
\cite[Lemma~3.13]{HZZZ24}.  We give the details, especially for the
exponential weights.  We have
\[
\begin{aligned}
[\sI,f\para]g(t)
={}&\sum_j\int_0^t\Bigl(P_{t-a}(S_{j-1}f(a)\Delta_jg(a))
       -S_{j-1}f(a)P_{t-a}\Delta_jg(a)\Bigr)\,\mathrm{d}a \\
&+\sum_j\int_0^t\Bigl((S_{j-1}f(a)-S_{j-1}f(t))P_{t-a}\Delta_jg(a)\Bigr)\,\mathrm{d}a \\
=:& \sum_j(I_{1,j}(t)+I_{2,j}(t)).
\end{aligned}
\]

\smallskip
\noindent\emph{The spatial commutator block.}
We first prove the dyadic estimate, valid for a large fixed $N$,
\[
\begin{aligned}
&\Bigl\|P_{t-a}(S_{j-1}f(a)\Delta_jg(a))
       -S_{j-1}f(a)P_{t-a}\Delta_jg(a)\Bigr\|_{L^\infty(w_1w_2)} \\
&\qquad\lesssim
  2^{-\alpha j}\bigl(1\wedge((t-a)2^{2j})^{-N}\bigr)
  \|f(a)\|_{C^\alpha(w_1)}\,2^{-\beta j}
  \|g\|_{L_T^\infty C^\beta(w_2)} .
\end{aligned}
\]
The unweighted cancellation is the standard one; we spell out the weighted kernel estimate.  Put $W=w_1w_2$.  We write
\[
  K_{j,s}(z)
  =\mathcal F^{-1}\!\left(\varphi_j(\xi)\widetilde\varphi_j(\xi)e^{-s|\xi|^2}\right)(z),
  \qquad s=t-a>0,
\]
where $\varphi_j$ is the multiplier of $\Delta_j$ and
$\widetilde\varphi_j$ is equal to one on the support of $\varphi_j$.
In the spatial commutator block the cancellation
\[
  S_{j-1}f(x-z)-S_{j-1}f(x)
\]
contributes a factor $|z|^\alpha\|f(a)\|_{C^\alpha(w_1)}$ before the kernel is
integrated.

For a fixed admissible weight $W$, insert the weight ratio:
\[
  W(x)|K_{j,s}*h(x)|
  \leqslant
  \int |K_{j,s}(z)|\frac{W(x)}{W(x-z)}\,\mathrm{d}z\,
       \|h\|_{L^\infty(W)} .
\]
For the exponential-polynomial weights we have
\[
  \frac{W(x)}{W(x-z)}
  \lesssim
  \langle z\rangle^{m_W}\exp\{c_W|z|^\varrho\},
\]
where  $m_W,c_W$ depend only on the  weights; for polynomial
weights one simply has $c_W=0$.  Hence the weighted block bound reduces to
moments
\[
  \mathcal M_{j,s}^{(\delta)}(W)
  :=
  \sup_x\int |K_{j,s}(z)|\,|z|^\delta
       \frac{W(x)}{W(x-z)}\,\mathrm{d}z,
  \qquad \delta\geqslant0 .
\]

Choose $\zeta_{\rm G}\in(\varrho,1)$ and take the dyadic cutoffs in a compactly
supported Gevrey class of order $1/\zeta_{\rm G}$.    By the standard
Gevrey--Paley--Wiener estimate, the inverse Fourier transform of a compactly
supported Gevrey-$1/\zeta_{\rm G}$ multiplier decays like
$\exp\{-c|x|^{\zeta_{\rm G}}\}$; see, for instance,
\cite[Chapter~1]{Rod93}.  Let $j\geqslant0$, set $\eta=s2^{2j}$, and write
$y=2^jz$.  Then
\[
  K_{j,s}(z)=2^{2j}k_\eta(2^jz),
  \qquad
  k_\eta=\mathcal F^{-1}\!\left(\varphi(\xi)\widetilde\varphi(\xi)e^{-\eta|\xi|^2}\right).
\]
The multiplier $\varphi\widetilde\varphi$ is supported in a fixed annulus.  For
every $L\geqslant0$, the family
\[
  (1+\eta)^L\varphi(\xi)\widetilde\varphi(\xi)e^{-\eta|\xi|^2},
  \qquad \eta\geqslant0,
\]
is bounded in the same compactly supported Gevrey class.  Consequently
\[
  |k_\eta(y)|\leqslant C_L(1+\eta)^{-L}e^{-c|y|^{\zeta_{\rm G}}},
\]
and hence
\[
  |K_{j,s}(z)|
  \leqslant C_L2^{2j}(1+s2^{2j})^{-L}
      \exp\{-c(2^j|z|)^{\zeta_{\rm G}}\} .
\]
  Since $\zeta_{\rm G}>\varrho$, after the change of variables
$y=2^jz$, for $0\leqslant\delta\leqslant\alpha$,
\[
\begin{aligned}
  \mathcal M_{j,s}^{(\delta)}(W)
  &\lesssim
  2^{-\delta j}(1+s2^{2j})^{-L}
  \int_{\mathbb R^2}|y|^\delta
       \langle 2^{-j}y\rangle^{m_W}
       \exp\{-c|y|^{\zeta_{\rm G}}+c_W2^{-j\varrho}|y|^\varrho\}\,\mathrm{d}y  \\
  &\lesssim
  2^{-\delta j}(1+s2^{2j})^{-L} .
\end{aligned}
\]

Taking $L$ large enough gives, in particular,
\[
  \mathcal M_{j,s}^{(\alpha)}(W)
  \lesssim
  2^{-\alpha j}\bigl(1\wedge(s2^{2j})^{-N}\bigr).
\]
Together with
$\|\Delta_jg(a)\|_{L^\infty(w_2)}\lesssim2^{-\beta j}
\|g\|_{L_T^\infty C^\beta(w_2)}$, this proves the dyadic estimate above.  Integrating in $a$, we get
\[
  \|I_{1,j}(t)\|_{L^\infty(w_1w_2)}
  \lesssim
  2^{-(\alpha+\beta)j}
  \int_0^t\bigl(1\wedge((t-a)2^{2j})^{-N}\bigr)a^{-\omega}\,\mathrm{d}a\,
  \|f\|_{C_T^{\alpha,\omega}(w_1)}
  \|g\|_{L_T^\infty C^\beta(w_2)} .
\]
For every $0\leqslant\theta\leqslant2$,
\begin{equation*}
  t^\Omega
  \int_0^t\bigl(1\wedge((t-a)2^{2j})^{-N}\bigr)a^{-\omega}\,\mathrm{d}a
  \lesssim_T
  2^{-\theta j}
\end{equation*}
whenever $\Omega+(2-\theta)/2-\omega\geqslant0$.  Indeed, the elementary bound
\[
  \int_0^t\bigl(1\wedge((t-a)2^{2j})^{-N}\bigr)a^{-\omega}\,\mathrm{d}a
  \lesssim_T 2^{-\theta j}t^{1-\omega-\theta/2}
\]
follows by splitting at $a=t/2$: on $[0,t/2]$ one has $t-a\simeq t$, while on
$[t/2,t]$ one has $a\simeq t$ and uses
\[
  \int_0^{t/2}\bigl(1\wedge(b2^{2j})^{-N}\bigr)\,\mathrm{d}b
  \lesssim 2^{-\theta j}t^{1-\theta/2}.
\]
After multiplication by $t^\Omega$, this is bounded on $(0,T]$.  Hence the spatial commutator block contributes to
\eqref{eq:app-lemma313-style-comm-Omega}.

\smallskip
\noindent\emph{The time-freezing block.}
For $I_{2,j}$, the same argument
gives
\[
  \|I_{2,j}(t)\|_{L^\infty(w_1w_2)}
  \lesssim
  2^{-\beta j}\|g\|_{L_T^\infty C^\beta(w_2)}
  \int_0^t\bigl(1\wedge((t-a)2^{2j})^{-N}\bigr)
  \|f(a)-f(t)\|_{L^\infty(w_1)}\,\mathrm{d}a .
\]
Using the time-increment component of
$\mathbb S_T^{\alpha,\omega;\nu}(w_1)$,
\[
  \|f(a)-f(t)\|_{L^\infty(w_1)}
  \leqslant a^{-\omega}|t-a|^\nu
  \|f\|_{\mathbb S_T^{\alpha,\omega;\nu}(w_1)} .
\]
The required elementary heat-convolution bound is
\begin{equation*}
  t^\Omega
  \int_0^t\bigl(1\wedge((t-a)2^{2j})^{-N}\bigr)
     |t-a|^\nu a^{-\omega}\,\mathrm{d}a
  \lesssim_T 2^{-(\alpha+\theta)j},
\end{equation*}
provided
\[
  \Omega+\frac{2+2\nu-\alpha-\theta}{2}-\omega\geqslant0,
  \qquad 0\leqslant\theta\leqslant2+2\nu-\alpha .
\]
To see this, split again into $a\leqslant t/2$ and $a>t/2$.  On the first part one
uses $t-a\simeq t$.  On the second part one has $a\simeq t$ and
uses $\int_0^{t/2} (1\wedge (b2^{2j})^{-N}) b^\nu\,\mathrm{d}b\lesssim 2^{-(\alpha+\theta)j}t^{1+\nu-\frac{\alpha+\theta}2}$.
Consequently
\[
  t^\Omega 2^{(\alpha+\beta+\theta)j}
  \|I_{2,j}(t)\|_{L^\infty(w_1w_2)}
  \lesssim
  \|f\|_{\mathbb S_T^{\alpha,\omega;\nu}(w_1)}
  \|g\|_{L_T^\infty C^\beta(w_2)} .
\]
Combining the estimates for $I_{1,j}$ and $I_{2,j}$ and taking the supremum in
$j$ proves \eqref{eq:app-lemma313-style-comm-Omega}.
\end{proof}

We finally record the weighted paralinearization estimate used in the exponential uniqueness argument.
The proof below follows the argument of Section \ref{sec:transport-representation}, with the exponential-polynomial weight inserted at the relevant places.

\begin{lemma}[Weighted paralinearization]\label{lem:weighted-paralinearization}
	Let \(0<\alpha<1\), let \(G\in C_b^2(\mathbb R)\), and assume that \(G''\) is globally Lipschitz.  Put
	\[
		\mathcal R_G(f):=G(f)-G'(f)\para f .
	\]
	Fix a finite polynomial input exponent \(M_{a}\geqslant0\), and set
	\[
		A_{M_{a}}(f,g):=1+\norm{f}_{C^\alpha(\rho_{M_{a}})}+\norm{g}_{C^\alpha(\rho_{M_{a}})} .
	\]
	Then there exists another finite polynomial exponent \(M_{*}=2M_a\),  such that
	\begin{equation*}
	\begin{aligned}
		&\norm{\mathcal R_G(f)-\mathcal R_G(g)}_{C^{2\alpha}(\mathfrak e_\tau\rho_{M_{*}})}  \\
		&\qquad\leqslant C_G\Bigl(
		\norm{f-g}_{C^\alpha(\mathfrak e_\tau)}A_{M_{a}}(f,g)
		+\norm{f-g}_{L^\infty(\mathfrak e_\tau)}A_{M_{a}}(f,g)^2
		\Bigr).
	\end{aligned}
	\end{equation*}
	
\end{lemma}

\begin{proof}
Write \(h=f-g\) and \(A=A_{M_{a}}(f,g)\).  As in Section \ref{sec:transport-representation}, decompose
\[
  \mathcal R_G(f)=\sum_{i\geqslant-1}u_i^f,
  \qquad
  u_i^f:=\Delta_iG(f)-S_{i-1}G'(f)\Delta_i f .
\]
For \(i\geqslant1\), let \(K_i\) be the kernel of \(\Delta_i\), and let \(K_{<i-1}\) be the kernel of \(S_{i-1}\).  Since \(\int K_i=0\) and \(\int K_{<i-1}=1\), the same computation as in Section \ref{sec:transport-representation} gives
\[
  u_i^f(x)=\iint K_i(x-y)K_{<i-1}(x-z)
  T_G\bigl(f(y),f(z)\bigr)\,\mathrm{d}y\,\mathrm{d}z,
\]
where
\[
  T_G(a,b):=G(a)-G(b)-G'(b)(a-b).
\]
The second-order Taylor formula and the Lipschitz property of \(G''\) imply the pointwise difference bound
\begin{equation*}
\begin{aligned}
&\bigl|T_G(f(y),f(z))-T_G(g(y),g(z))\bigr| \\
&\quad\lesssim_G
  \bigl(|\delta_{yz}f|+|\delta_{yz}g|\bigr)|\delta_{yz}h|
  +\bigl(|h(y)|+|h(z)|\bigr)
    \bigl(|\delta_{yz}f|^2+|\delta_{yz}g|^2\bigr),
\end{aligned}
\end{equation*}
with \(\delta_{yz}q=q(y)-q(z)\).  This is the only place where the Lipschitz bound on \(G''\) is used.
Indeed, by the integral form of the second-order Taylor remainder, for any \(a,b\in\mathbb R\),
\[
T_G(a,b)
=
G(a)-G(b)-G'(b)(a-b)
=
(a-b)^2\int_0^1 (1-\theta)
G''\bigl(b+\theta(a-b)\bigr)\,\mathrm{d}\theta .
\]
We apply this with
\[
a=f(y),\qquad b=f(z),\qquad c=g(y),\qquad d=g(z).
\]
Then
\[
\begin{aligned}
&T_G(f(y),f(z))-T_G(g(y),g(z))  \\
&\quad =
\int_0^1 (1-\theta)
\Big[
G''\bigl(b+\theta\delta_{yz}f\bigr)(\delta_{yz}f)^2
-
G''\bigl(d+\theta\delta_{yz}g\bigr)(\delta_{yz}g)^2
\Big]\,\mathrm{d}\theta .
\end{aligned}
\]
We split the integrand as
\[
\begin{aligned}
& G''\bigl(b+\theta\delta_{yz}f\bigr)
\bigl((\delta_{yz}f)^2-(\delta_{yz}g)^2\bigr) \\
&\quad +
\Big[
G''\bigl(b+\theta\delta_{yz}f\bigr)
-
G''\bigl(d+\theta\delta_{yz}g\bigr)
\Big](\delta_{yz}g)^2 .
\end{aligned}
\]
We have
\[
\bigl|(\delta_{yz}f)^2-(\delta_{yz}g)^2\bigr|
\leqslant
\bigl(|\delta_{yz}f|+|\delta_{yz}g|\bigr)
|\delta_{yz}h|.
\]
By the Lipschitz continuity of \(G''\),
\[
\begin{aligned}
&\bigl|
G''\bigl(b+\theta\delta_{yz}f\bigr)
-
G''\bigl(d+\theta\delta_{yz}g\bigr)
\bigr|   \lesssim |h(z)|+|h(y)| .
\end{aligned}
\]
Combining the previous estimates yields the bound.

We now insert the weights.  The compactly supported Gevrey cutoffs used for the exponential topology give the kernel moment estimate, uniformly for \(\tau\) in the fixed bounded interval,
\begin{equation*}
  \int |K_i(x-y)|\,\langle 2^i(x-y)\rangle^{N}
       \frac{\mathfrak e_\tau(x)}{\mathfrak e_\tau(y)}\,\mathrm{d}y
  +
  \int |K_{<i-1}(x-z)|\,\langle 2^i(x-z)\rangle^{N}
       \frac{\mathfrak e_\tau(x)}{\mathfrak e_\tau(z)}\,\mathrm{d}z
  \lesssim 1,
\end{equation*}
for every fixed \(N\); the proof is the same subexponential moment estimate used in Lemma~\ref{lem:weighted-time-freezing-commutator}. The remaining argument is the same as in Section \ref{sec:transport-representation} and \cite[Lemma 5.2]{GIP15}. Choosing \(M_*= 2M_a\), the polynomial factors coming from the two
increments of \(f\) and \(g\) are absorbed by the  weight
\(\rho_{M_*}\).
\end{proof}

\subsection{Some basic estimates}

The following result is proved in \cite[Lemma~2.7]{ZZZ22}.

\begin{lemma}\label{lem:interp}
Let $0\leqslant \theta\leqslant1$, $\beta=(1-\theta)\beta_0+\theta\beta_1$, and $\rho_D=\rho^{(1-\theta)}_{D_0}\rho^{\theta}_{ D_1}$.  Then
\begin{equation*}
  \norm{f}_{C^\beta(\rho_D)}
  \lesssim
  \norm{f}_{C^{\beta_0}(\rho_{D_0})}^{1-\theta}
  \norm{f}_{C^{\beta_1}(\rho_{D_1})}^{\theta}.
\end{equation*}
Consequently, if $f\in L_T^\infty C^{\beta_0}(\rho_{D_0})\cap C_T^{\beta_1,\gamma}(\rho_{D_1})$, then
\begin{equation*}
  \norm{f}_{C_T^{\beta,\theta\gamma}(\rho_D)}
  \lesssim
  \norm{f}_{L_T^\infty C^{\beta_0}(\rho_{D_0})}^{1-\theta}
  \norm{f}_{C_T^{\beta_1,\gamma}(\rho_{D_1})}^{\theta}.
\end{equation*}
If $\beta_0=0$, $\beta=\theta\beta_1$, then it holds that
\begin{equation*}
  \norm{f}_{\mathbb S_T^{\beta,\theta\gamma;\theta\nu}(\rho_D)}
  \lesssim
  \norm{f}_{L_T^\infty L^\infty(\rho_{D_0})}^{1-\theta}
  \norm{f}_{\mathbb S_T^{\beta_1,\gamma;\nu}(\rho_{D_1})}^{\theta}.
\end{equation*}
\end{lemma}

\begin{lemma}\label{lem:mild-time-increments}
Let $0<\sigma<2$ and $0\leqslant\omega<1$.
\begin{enumerate}[label=\textup{(\roman*)}]
\item If $h\in L_T^{\infty,\omega}C^{\sigma-2}(\rho_A)$ and $z=\sI h$, then
\begin{equation*}
  \norm{z}_{C_T^{\sigma,\omega}(\rho_A)}
  +\norm{z}_{C_T^{\sigma/2,\omega}L^\infty(\rho_A)}
  \lesssim
  \norm{h}_{L_T^{\infty,\omega}C^{\sigma-2}(\rho_A)}.
\end{equation*}
\item If $\xi\in C^\beta(\rho_A)$, $0<\beta\leqslant\sigma<2$, and $z(t)=P_t\xi$, then
\begin{equation*}
  \norm{z}_{C_T^{\sigma,(\sigma-\beta)/2}(\rho_A)}
  +\norm{z}_{C_T^{\sigma/2,(\sigma-\beta)/2}L^\infty(\rho_A)}
  \lesssim \norm{\xi}_{C^\beta(\rho_A)}.
\end{equation*}
\end{enumerate}
\end{lemma}

\begin{proof}
Parts \textup{(i)}--\textup{(ii)} are the weighted heat-semigroup and Schauder estimates of \cite[Lemmas~2.8--2.9]{ZZZ22}, with the same elementary convolution estimate and the harmless extra factor \(t^{-\omega}\).
\end{proof}

\label{sec:appendix-unique-weighted}
We also use the following elementary Volterra lemma.  For later use we introduce the following notation:
\begin{equation}\label{eq:volterra-envelope-def}
	\bigl(\mathcal V_{\mu,\lambda}X\bigr)(t)
	:=\sup_{0<r\leqslant t}\int_0^r(r-a)^{-\mu}a^{-\lambda}X(a)\,\mathrm{d}a .
\end{equation}

\begin{lemma}\label{lem:duhamel-volterra-reduction}
Let \(T>0\), \(0<\varrho<1\), and let \(X:(0,T]\to[0,\infty)\) be bounded and measurable.  Let \(h\) be a time-dependent distribution and set
\[
	z(t):=(\sI h)(t)=\int_0^tP_{t-a}h(a)\,\mathrm{d}a .
\]
Assume that for some \(A_h\geqslant0\), \(M\geqslant0\), \(0\leqslant\lambda<1\), and \(\theta\in\mathbb R\),
\begin{equation}\label{eq:app-duhamel-volterra-assumption}
	\|h(a)\|_{C^\theta(\mathfrak e_t)}
	\leqslant A_h(t-a)^{-M/\varrho}a^{-\lambda}X(a),
	\qquad 0<a<t\leqslant T .
\end{equation}
Fix  \(\beta\in\mathbb R\) and  \(\Omega\geqslant0\).  Put
\begin{equation*}
	\mu:=\frac{(\beta-\theta)_+}{2}+\frac{M}{\varrho},
	\qquad
	\mu_0:=\frac{(-\theta)_+}{2}+\frac{M}{\varrho}.
\end{equation*}

\begin{enumerate}[label=\textup{(\roman*)}]
\item If
\begin{equation*}
	\mu<1,
	\qquad
	\mu+\lambda-\Omega<1,
\end{equation*}
then
\begin{equation}\label{eq:app-duhamel-volterra-spatial-out}
\begin{aligned}
	\sup_{0<t\leqslant T}t^\Omega\|z(t)\|_{C^\beta(\mathfrak e_t)}
	&\leqslant C A_h
	\sum_{(\bar\mu,\bar\lambda)\in\mathfrak K_{\rm sp}}
	\bigl(\mathcal V_{\bar\mu,\bar\lambda}X\bigr)(T),
\end{aligned}
\end{equation}
where
\begin{equation*}
	\mathfrak K_{\rm sp}:=
	\Bigl\{\bigl(\mu,(\lambda-\Omega)_+\bigr),
	\bigl((\mu-\Omega)_+,\lambda\bigr)\Bigr\}.
\end{equation*}
Every pair in \(\mathfrak K_{\rm sp}\) satisfies \(0\leqslant\bar\mu,\bar\lambda<1\) and \(\bar\mu+\bar\lambda<1\).

\item Let \(0<\nu\leqslant\beta/2\).  If, in addition,
\begin{equation*}
	\mu_0+\nu<1,
	\qquad
	\mu_0+\nu+\lambda-\Omega<1,
\end{equation*}
then
\begin{equation}\label{eq:app-duhamel-volterra-holder-out}
\begin{aligned}
	\sup_{0<s<t\leqslant T}s^\Omega
	\frac{\|z(t)-z(s)\|_{L^\infty(\mathfrak e_t)}}{|t-s|^\nu}
	&\leqslant C A_h
	\sum_{(\bar\mu,\bar\lambda)\in\mathfrak K_{\rm tm}}
	\bigl(\mathcal V_{\bar\mu,\bar\lambda}X\bigr)(T),
\end{aligned}
\end{equation}
where
\begin{equation*}
	\mathfrak K_{\rm tm}:=
	\Bigl\{\bigl(\mu_0+\nu,(\lambda-\Omega)_+\bigr),
	\bigl((\mu_0+\nu-\Omega)_+,\lambda\bigr)\Bigr\}.
\end{equation*}
Every pair in \(\mathfrak K_{\rm tm}\) satisfies \(0\leqslant\bar\mu,\bar\lambda<1\) and \(\bar\mu+\bar\lambda<1\).
\end{enumerate}
\end{lemma}

\begin{proof}
  By \eqref{eq:app-duhamel-volterra-assumption} and \cite[Lemma 2.8]{ZZZ22},
\[
\begin{aligned}
	t^\Omega\|z(t)\|_{C^\beta(\mathfrak e_t)}
	&\lesssim A_h t^\Omega\int_0^t
	(t-a)^{-\frac{(\beta-\theta)_+}{2}}
	(t-a)^{-M/\varrho}a^{-\lambda}X(a)\,\mathrm{d}a  \\
	&=A_h t^\Omega\int_0^t(t-a)^{-\mu}a^{-\lambda}X(a)\,\mathrm{d}a .
\end{aligned}
\]
Since \(t^\Omega\lesssim a^\Omega+(t-a)^\Omega\), and negative powers are discarded by the convention \(r_+=\max\{r,0\}\), this gives \eqref{eq:app-duhamel-volterra-spatial-out}.
For time increments, write for \(0<s<t\leqslant T\)
\[
	z(t)-z(s)=\int_0^s(P_{t-a}-P_{s-a})h(a)\,\mathrm{d}a
	+
	\int_s^tP_{t-a}h(a)\,\mathrm{d}a
	=:J_1+J_2 .
\]
For the terminal part \(J_2\), \cite[Lemma 2.8]{ZZZ22} and \(|t-s|^{-\nu}\leqslant(t-a)^{-\nu}\) on \([s,t]\) yield
\[
\begin{aligned}
	s^\Omega|t-s|^{-\nu}\|J_2\|_{L^\infty(\mathfrak e_t)}
	&\lesssim A_h s^\Omega\int_s^t(t-a)^{-\mu_0-\nu}a^{-\lambda}X(a)\,\mathrm{d}a \\
	&\lesssim A_h\int_0^t(t-a)^{-\mu_0-\nu}a^{-(\lambda-\Omega)_+}X(a)\,\mathrm{d}a,
\end{aligned}
\]
because \(s\leqslant a\) on the terminal interval.
For \(J_1\), \cite[Lemma 2.8]{ZZZ22} gives
\[
	\|(P_{t-a}-P_{s-a})h(a)\|_{L^\infty(\mathfrak e_t)}
	\lesssim A_h |t-s|^\nu(s-a)^{-\mu_0-\nu}a^{-\lambda}X(a),
	\qquad 0<a<s.
\]
Here we used \((t-a)^{-M/\varrho}\leqslant(s-a)^{-M/\varrho}\).  Therefore
\[
\begin{aligned}
	s^\Omega|t-s|^{-\nu}\|J_1\|_{L^\infty(\mathfrak e_t)}
	&\lesssim A_h s^\Omega\int_0^s(s-a)^{-\mu_0-\nu}a^{-\lambda}X(a)\,\mathrm{d}a .
\end{aligned}
\]
Using \(s^\Omega\lesssim a^\Omega+(s-a)^\Omega\) gives the two pairs in \(\mathfrak K_{\rm tm}\).  This proves \eqref{eq:app-duhamel-volterra-holder-out} and completes the proof.
\end{proof}

\begin{lemma}[Volterra--Gronwall]\label{lem:volterra-gronwall}
Let \(X:[0,T]\to[0,\infty)\) be bounded and measurable.  Assume that, for finitely many indices \(m\),
\[
	0\leqslant\mu_m,\lambda_m<1,
	\qquad \mu_m+\lambda_m<1,
\]
and that
\[
	X(t)\leqslant \sum_m C_m\bigl(\mathcal V_{\mu_m,\lambda_m}X\bigr)(t),
	\qquad 0<t\leqslant T .
\]
Then \(X\equiv0\) on \([0,T]\).
\end{lemma}

\begin{proof}
Let \(Y(t):=\sup_{0<s\leqslant t}X(s)\).  For \(0<r\leqslant t\), the change of variables \(a=rb\) gives
\[
	\int_0^r(r-a)^{-\mu_m}a^{-\lambda_m}X(a)\,\mathrm{d}a
	\leqslant \mathrm B(1-\mu_m,1-\lambda_m)r^{1-\mu_m-\lambda_m}Y(t),
\]
where \(\mathrm B(p,q):=\int_0^1(1-b)^{p-1}b^{q-1}\,\mathrm{d}b\) is the beta function.
Hence, for \(t>0\) small enough,
\[
	\sum_m C_m\mathrm B(1-\mu_m,1-\lambda_m)t^{1-\mu_m-\lambda_m}\leqslant\frac12,
\]
and therefore \(Y(t)\leqslant Y(t)/2\).  Thus \(X=0\) on $[0,t]$.
Repeating the same argument from the right endpoint of an interval on which \(X\) is already zero gives the conclusion on the whole of \([0,T]\).
\end{proof}

\end{document}